\documentclass[11pt,leqno]{amsart}

\usepackage{graphicx}
\usepackage{epstopdf,epsfig}
\usepackage[colorlinks=false]{hyperref}
\usepackage{csquotes}

\usepackage{todonotes} 
\usepackage{xcolor}

\usepackage{amsmath,amsthm,amsfonts,amssymb,latexsym,amscd,enumerate,mathtools}

\usepackage{palatino}
\usepackage[mathcal]{euler}

\usepackage{xy}
\xyoption{all}

\theoremstyle{remark}

\theoremstyle{plain}

\newcounter{theoremintro} 

\newtheorem*{definition*}{Definition} 
\newtheorem*{theorem*}{Theorem} 
\newtheorem*{lemma*}{Lemma}
\newtheorem*{corollary*}{Corollary} 
\newtheorem*{proposition*}{Proposition} 

\newtheorem{theorem}[subsection]{Theorem} 
\newtheorem{lemma}[subsection]{Lemma}
\newtheorem{corollary}[subsection]{Corollary}
\newtheorem{proposition}[subsection]{Proposition}

\newtheorem{definition}[subsection]{Definition}
\newtheorem{example}[subsection]{Example}
\newtheorem{remark}[subsection]{Remark}

\numberwithin{equation}{section}

\newcommand{\as}{\mathfrak{A}}

\newcommand{\R}{\mathbb{R}}
\newcommand{\bC}{\mathbb{C}} 
\newcommand{\bD}{\mathbb{D}} 
\newcommand{\N}{\mathbb{N}}

\newcommand{\cf}{\mathcal{F}}
\newcommand{\E}{\mathcal{E}}
\newcommand{\bL}{\mathbb{L}}

\newcommand{\Compacts}{\mathfrak{K}}
\newcommand{\Linears}{\mathfrak{L}}

\newcommand{\limt}{\lim_{t\to \infty}} 
\newcommand{\limn}{\lim_{n\to \infty}} 

\newcommand{\cA}{\mathcal{A}}

\newcommand{\cC}{\mathcal{C}}
\newcommand{\cD}{\mathcal{D}}
\newcommand{\cE}{\mathcal{E}}
\newcommand{\cF}{\mathcal{F}}

\newcommand{\cP}{\mathcal{P}}

\newcommand{\rmax}{\mathrm{max}}
\newcommand{\supp}{\mathrm{supp}}
\newcommand{\prop}{\mathrm{prop}}

\newcommand{\EG}{\underline{E}G}

\newcommand{\s}[1]{\langle #1 \rangle}

\newcommand{\hill}{\mathcal{H}}

\newcommand{\proper}{\mathrm{proper}}
\newcommand{\Gcont}{\mathrm{Gcont}}

\newtheorem*{theoremA*}{Theorem A}
\newtheorem*{theoremB*}{Theorem B}
\newtheorem*{theoremC*}{Theorem C}
\newtheorem*{theoremD*}{Theorem D}
\newtheorem*{theoremE*}{Theorem E}
\newtheorem*{theoremF*}{Theorem F}
\newtheorem*{theoremG*}{Theorem G}
\newtheorem*{theoremH*}{Theorem H}
\newtheorem*{theoremI*}{Theorem I}
\newtheorem*{theoremJ*}{Theorem J}

\newtheorem*{corollaryA*}{Corollary A}
\newtheorem*{corollaryB*}{Corollary B}
\newtheorem*{corollaryC*}{Corollary C}
\newtheorem*{corollaryD*}{Corollary D}
\newtheorem*{corollaryE*}{Corollary E}
\newtheorem*{corollaryF*}{Corollary F}
\newtheorem*{corollaryG*}{Corollary G}
\newtheorem*{corollaryH*}{Corollary H}

\begin{document} 
\title{crossed product approach to equivariant localization algebras} 

\author{Shintaro Nishikawa}

\address{S.N.: Mathematisches Institut, Fachbereich Mathematik und Informatik der Universit\"at M\"unster, Einsteinstrasse 62, 48149 M\"unster, Germany.} 
\email{snishika@uni-muenster.de}

\thanks{This research was supported by the Deutsche Forschungsgemeinschaft (DFG, German Research Foundation) under Germany's Excellence Strategy EXC 2044-390685587, Mathematics M\"unster: Dynamics-Geometry-Structure.}




\date{\today}

\maketitle

\begin{abstract} The goal of this article is to provide a bridge between the gamma element method for the Baum--Connes conjecture (the Dirac dual-Dirac method) and the controlled algebraic approach of Roe and Yu (localization algebras). For any second countable, locally compact group $G$, we study the reduced crossed product algebras of the representable localization algebras for proper $G$-spaces. We show that the naturally defined forget-control map is equivalent to the Baum--Connes assembly map for any locally compact group $G$ and for any coefficient $G$-$C^*$-algebra $B$. We describe the gamma element method for the Baum--Connes conjecture from this controlled algebraic perspective. As an application, we extend the recent new proof of the Baum--Connes conjecture with coefficients for CAT(0)-cubical groups to the non-cocompact setting. 
 \end{abstract}
 
\tableofcontents

\section{Introduction} 
Let $G$ be a second countable, locally compact group. The Baum--Connes conjecture with coefficients (BCC) states that the Baum--Connes assembly map (see \cite{BCH93}, \cite{Valette02}, \cite{GJV19}) 
\begin{equation*}
\mu_r^{B, G}\colon K_\ast^{\mathrm{top}}(G; B)  =\varinjlim_{Y\subset \EG, \mathrm{Gcpt}}KK^G(C_0(Y), B) \to K_\ast(B\rtimes_rG)
\end{equation*} 
is an isomorphism for any separable $G$-$C^*$-algebra $B$. Here, $B\rtimes_rG$ is the reduced crossed product of $B$ and $\EG$ is the universal proper $G$-space \cite{BCH93} which exists uniquely up to $G$-equivariant homotopy. The classical Baum--Connes conjecture (BC) \cite{BaumConnes2000} corresponds to the case when $B=\bC$.

For a countable discrete group $G$, there is a controlled algebraic reformulation of the Baum--Connes assembly map, due to Yu, in terms of the forget-control map \cite{Yu10}. For simplicity, we let $B=\bC$ and assume that a proper $G$-space $X$ is equipped with a proper $G$-invariant metric $d$ compatible with its topology. The equivariant Roe algebra $C^*(H_X)^G$ \cite{Roe96} is defined as the norm-completion of the $\ast$-algebra $\bC(H_X)^G$ consisting of $G$-equivariant, locally compact operators with finite propagation on an $X$-$G$-module $H_X$, a $G$-Hilbert space equipped with a non-degenerate representation of the $G$-$C^*$-algebra $C_0(X)$. Here, an operator $S$ on $H_X$ is locally compact if $\phi S, S\phi$ is compact for any $\phi$ in $C_0(X)$ and has finite propagation if there is $r\geq0$ such that $\phi S\psi=0$ whenever the distance between the support of $\phi, \psi$ in $C_0(X)$ is greater than $r$. The infimum of such $r$ for $S$ is called the propagation of $S$. An important fact is an isomorphism (see \cite{Roe96})
\[
C^*(H_X)^G \cong \Compacts(l^2(\N))\otimes C^*_r(G)
\]
when $X$ is $G$-compact and when $H_X$ is big enough (if $H_X$ is ample as an $X$-$G$-module). The equivariant localization algebra (or the localized equivariant Roe algebra) $C_L^*(H_X)^G$ is defined as the norm-completion of the $\ast$-algebra consisting of uniformly norm-continuous functions $t\mapsto S_t \in \bC(H_X)^G$ on $[1, \infty)$ such that the propagation $\prop(S_t)\to 0$ as $t\to \infty$. The evaluation at $t=1$ induces a $\ast$-homomorphism (the forget-control map)
\[
\mathrm{ev}_1\colon C_L^*(H_X)^G \to C^*(H_X)^G 
\]
whose induced map on K-theory groups 
\[
\mathrm{ev}_{1\ast} \colon K_\ast(C_L^*(H_X)^G) \to K_\ast(C^*(H_X)^G)
\]
is known to be equivalent to the Baum--Connes assembly map
\[
\mu^G_X\colon KK_\ast^G(C_0(X), \bC) \to K_\ast(C^*_r(G))
\]
if $X$ is $G$-compact (see \cite{Shan08},\cite{Yu10}, \cite{FuWang}, \cite{FuWangYu}). Taking the inductive limit of $\mathrm{ev}_{1\ast}$ over the $G$-compact $G$-invariant closed subspaces $Y$ of $\EG$, we get 
\[
\mathrm{ev}_{1\ast} \colon \lim_{Y\subset \EG, \mathrm{Gcpt}}K_\ast(C_L^*(H_Y)^G) \to K_\ast(C^*_r(G))
\]
which is equivalent to the Baum--Connes assembly map $\mu_r^G$. This is the controlled algebraic reformulation of the Baum--Connes conjecture due to Yu.

For a (second countable) locally compact space $X$ and an $X$-module $H_X$, a representable localization algebra $RL^*_c(H_X)$ (see below for definition) is introduced in \cite[Section 9.4]{WY2020} (see \cite{Yu1997} for a localization algebra). It is shown that for a suitable choice of an $X$-module $H_X$ for each locally compact space $X$ and for a suitable choice of covering isometries $(V^f_t\colon H_X\to H_Y)_{t\in [1,\infty)}$ for each continuous map $f\colon X\to Y$, the assignment 
\[
X\mapsto \bD_\ast(X)=K_\ast(RL^*_c(H_X)), \,\,\,f\mapsto \bD_\ast(f)=\mathrm{Ad}_{V^f_t\ast}
\]
is a functor from the category $\mathcal{LC}$ of locally compact spaces to the category $\mathcal{GA}$ of abelian groups. It is shown in \cite{WY2020} that this functor is naturally the representable K-homology $RK_\ast$ on locally compact spaces. 

\noindent \textbf{Main Results:}
Let $G$ be a second countable, locally compact group, $X$ be a (second countable) locally compact, proper $G$-space and $B$ be a separable $G$-$C^*$-algebra. For any $X$-$G$-module $H_X$, we define the $G$-$C^*$-algebra (the representable localization algebra) $RL^*_c(H_X\otimes B)$ as the completion of the $\ast$-algebra of bounded, $G$-continuous, norm-continuous $\Compacts(H_X)\otimes B$-valued functions $T$ on $[1, \infty)$ such that 
 \begin{enumerate}
\item $\lim_{t\to \infty}\lVert[\phi, T_t]\rVert= \lim_{t\to \infty}\lVert\phi T_t - T_t \phi\rVert = 0$ for any $\phi \in C_0(X)$, and
 \item $T$ has uniform compact support in a sense that for some compact subset $K$ of $X$, $T_t=\chi_KT_t\chi_K$ for all $t\geq1$ ($\chi_K$ is the characteristic function of $K$).
 \end{enumerate}
 Here, $\Compacts(H_X)$ is the $G$-$C^*$-algebra of compact operators on $H_X$. 
We have a natural evaluation map (the forget-control map)
 \[
 \mathrm{ev}_{1}\colon RL^*_c(H_X\otimes B) \to \Compacts(H_X)\otimes B
 \]
 at $t=1$ and it induces a $\ast$-homomorphism 
 \[
  \mathrm{ev}_{1}\colon RL^*_c(H_X\otimes B)\rtimes_rG \to (\Compacts(H_X)\otimes B)\rtimes_rG.
  \]

 For a suitable choice of an $X$-$G$-module $H_X$ (a universal $X$-$G$-module, see Definition \ref{def_universalXG}) for each proper $G$-space $X$ and for a suitable choice of $G$-equivariant covering isometries $(V^f_t\colon H_X\to H_Y)_{t\in [1,\infty)}$ for each $G$-equivariant continuous map $f\colon X\to Y$, we obtain a well-defined functor (see Definition \ref{def_DBG})
\[
X\mapsto \bD_\ast^{B, G}(X)=K_\ast(RL^*_c(H_X\otimes B)\rtimes_rG),  \,\,\,f \mapsto \bD^{B, G}_\ast(f) = \mathrm{Ad}_{V^f_t\ast}
\]
from the category $\mathcal{PR}^G$ of (second countable, locally compact) proper $G$-spaces to the category $\mathcal{GA}$. The functor $\bD_\ast^{B, G}\colon \mathcal{PR}^G \to \mathcal{GA}$ satisfies several expected properties for the representable $G$-equivariant K-homology (see Theorem \ref{thm_coeff}) and one may extend $\bD_\ast^{B, G}$ to all (not necessarily locally compact) proper $G$-spaces by declaring $\bD_\ast^{B, G}(X)=\varinjlim_{Y\subset X, \mathrm{Gcpt}}\bD_\ast^{B, G}(Y)$.

 The forget-control map $\mathrm{ev}_{1}$ induces a group homomorphism (the forget-control map)
 \[
\cF \colon  \bD_\ast^{B, G}(X)  \to K_\ast(B\rtimes_rG).
 \]
 
The following is one of our main results.
\begin{theoremA*}(Theorem \ref{thm_main_isom}, Theorem \ref{thm_main_equivalent}) The forget-control map $\cF\colon  \bD_\ast^{B, G}(X)  \to K_\ast(B\rtimes_rG)$ is naturally equivalent to the Baum--Connes assembly map
\[
\mu_X^{B, G}\colon \varinjlim_{Y\subset X, \mathrm{Gcpt}}KK_\ast^G(C_0(Y), B) \to KK_\ast(\bC, B\rtimes_rG)
\]
for any second countable, locally compact group $G$, for any proper $G$-space $X$ and for any separable $G$-$C^*$-algebra $B$. That is, there is a natural isomorphism 
\[
\rho_X\colon  \bD_\ast^{B, G}(X) \to   \varinjlim_{Y\subset X, \mathrm{Gcpt}}KK_\ast^G(C_0(Y), B) 
\]
of the functors from $\mathcal{PR}^G$ to $\mathcal{GA}$ and the following diagram commutes
\begin{equation*}
\xymatrix{ \bD_\ast^{B, G}(X)   \ar[dr]^{\rho_X}_-{\cong}  \ar[rr]^{\cf} & &  K_\ast(B\rtimes_rG)  \\
   &  \varinjlim_{Y\subset X, \mathrm{Gcpt}}KK_\ast^G(C_0(Y), B). \ar[ur]^{\mu^{B, G}_X}   & 
   }
   \end{equation*}
\end{theoremA*}
 
 Let $RL^0_c(H_X\otimes B)$ be the kernel of the evaluation map $\mathrm{ev}_1$ on $RL^*_c(H_X\otimes B)$. The short exact sequence
 \[
 0 \to RL^0_c(H_X\otimes B) \to RL^*_c(H_X\otimes B) \to \Compacts(H_X)\otimes B \to 0
 \] 
 admits a $G$-equivariant c.c.p.\ splitting and thus it descends to the short exact sequence
 \[
  0 \to RL^0_c(H_X\otimes B)\rtimes_rG \to RL^*_c(H_X\otimes B)\rtimes_rG \to (\Compacts(H_X)\otimes B)\rtimes_rG \to 0.
  \]
  Hence, the following is a consequence of Theorem A.
 
 \begin{corollaryB*}(Corollary \ref{cor_N}) Let $G$ be a second countable, locally compact group and $B$ be a separable $G$-$C^*$-algebra. The Baum--Connes assembly map $\mu^{B, G}_r$ is an isomorphism if and only if
 \[
 K_\ast(RL^0_c(H_X\otimes B)\rtimes_rG)=0
 \]
for a universal $X$-$G$-module $H_X$ for $X=\EG$.
 \end{corollaryB*}

 We remark that for any ample $X$-module $H_X$, the $X$-$G$-module $H_X\otimes L^2(G)$ with structure given by the left (or right) regular representation of $C_0(X)$, is an example of a universal $X$-$G$-module (Proposition \ref{prop_universal}).

 One of the motivations of this article is to describe the so-called gamma element method (see below) for the Baum--Connes conjecture using the representable localization algebra $RL^*_c(H_X)$.
 
Recall that a $G$-$C^*$-algebra $A$ is called a proper $G$-$C^*$-algebra if for some proper $G$-space $X$, there is a non-degenerate, central representation of $C_0(X)$ to the multiplier algebra $M(A)$ of $A$. Kasparov's equivariant $KK$-theory $KK^G$ \cite{Kasparov88} is an additive category with separable $G$-$C^*$-algebras as objects and with Kasparov's group $KK^G(B_1, B_2)$ as the morphism group between two $G$-$C^*$-algebras $B_1, B_2$. The composition law is given by the Kasparov product. In particular, the group $KK^G(\bC, \bC)$ has a commutative ring structure and it is often denoted as $R(G)$ and called Kasparov's representation ring. The unit of  the ring $R(G)$ is denote by $1_G$. We say that an element $x$ in $R(G)$ factors through a proper $G$-$C^*$-algebra if there is a proper $G$-$C^*$-algebra $A$ and elements $y\in KK^G(\bC, A)$, $z\in KK^G(A, \bC)$ such that $x=z\circ y$ in $R(G)=KK^G(\bC, \bC)$. There is a natural restriction functor $KK^{G}(A, B) \to KK^{H}(A, B)$ for any closed subgroup $H$ of $G$, and in particular we have a ring homomorphism $R(G) \to R(H)$.

The following is a formulation of the gamma element method by Tu, formalizing the work of Kasparov \cite{Kasparov88}.

\begin{theorem*}(The gamma element method \cite{Tu00}) Suppose there is an element $x \in R(G)$ such that
\begin{enumerate}
\item[(a)] $x$ factors through a proper $G$-$C^*$-algebra, and  
\item[(b)]  $x=1_K$ in $R(K)$ for any compact subgroup $K$ of $G$.
\end{enumerate}
Then, $x$ is the unique idempotent in $R(G)$ characterized by these properties and called the gamma element $\gamma$ for $G$. The existence of $\gamma$ implies 
\begin{enumerate}
\item the Baum--Connes assembly map $\mu_r^{B, G}$ is split injective for any separable $G$-$C^*$-algebra $B$, and 
\item the image of $\mu_r^{B, G}$ coincides with the image of the action of $\gamma$ on the K-theory group $K_\ast(B\rtimes_rG)$ defined by the canonically defined ring homomorphism
\[
R(G) \to \mathrm{End} (K_\ast(B\rtimes_rG) ).
\]
\end{enumerate}
In particular, if $\gamma=1_G$ in $R(G)$, BCC holds for $G$. 
\end{theorem*}

Based on Theorem A, we seek a natural assumption under which we have a splitting (natural with respect to $B$)
\[
K_\ast(B\rtimes_rG) \to   \bD_\ast^{B, G}(X)  
\]
of the forget-control map $\cf$. By Theorem $A$, this automatically gives us a natural splitting of the Baum--Connes assembly map $\mu_r^{B, G}$. Our next result provides us such an assumption which is analogous to the one for the gamma element method, but it is purely in term of the representable localization algebra $RL^*_c(H_X)$. Since this assumption is satisfied when the gamma element exists, we would like to think this as a controlled algebraic aspect of the gamma element method.

For this, we first recall the definition of cycles of $R(G)=KK^G(\bC, \bC)$. A cycle for $R(G)$ consists of pairs of the form $(H, T)$ where $H$ is a (separable) graded $G$-Hilbert space and $T$ is an odd, self-adjoint, bounded, $G$-continuous operator on $H$ satisfying
\[
1-T^2, \,\,\, g(T)-T \in \Compacts(H)  \,\,\,\,\, (g\in G).
\]

A graded $X$-$G$-module is just the direct sum of $X$-$G$-modules $H_X^{0}$ (even space) and $H_X^{(1)}$ (odd space). 
\begin{definition*}[Definition \ref{def_XGlocalized}] An $X$-$G$-localized Kasparov cycle for $KK^G(\bC, \bC)$ is a pair $(H_X, T)$ of a graded $X$-$G$-module $H_X$ and an odd, self-adjoint, $G$-continuous element $T$ in the multiplier algebra $M(RL_c^*(H_X))$ satisfying for any $g\in G$,
\[
1-T^2, \,\,\, g(T)-T \in RL^*_c(H_X).
\]
\end{definition*}

If one prefers, an $X$-$G$-localized Kasparov cycle for $KK^G(\bC, \bC)$ is a cycle for $KK^G(\bC,  RL^*_c(H_X))$ but we do not take it as our definition.

The evaluation at $t=1$,
\[
\mathrm{ev}_{1}\colon RL^*_c(H_X) \to \Compacts(H_X)
\]
extends to
\[
\mathrm{ev}_{1}\colon M(RL^*_c(H_X)) \to \Linears(H_X).
\]
For any $T \in M(RL^*_c(H_X))$, we write $T_1 \in \Linears(H_X)$, its image by $\mathrm{ev}_{1}$.
 
 For any $X$-$G$-localized Kasparov cycle $(H_X, T)$, the pair $(H_X, T_1)$ is a cycle for $KK^G(\bC, \bC)$.

\begin{definition*}(Definition \ref{def_XGlocalized_element}) Let $X$ be a proper $G$-space. We say that an element $x$ in $R(G)=KK^G(\bC, \bC)$ is $X$-$G$-localized if there is an $X$-$G$-localized Kasparov cycle $(H_X, T)$ for $KK^G(\bC, \bC)$ such that 
\[
[H_X, T_1] = x  \,\,\text{in}\,\,\, KK^G(\bC, \bC).
\]
\end{definition*}

That is, an element $x$ in $R(G)$ is $X$-$G$-localized if it can be represented by a cycle $(H_X, T_1)$ with a graded $X$-$G$-module $H_X$, which extends to a continuous family $(H_X, T_t)_{t\in [1,\infty)}$ of cycles of $R(G)$ so that the family $T=(T_t)_{t\in [1,\infty)}$ of operators on $H_X$ multiplies the representable localization algebra $RL^*_c(H_X)$ and satisfies the support condition
\[
1-T^2, \,\,\, g(T)-T \in RL_c^*(H_X)  \,\,\,\,\, (g\in G).
\]
We remark that $T\in M(RL^*_c(H_X))$ holds if $\lim_{t\to \infty}\lVert[T_t, \phi]\rVert = 0$. 

The following is our controlled algebraic reformulation of the gamma element method. For $x\in R(G)$, let us denote by $x^{B\rtimes_rG}_\ast$, its image under the canonically defined ring homomorphism
\[
\xymatrix{
KK^G(\bC, \bC) \ar[r]^{\sigma_B} & KK^G(B, B) \ar[r]^-{j^G_r} & KK(B\rtimes_rG, B\rtimes_rG)  \ar[r]^-{} & \mathrm{End}(K_\ast(B\rtimes_rG) ).
}
\]

\begin{theoremC*}(Theorem \ref{thm_XGfactor}, Theorem \ref{thm_XGgamma}) Let $X$ be a proper $G$-space.
\begin{enumerate} 
\item Suppose that an element $x\in R(G)$ is $X$-$G$-localized, that is $x=[H_X, T_1]$ for an $X$-$G$-localized Kasparov cycle $(H_X, T)$ for $KK^G(\bC, \bC)$. Then, there is a homomorphism
\[
\nu^{B, T}\colon K_\ast(B\rtimes_rG) \to \bD^{B, G}_\ast(X)
\]
for any separable $G$-$C^*$-algebra $B$ which is natural with respect to a $G$-equivariant $\ast$-homomorphism $B_1\to B_2$, such that
\[
x^{B\rtimes_rG}_\ast =  \cf \circ \nu^{B, T} \colon K_\ast(B\rtimes_rG)  \to \bD^{B, G}_\ast(X)  \to K_\ast(B\rtimes_rG).
\]
In particular, $x^{B\rtimes_rG}_\ast$ factors through the Baum--Connes assembly map $\mu_r^{B ,G}$.
\item  If there is an $X$-$G$-localized element $x\in R(G)$ such that $x=1_K$ in $R(K)$ for any compact subgroup $K$, the Baum--Connes assembly map $\mu_r^{B, G}$ is split-injective for any $B$ and in this case the image of $\mu_r^{B, G}$ coincides with the image of  $x^{B\rtimes_rG}_\ast$. In particular, if $x=1_G$, BCC holds for $G$.
\end{enumerate}
\end{theoremC*}

The relationship with the gamma element is as follows:

\begin{theoremD*}(Theorem \ref{thm_XGgamma0})  The gamma element for $G$, if exists, is $X$-$G$ localized for $X=\EG$.
\end{theoremD*}

According to Mayer and Nest \cite{MeyerNest}, for any (second countable) locally compact group $G$, there is a separable $G$-$C^*$-algebra $\cP$ and the canonical Dirac element
\[
D \in KK^G(\cP, \bC)
\]
such that for $\cP_B=\cP\otimes B$, $D_B=D\otimes 1_B \in KK^G(\cP_B, B)$, the induced map
\[
j^G_r(D_B)_\ast \colon K_\ast(\cP_B\rtimes_rG) \to  K_\ast(B\rtimes_rG)
\]
on K-theory is naturally equivalent to the Baum--Connes assembly map $\mu^{B, G}_r$ for any separable $G$-$C^*$-algebra $B$.

According to Theorem A, the representable localization algebra $RL^*_c(H_X\otimes B)$ (for a universal $X$-$G$-module $H_X$ for $X=\EG$) can be, at a formal level, regarded as $\cP_B$ and the evaluation map
\[
\mathrm{ev}_1\colon RL^*_c(H_X\otimes B) \to \Compacts(H_X)\otimes B
\]
can be regarded as $D_B$. The identification is formal since, for example, $RL^*_c(H_X\otimes B)$ is not separable.

At least when $G$ is discrete, the comparison between $\cP$ and $RL^*_c(H_X)$ can be made stronger because of the following: 
\begin{theoremE*}(Theorem \ref{thm_discrete_natural_isom}, Theorem \ref{thm_otimesB}) Let $G$ be a countable discrete group, $X$ be a proper $G$-space which is $G$-equivariantly homotopic to a $G$-CW complex, and $H_X$ be a universal $X$-$G$-module. Then, the (well-defined) inclusion
\[
RL^*_c(H_X)\otimes B \to RL^*_c(H_X\otimes B)
\]
induces an isomorphism
\[
K_\ast((RL^*_c(H_X)\otimes B)\rtimes_rG) \cong  K_\ast(RL^*_c(H_X\otimes B)\rtimes_rG).
\]
Hence, the group homomorphism 
\[
K_\ast((RL^*_c(H_X)\otimes B)\rtimes_rG) \to K_\ast((\Compacts(H_X)\otimes B)\rtimes_rG) \cong K_\ast(B\rtimes_rG)
\]
induced by the evaluation map $\mathrm{ev}_1\colon RL^*_c(H_X) \to \Compacts(H_X)$ is naturally equivalent to the Baum--Connes assembly map $\mu^{B, G}_X$. 
\end{theoremE*}

That is, at a formal level, $RL^*_c(H_X)\otimes B$ (for $X=\EG$) can be regarded as $\cP_B=\cP\otimes B$ for any separable $G$-$C^*$-algebra $B$ (at least for $G$ discrete).

 \begin{corollaryF*}(Corollary \ref{cor_discrete_N})   Let $G$ be a countable discrete group and $B$ be a separable $G$-$C^*$-algebra. The Baum--Connes assembly map $\mu^{B, G}_r$ is an isomorphism if and only if
 \[
 K_\ast((RL^0_c(H_X)\otimes B)\rtimes_rG)=0
 \]
for a universal $X$-$G$-module $H_X$ for $X=\EG$.
 \end{corollaryF*}
 
In the last section, as an application of Theorem C, we extend the recently obtained new proof \cite{BGHN} of the Baum--Connes conjecture with coefficients for CAT(0)-cubical groups to the non-cocompact setting:

\begin{theoremG*}(Theorem \ref{thm_cube}) Let $G$ be a second countable, locally compact group $G$ which acts properly and continuously on a finite-dimensional CAT(0)-cubical space with bounded geometry by automorphisms. Then, the Baum--Connes assembly map $\mu^{B, G}_r$ is an isomorphism for any separable $G$-$C^*$-algebra $B$, i.e. BCC holds for $G$.
\end{theoremG*}
We remark that any group $G$ which acts properly on a CAT(0)-cubical space has the Haagerup approximation property \cite{CCJJV}, so $G$ is a-T-menable. Thus, BCC for these groups are already known by the Higson--Kasparov Theorem \cite{HK97}, \cite{HigsonKasparov}.

\section*{Acknowledgement } I would like to thank Arthur Bartels, Siegfried Echterhoff, Julian Kranz, Rufus Willett, Rudolf Zeidler for helpful discussions and comments. I would also like to thank Jacek Brodzki, Erik Guentner and Nigel Higson for their encouragement on extending the work \cite{BGHN} to the non-cocompact setting. 



\section*{Notations} 

Let $A$ be a $C^*$-algebra, $\E_1$, $\E_2$ be Hilbert $A$-modules and $X$ be a locally compact topological space.
\begin{itemize}

\item $\Linears(\E_1)$, resp. $\Compacts(\E_1)$, is the $C^*$-algebra of adjointable bounded, resp. compact, operators on $\E_1$.

\item $\Linears(\E_1, \E_2)$, resp. $\Compacts(\E_1, \E_2)$, is the space of the adjointable bounded, resp. compact, operators from $\E_1$ to $\E_2$.

\item $C_0(X, A)$ is the $C^*$-algebra of $A$-valued, bounded, norm-continuous functions on $X$ vanishing at infinity.

\item $C_b(X, A)$, resp. $C_{b, u}(X, A)$, is the $C^*$-algebra of $A$-valued, bounded, norm-continuous, resp. uniformly norm-continuous, functions on $X$.

\item $M(A)$ is the multiplier algebra of $A$.

\end{itemize}

We use the notations $\otimes$, resp. $\otimes_{\rmax}$, for the minimal, resp. maximal, tensor product and  the notations $\rtimes_r$, resp. $\rtimes_{\rmax}$, for the reduced, resp. maximal, crossed product.


\section{Representable localization algebra}

Most of the materials in this and the next section are basically what are proven in \cite{WY2020}. We will carefully review these because they will be important in later sections.

Let $X$ be a second countable, locally compact topological space (a locally compact space in short). An $X$-module $H_X$ (or a module over $X$) is a separable Hilbert space equipped with a representation of the $C^*$-algebra $C_0(X)$ which is non-degenerate, i.e. $C_0(X)H_X$ is dense in $H_X$. Any Borel function on $X$ is naturally represented on $H_X$ by Borel functional calculus. The characteristic function of a Borel subset $E$ of $X$ is denoted by $\chi_E$. We will use a convention that a module over the empty set is the zero Hilbert space.

An $X$-module $H_X$ is ample if no nonzero element in $C_0(X)$ acts as a compact operator.

Let $X, Y$ be locally compact spaces. The support $\mathrm{supp}(T)$ of a bounded operator $T$ from an $X$-module $H_X$ to a $Y$-module $H_Y$ is defined as the set of $(y, x)$ in $Y\times X$ such that $\chi_VT\chi_U\neq0$ for all open neighborhoods $U$ of $x$ and $V$ of $y$. The support is a closed subset of $Y\times X$.

When $X$ is a metric space, the propagation of $T$ in $\Linears(H_X)$ is defined as 
\[
\mathrm{prop}(T) = \sup\{ \, d(x,y)\mid  \, (x, y) \in \mathrm{supp}(T)\, \}.
\]

For an $X$-module $H_X$, we introduce a representable localization algebra $RL^*_c(H_X)$ as follows. It is slightly different from the representable localization algebra $RL^*(H_X)$ defined in \cite[Definition 9.4.1]{WY2020} but the two algebras behave basically in the same way. Our definition has an advantage that the ideal $RL^*_0(H_X)$ of ``the negligible part'' of $RL^*_c(H_X)$ is $C_0([1,\infty), \Compacts(H_X))$, not $C_0([1,\infty), \Linears(H_X))$. This will be important when we define the forget-control map, for example.

\begin{definition}\label{def_alg}  Let $H_X$ be an $X$-module. We define a $\ast$-subalgebra $RL^{\mathrm{alg}}_c(H_X)$ of $C_b([1, \infty), \Compacts(H_X))$ consisting of bounded, norm-continuous $\Compacts(H_X)$-valued functions $T\colon t\mapsto T_t$ on $[1, \infty)$, such that
\begin{enumerate}
\item $T$ has uniform compact support in a sense that there is a compact subset $K$ of $X$ such that $T_t=\chi_KT_t\chi_K$ (i.e. $\supp(T_t)\subset K\times K$) for all $t\geq1$, 
\item for any $\phi$ in $C_0(X)$, we have
\[
 \lim_{t\to \infty }\lVert[\phi, T_t]\rVert=\lim_{t\to \infty}\lVert\phi T_t-T_t\phi\rVert = 0.
\]
\end{enumerate}
We define $RL^{\mathrm{alg}}_u(H_X)$ to be a subalgebra of $RL^{\mathrm{alg}}_c(H_X)$ consisting of uniformly norm-continuous functions.
\end{definition}

\begin{definition} Let $H_X$ be an $X$-module. We define a $C^*$-algebra $RL_c^\ast(H_X)$ as the norm completion of $RL^{\mathrm{alg}}_c(H_X)$ inside $C_b([1, \infty), \Compacts(H_X))$. A $C^*$-algebra $RL_u^\ast(H_X)$ is defined as the completion of $RL^{\mathrm{alg}}_u(H_X)$. We call the algebras $RL_c^\ast(H_X)$, $RL_u^\ast(H_X)$ as the (continuous/ uniformly continuous) representable localization algebras.
\end{definition}

\begin{remark} When $X$ is compact, the first condition in Definition \ref{def_alg} is vacuous and the algebra $RL_u^\ast(H_X)$ is same as the localization algebra $\mathcal{C}_L(\pi)$ introduced and studied in \cite{DWW18} for the structure map $\pi\colon C(X) \to \Linears(H_X)$. For general $X$, $RL_u^\ast(H_X)$ is just the inductive limit (the union) of  $RL_u^\ast(H_Y)$ over the compact subspaces $Y$ of $X$ where $H_Y=\chi_YH_X$.
\end{remark}

We will mostly study $RL_c^\ast(H_X)$ but the entire discussion and results on $RL_c^\ast(H_X)$ have obvious analogues for $RL_u^\ast(H_X)$.

The non-degeneracy of the representation of $C_0(X)$ on $H_X$ implies that $RL_u^\ast(H_X)$ and hence $RL_c^\ast(H_X)$ contain the following (essential) ideal
\[
RL_0^\ast(H_X)=C_0([1, \infty), \Compacts(H)).
\] 
We define $RL_{c, Q}^\ast(H_X)$ to be the quotient of $RL_c^\ast(H_X)$ by the ideal $RL_0^\ast(H_X)$. Similarly, $RL_{u, Q}^\ast(H_X)$ is the quotient of $RL_u^\ast(H_X)$ by $RL_0^\ast(H_X)$. Thus, we have the following diagram of short exact sequences
\begin{equation*}
\xymatrix{
0  \ar[r]^-{} &  RL_0^\ast(H_X)   \ar[r]^-{} \ar[d]^-{=} &     RL_u^\ast(H_X) \ar[r]^-{} \ar[d]^-{} &   RL_{u, Q}^\ast(H_X)  \ar[r]^-{} \ar[d]^-{} &    0   \\
0  \ar[r]^-{}     &   RL_0^\ast(H_X)            \ar[r]^-{}         &  RL_c^\ast(H_X)           \ar[r]^-{} &               RL_{c, Q}^\ast(H_X) \ar[r]^-{}       &           0.
}
\end{equation*}

\begin{proposition}\label{prop_quotient_isom} The quotient map from $RL^*_c(H_X)$ to $RL^*_{c, Q}(H_X)$ induces an isomorphism on the K-theory groups. The same holds for the quotient map from $RL^*_u(H_X)$ to $RK^*_{u, Q}(H_X)$.
\end{proposition}
\begin{proof} This follows from $K_\ast(RL_0^\ast(H_X))=K_\ast(C_0([1, \infty), \Compacts(H)))=0$.
\end{proof}

\begin{lemma}(See \cite[Proposition 6.1.1]{WY2020})   \label{lem_commutator0} Let $X$ be a compact space, $H_X$ be an $X$-module and $T \in C_b([1, \infty), \Linears(H_X))$. Fix any metric $d$ on $X$. The following are equivalent:
\begin{enumerate}
\item For any $\phi \in C(X)$, $\lim_{t\to \infty }\lVert[\phi, T_t]\rVert= 0$.
\item There is $S \in C_b([1, \infty), \Linears(H_X))$ such that propagation $\prop(S_t) \to 0$ as $t \to \infty$ and such that $\limt\lVert T_t-S_t\rVert= 0$.
\item There is  $S \in C_b([1, \infty), \Linears(H_X))$ such that for any open neighborhood $U$ of the diagonal in $X\times X$ there exists $t_U\geq1$ so that for all $t> t_U$, $\mathrm{supp}(S_t)\subset U$ and such that $\limt\lVert T_t-S_t\rVert= 0$.
\end{enumerate}
The same is true if we replace $C_b([1, \infty), \Linears(H_X))$ to any one of $C_{b, u}([1, \infty), \Linears(H_X))$, $C_b([1, \infty), \Compacts(H_X))$ and $C_{b, u}([1, \infty), \Compacts(H_X))$ everywhere.
\end{lemma}
\begin{proof} The equivalence of (1) and (2) is proven in \cite[Proposition 6.1.1]{WY2020} and the condition on $S_t$ in (2) and the one in (3) are equivalent. The proof is still valid when we use $C_b([1, \infty), \Compacts(H_X))$ instead. One can replace continuous to uniform continuous because if $T_t$ is uniformly continuous, $S_t$ is continuous and if $\lVert T_t-S_t\rVert\to 0$ as $t\to \infty$, then $S_t$ is uniformly continuous. 
\end{proof}

For any locally compact space $X$, let $X^+=X\cup \{\infty\}$ be the one-point compactification of $X$. Any $X$-module is naturally an $X^+$-module.

\begin{lemma}\label{lem_commutator} Let $X$ be a locally compact space, $H_X$ be an $X$-module and $T \in C_b([1, \infty), \Linears(H_X))$. Fix any metric $d$ on $X^+$. The following are equivalent:
\begin{enumerate}
\item For any $\phi \in C_0(X)$, $\lim_{t\to \infty }\lVert[\phi, T_t]\rVert= 0$.
\item There is $S \in C_b([1, \infty), \Linears(H_X))$ such that propagation $\mathrm{prop}(S_t) \to 0$ as $t \to \infty$ and such that $\limt\lVert T_t-S_t\rVert= 0$.
\item There is  $S \in C_b([1, \infty), \Linears(H_X))$ such that for any open neighborhood $U$ of the diagonal in $X^+\times X^+$ there exists $t_U\geq1$ so that for all $t> t_U$, $\mathrm{supp}(S_t)\subset U$ and such that $\limt\lVert T_t-S_t\rVert= 0$.
\end{enumerate}
The same is true if we replace $C_b([1, \infty), \Linears(H_X))$ to any one of $C_{b, u}([1, \infty), \Linears(H_X))$, $C_b([1, \infty), \Compacts(H_X))$ or $C_{b, u}([1, \infty), \Compacts(H_X))$ everywhere.
\end{lemma}
\begin{proof} This follows from Lemma \ref{lem_commutator0} by viewing $H_X$ as an $X^+$-module.
\end{proof}

\begin{proposition}\label{prop_same} The $C^*$-algebra $RL_c^\ast(H_X)$ coincides with the completion of the $\ast$-subalgebra $R\bL^{\mathrm{alg}}_c(H_X)$ of $C_b([1, \infty), \Compacts(H_X))$ that consists of $T$ such that
\begin{enumerate}
\item $T$ has eventually, uniform compact support in a sense that there is a compact subset $K$ of $X$ and $t_0\geq1$ such that $T_t=\chi_KT\chi_K$ for all $t\geq t_0$, 
\item for any open neighborhood $U$ of the diagonal in $X^+\times X^+$ there exists $t_U\geq1$ such that for all $t> t_U$, $\mathrm{supp}(T_t)\subset U$.
\end{enumerate}
\end{proposition}
\begin{proof} Since both $RL_c^\ast(H_X)$ and the completion of $R\bL^{\mathrm{alg}}_c(H_X)$ contain the ideal $RL_0^*(H_X)$, using Lemma \ref{lem_commutator}, we can see that each one of the algebras contains all the generators of the other.
\end{proof}

Any bounded Borel function on $X$ multiplies $RL_c^\ast(H_X)$ and hence the ideal $RL_0^\ast(H_X)$ so it naturally defines a multiplier on the quotient $RL_{c, Q}^\ast(H_X)$. In particular, we have a $\ast$-homomorphism
\begin{equation}\label{eq_Xalgebra}
C_b(X) \to M(RL_{c, Q}^\ast(H_X)).
\end{equation}

Recall that a $C_0(X)$-algebra is a $C^*$-algebra $A$ equipped with a representation 
\[
C_0(X) \to M(A)
\]
from $C_0(X)$ to the multiplier algebra of $A$ which is non-degenerate and central. Here, non-degenerate means $C_0(X)A$ is dense in $A$ and central means it maps $C_0(X)$ to the center of the multiplier algebra.

\begin{proposition} The $C^*$-algebra $RL_{c, Q}^\ast(H_X)$ is naturally a $C_0(X)$-algebra. That is, the natural representation \eqref{eq_Xalgebra} of $C_0(X)$ to the multiplier algebra $M(RL_{c, Q}^\ast(H_X))$ is non-degenerate and central. 
\end{proposition}
\begin{proof} The image of $RL^{\mathrm{alg}}_c(H_X)$ in $RL_{c, Q}^\ast(H_X)$ is dense. On this dense subalgebra, the non-degeneracy follows from the first condition in Definition \ref{def_alg}. The centrality follows from the second condition in Definition \ref{def_alg}.
\end{proof}

\begin{remark} In fact, we can see that $RL_{c, Q}^\ast(H_X)$ is the largest possible $C_0(X)$-subalgebra inside the quotient $C_b([1, \infty),  \Compacts(H_X)) / C_0([1, \infty), \Compacts(H_X))$ with respect to the natural $C_0(X)$-action. 
\end{remark}

Let $H_X$ be an $X$-module. Let $L=C_0[1, \infty)$. Consider a Hilbert $L$-module $H_X\otimes L$. The $C^*$-algebra $\Compacts(H_X\otimes L)$ of ($L$-)compact (adjointable) operators on $H_X\otimes L$ is naturally identified as $\Compacts(H_X)\otimes L \cong C_0([1, \infty), \Compacts(H_X))$. Similarly, the $C^*$-algebra $\Linears(H_X\otimes L)$ of adjointable operators on $H_X\otimes L$ is identified as $C_{b, \mathrm{SOT}^*}([1, \infty), \Linears(H_X))$ consisting of bounded, $\mathrm{SOT}^*$-continuous $\Linears(H_X)$-valued functions. We sometimes view $RL_c^*(H_X)$ as a subalgebra of $\Linears(H_X\otimes L)$ so we have
\begin{align*}
\Compacts(H_X\otimes L) \cong C_0([1, \infty), \Compacts(H_X)) = RL^*_0(H_X) \subset RL_c^*(H_X)   \\ 
 \subset C_b([1, \infty), \Compacts(H_X)) \subset  C_{b, \mathrm{SOT}^*}([1, \infty), \Linears(H_X)) \cong \Linears(H_X\otimes L).
\end{align*}
Note that the multiplier algebra $M(RL^*_c(H_X))$ of $RL^*_c(H_X)$ is naturally a subalgebra of $\Linears(H_X\otimes L)$. This is because for any $C^*$-algebra $J$, if a $C^*$-subalgebra $B\subset M(J)$ contains $J$, then $M(B)$ naturally coincides with the subalgebra of $M(J)$ consisting of multipliers of $B$ inside $M(J)$.

The following is an easy exercise.

\begin{lemma}\label{lem_mult} Let $T\in \Linears(H_X\otimes L)\cong C_{b, \mathrm{SOT}^*}([1, \infty), \Linears(H_X))$ such that for any $\phi \in C_0(X)$, $\limt \lVert[T_t, \phi]\rVert= 0$. Then $T\in M(RL^*_c(H_X))$. In particular, if $\prop(T_t)\to 0$ as $t\to \infty$ with respect to a (any) fixed metric on $X^+$, then  $T\in M(RL^*_c(H_X))$. 
\end{lemma}


\section{Representable K-homology}
Almost all the materials in this section are what are proven in \cite{WY2020}, in particular in \cite[Section 9.4]{WY2020}. We will review these because they will be important in later sections.
As before, we mainly study $RL_{c}^\ast(H_X)$ but the entire discussion and results on $RL_{c}^\ast(H_X)$ have obvious analogues for $RL^*_u(H_X)$.

Let $X$ and $Y$ be locally compact spaces. Let $H_X , H_Y$ be an $X$-module and a $Y$-module respectively. 

Given a continuous map $f\colon X\to Y$, a family of isometries $(V_t\colon H_X \to H_Y)_{t\in[1, \infty)}$ is called a continuous cover of $f$ \cite[Definition 4.4.8]{WY2020} if 
 \begin{enumerate} 
\item the function $t\mapsto V_t$ from $[1, \infty)$ to $\Linears(H_X, H_Y)$ is uniformly norm-continuous, and
\item for any open neighborhood $U$ of the diagonal in $Y^+\times Y^+$, there exists $t_U\geq1$ such that for all $t\geq t_U$, 
\[
\mathrm{supp}(V_t) \subset \{\, (y, x)\in Y\times X \mid  (y, f(x)) \in U \}. 
\]
\end{enumerate}

For example, if $H_X$ is an $X$-module and $f\colon X\to Y$ is a continuous map, we may view $H_X$ as a $Y$-module via $f^*\colon C_0(Y)\to C_b(X)$. The obtained representation of $C_0(Y)$ on $H_X$ is non-degenerate because $f^\ast\colon C_0(X) \to M(C_0(X))=C_b(X)$ is non-degenerate on $C_0(X)$. Let $(H_X)_Y$ be this $Y$-module. Then, the identity map $V\colon H_X\to (H_X)_Y$ is a strict cover of $f$ in a sense that for any open neighborhood $U$ of the diagonal in $Y^+\times Y^+$, we have
\[
\mathrm{supp}(V) \subset \{\, (y, x)\in Y\times X \mid  (y, f(x)) \in U \},
\]
or equivalently, the support of $V$ is contained in the graph of $f$. In particular, $V$ as a constant family is a continuous cover of $f$ from $H_X$ to $(H_X)_Y$.

\begin{lemma}\label{lem_coverid} A family of isometries $(V_t\colon H_X \to H_Y)_{t\in[1, \infty)}$ is a continuous cover of $f\colon X\to Y$ if and only if it is a continuous cover of the identity map on $Y$ when we view $H_X$ as a $Y$-module $(H_X)_Y$ via $f^\ast\colon C_0(Y) \to C_b(X)$.  
\end{lemma}
\begin{proof} It can be checked directly that for any $V\colon H_X\to H_Y$ and for any open neighborhood $U$ of the diagonal in $Y^+\times Y^+$, for the following conditions
\begin{enumerate}
\item $\supp(V\colon (H_X)_Y \to H_Y) \subset \{\, (y_1, y_2)\in Y\times Y \mid  (y_1, y_2) \in U \}$,
\item $\supp(V\colon H_X \to H_Y) \subset \{\, (y, x)\in Y\times X \mid  (y, f(x)) \in U \}$,
\item $\supp(V\colon (H_X)_Y \to H_Y) \subset \{\, (y_1, y_2)\in Y\times Y \mid  (y_1, y_2) \in \bar U \}$,
\end{enumerate}
we have (1) $\implies$ (2) $\implies$ (3). The assertion follows from this.
\end{proof}

Given a continuous cover $(V_t\colon H_X \to H_Y)_{t\in[1, \infty)}$ of a continuous map $f\colon X\to Y$, the conjugation by $V_t$ defines a $\ast$-homomorphism
\[
\mathrm{Ad}_{V_t} \colon RL_{c}^\ast(H_X) \to RL_{c}^\ast(H_Y).
\]
This is because $\mathrm{Ad}_{V_t}$ maps $R\mathbb{L}^{\mathrm{alg}}_c(H_X)$ (see Proposition \ref{prop_same} for this algebra) to $R\mathbb{L}^{\mathrm{alg}}_c(H_Y)$. This $\ast$-homomorphism depends on the continuous cover $V_t$ of $f$, but the induced map on their K-theory groups is independent of the choice of $V_t$. This is because given two continuous covers $(V_{i, t}\colon H_X\to H_Y)_{t\in [1, \infty)}$ $(i=1,2)$ of $f$, the two maps
\[
T_t \mapsto \begin{bmatrix} \mathrm{Ad}_{V_{1,t}}(T_t) & 0  \\ 0 & 0  \end{bmatrix}, \,\,\,
T_t \mapsto \begin{bmatrix} 0 & 0  \\ 0 & \mathrm{Ad}_{V_{2,t}}(T_t)  \end{bmatrix}
\]
from $RL_{c}^\ast(H_X)$ to the matrix algebra $M_2(RL_{c}^\ast(H_Y))$ are conjugate to each other by the unitary 
\[
 \begin{bmatrix} 1-V_{1,t}V_{1,t}^\ast & V_{1,t}V_{2,t}^\ast  \\ V_{2,t}V_{1,t}^\ast & 1-V_{2,t}V_{2,t}^\ast  \end{bmatrix}
\]
in the $2\times 2$ matrix algebra $M_2(M(RL_{c}^\ast(H_Y)))$ of the multiplier algebra. Here, the following important fact is used: for any two continuous covers $V_{1,t}$, $V_{2,t}$ of $f$, the partial isometries $V_{1,t}V_{2,t}^\ast$ multiplies $RL_{c}^\ast(H_Y)$. Indeed, for a fixed metric $d$ on $Y^+$ and for any $\epsilon>0$, $\supp(V_{1,t}V_{2,t}^\ast)$ is contained in the closure of 
\[
\{ (y_1, y_2) \mid \text{there is $x \in X$ such that $d(y_1, f(x)), d(y_2, f(x)) < \epsilon$ } \}
\]
for large enough $t$. It follows that $\prop (V_{1,t}V_{2,t}^\ast) \to 0$ as $t\to 0$ with respect to $d$. By Lemma \ref{lem_mult}, the family $t\mapsto V_{1,t}V^\ast_{2,t}$ multiplies $RL^*_c(H_Y)$.

Note that $\mathrm{Ad}_{V_t}$ for a cover $V_t$ of $f\colon X\to Y$ induces a $\ast$-homomorphism from $RL_{c, Q}^\ast(H_X)$ to $RL_{c, Q}^\ast(H_Y)$ as it maps the ideal $RL_{0}^\ast(H_X)$ to $RL_{0}^\ast(H_Y)$ and the induced map on their K-theory groups is also independent of the choice of a cover.

If $H_Y$ is an ample $Y$-module, a continuous cover $(V_t\colon H_X \to H_Y)_{t\in[1, \infty)}$ exists for any continuous map $f\colon X\to Y$ and for any $X$-module $H_X$ \cite[Corollary 4.4.7]{WY2020}. 

\begin{definition}\cite[Definition 9.4.5]{WY2020} Choose any ample $X$-module $H_X$ for each locally compact space $X$ and any continuous cover $V^f_t\colon H_X\to H_Y$ for each continuous map $f\colon X\to Y$. A functor $\mathbb{D}_\ast$ from the category $\mathcal{LC}$ of (second countable) locally compact spaces to the category $\mathcal{GA}$ of graded abelian groups is defined as
\[
\mathbb{D}_\ast(X)=K_\ast(RL^\ast_c(H_X)),
\]
\[
\mathbb{D}_\ast(f\colon X\to Y)=\mathrm{Ad}_{V^f_t\ast} \colon K_\ast(RL^\ast_c(H_X)) \to K_\ast(RL^\ast_c(H_Y)).
\]
\end{definition}

\begin{proposition}\cite[Theorem 9.4.4]{WY2020}\label{prop_welldef} The functor  $\mathbb{D}_\ast$ from $\mathcal{LC}$ to $\mathcal{GA}$ is well-defined. The functor does not depend on the choice of ample modules $H_X$ up to canonical equivalence.
\end{proposition}
\begin{proof} There is a technical point which is not mentioned in \cite{WY2020}, which we now explain. Let  $f_1\colon X\to Y$, $f_2\colon Y\to Z$ be continuous maps and $V_1\colon H_X\to H_Y$ and $V_2\colon H_Y\to H_Z$ be continuous covers of $f_1$ and $f_2$ respectively. The thing is that $V_2V_1$ is a continuous cover of $f_2\circ f_1$ if $f_2$ is proper but it may not be in general. Also, if $V_3\colon H_X\to H_Z$ is a continuous cover of $f_2\circ f_1$, $V_3 (V_2V_1)^*$ may not multiply $RL^*_c(H_Z)$ if $f_2$ is not proper. Because of this, we reduce the functoriality to the case of proper maps. Thus, we first show the representability of $\mathbb{D}_\ast(X)$: if $(K_i)_{i \in I}$ is the net of compact subsets of $X$, ordered by the inclusion, the canonical maps $\bD_\ast(K_i) \to \bD_\ast(X)$ induce a natural isomorphism 
\begin{equation}\label{eq_representable}
\varinjlim_{i \in I}\bD_\ast(K_i) \cong \bD_\ast(X).
\end{equation}
Note that the functoriality for proper maps is used for defining this inductive system. For this, it suffices to show that if $U_n$ is an increasing sequence of relatively compact, open subsets of $X$ such that $\cup_n U_n=X$, then for their closures $K_n=\bar U_n$, the natural inclusions induce an isomorphism 
\[
\varinjlim_nK_\ast(RL_c^\ast(H_{K_n})) \cong K_\ast(RL_c^\ast(H_{X})).
\]
Since $K_n=\bar U_n$, the subspace $\chi_{K_n}H_X$ is an ample $K_n$-module and we may assume $H_{K_n}=\chi_{K_n}H_X$. Then, we see that $RL_c^\ast(H_{K_n})=\chi_{K_n}RL_c^\ast(H_{X})\chi_{K_n}$ and $RL_c^\ast(H_{K_n})$ is an increasing sequence of $C^*$-subalgebras of $RL_c^\ast(H_{X})$ whose union is dense in $RL_c^\ast(H_{X})$. The claim follows from the continuity of K-theory. Next, we note that as in \cite[Corollary 9.4.10]{WY2020}, this identification \eqref{eq_representable} is compatible with $\mathbb{D}_\ast(f\colon X\to Y)$: if $(K_i)_{i \in I}$, $(K'_j)_{j\in J}$ are the nets of compact subsets of $X$ and $Y$ respectively, and if we consider the map
\[
\varinjlim_{i\in I} \bD_\ast(f\mid_{K_i})\colon  \varinjlim_{i \in I} \bD_\ast(K_i)   \to  \varinjlim_{j \in J}\bD_\ast(K'_j) 
\]
defined as the limit of 
\[
\bD_\ast(K_i)  \to \bD_\ast(f(K_i))    \to  \varinjlim_{j \in J}\bD_\ast(K'_j) 
\]
which is the composition of $\bD_\ast(f\mid_{K_i}\colon K_i\to f(K_i))$ and the natural map, the following diagram commutes
\begin{equation*}
\xymatrix{
\varinjlim_{i \in I}\bD_\ast(K_i) \ar[r]^-{\varinjlim_{i\in I} \bD_\ast(f\mid_{K_i})} \ar[d]^-{\cong} &     \varinjlim_{j \in J}\bD_\ast(K_j) \ar[d]^-{\cong}  \\
K_\ast(RL_c^\ast(H_{X}))           \ar[r]^-{\mathbb{D}_\ast(f)}         &  K_\ast(RL_c^\ast(H_{Y})). 
}
\end{equation*}
The functoriality $\bD_\ast(f_2)\circ \bD_\ast(f_1) = \bD_\ast(f_2\circ f_1)$ for not necessarily proper maps $f_1\colon X\to Y$, $f_2\colon Y\to Z$ now follows from this.
\end{proof}

Note that we have naturally
\[
\mathbb{D}_\ast(X)=K_\ast(RL^\ast_c(H_X)) \cong K_\ast(RL^\ast_{c, Q}(H_X))
\]
by Proposition \ref{prop_quotient_isom}.

\begin{theorem}\label{thm_homology}\cite[Section 9.4]{WY2020} The functor $\mathbb{D_\ast}$ satisfies the following: 
\begin{enumerate}
\item  $\mathbb{D}_\ast(\mathrm{empty\, set})\cong 0$.
\item $\mathbb{D}_\ast(\mathrm{point})\cong  \bigg{\{ } \begin{array}{cc}  \mathbb{Z} & \ast=0, \\ 0 & \ast=1.  \end{array} $
\item Representable: if $(K_i)_{i \in I}$ is the net of compact subsets of $X$, ordered by the inclusion, the canonical maps $\bD_\ast(K_i) \to \bD_\ast(X)$ induce a natural isomorphism 
\[
\varinjlim_{i \in I}\bD_\ast(K_i) \cong \bD_\ast(X).
\]
\item Mayer--Vietoris sequence for an open cover: if $X=U\cup V$ for open subsets $U$ and $V$ of $X$, we have a natural Meyer--Vietoris sequence
\[
\xymatrix{
 \bD_0(U\cap V) \ar[r]^-{}&    \bD_0(U) \oplus \bD_0(V) \ar[r]^-{} &   \bD_0(X) \ar[d]^-{}   \\
\bD_1(X)     \ar[u]^-{}         &   \bD_1(U) \oplus \bD_1(V)        \ar[l]^-{} &            \bD_1(U\cap V) \ar[l]^-{}.
}
\]
\item  Homotopy invariance: if $h\colon X\times[0 ,1]\to Y$ is a continuous homotopy between $f_0, f_1\colon X\to Y$, then $\bD_\ast(f_0)=\bD_\ast(f_1)$.
\end{enumerate}
\end{theorem}
\begin{proof}
We review proofs of these properties of $\bD_\ast$ because they will be relevant when we consider an equivariant setting.

(1) If $X$ is the empty set, $H_X=0$ and $RL_c^\ast(H_X)=0$.

(2) If $X$ is a point, $RL_c^\ast(H_X)=C_b([1, \infty), \Compacts(H))$ for a separable, infinite-dimensional Hilbert space $H$. In the uniformly-continuous case, 
\[
K_\ast(C_{b, u}([1, \infty), \Compacts(H))\cong  \bigg{\{}  \begin{array}{cc}  \mathbb{Z} & \ast=0 \\ 0 & \ast=1  \end{array}
\]
is proven in \cite[Proposition 6.3.3]{WY2020}  by showing that the kernel $I_u$ of the evaluation map at $1$
\[
\mathrm{ev}_1\colon C_{b, u}([1, \infty), \Compacts(H)) \to \Compacts(H),
\]
has zero K-theory groups by a simple Eilenberg swindle argument. In \cite[Theorem 3.4]{WY2021}, it is shown that the canonical inclusion induces an isomorphism
\[
 K_\ast(C_{b, u}([1, \infty), \Compacts(H))) \cong  K_\ast(C_{b}([1, \infty), \Compacts(H))),
\] 
with a more elaborate Eilenberg Swindle argument. The assertion follows from these.

(3) We have already proved this in the proof of Proposition \ref{prop_welldef}.

(4) For any open subset $U$ of $X$, $C_0(U)H_X$ is an ample $U$-module and we may assume $H_U=C_0(U)H_X$. We see that $RL_c^\ast(H_{U})$ coincides with the $C^*$-subalgebra of $RL_c^\ast(H_{X})$ generated by $C_0(U)RL_c^\ast(H_{X})C_0(U)$. Now let $X=U\cup V$ with $U, V$ open. Recall that $RL_{c, Q}^\ast(H_X)$ is a $C_0(X)$-algebra.  Let $RL_{c, Q}^\ast(H_X)_U$ be the ideal of $RL_{c, Q}^\ast(H_X)$ generated by $C_0(U)RL_{c, Q}^\ast(H_X)$. Define $RL_{c, Q}^\ast(H_X)_V$, $RL_{c, Q}^\ast(H_X)_{U\cap V}$ analogously. We see that $RL_{c, Q}^\ast(H_{U})$, $RL_{c, Q}^\ast(H_{V})$ and $RL_{c, Q}^\ast(H_{U\cap V})$ are naturally identified with $RL_{c, Q}^\ast(H_X)_U$, $RL_{c, Q}^\ast(H_X)_V$, and $RL_{c, Q}^\ast(H_X)_{U\cap V}$ respectively. These ideals of $RL_{c, Q}^\ast(H_{X})$ satisfy
\[
RL_{c, Q}^\ast(H_X)_U + RL_{c, Q}^\ast(H_X)_V = RL_{c, Q}^\ast(H_X),
\]
\[
RL_{c, Q}^\ast(H_X)_U \cap RL_{c, Q}^\ast(H_X)_V = RL_{c, Q}^\ast(H_X)_{U\cap V}.
\]
The corresponding to these ideals, we have a Mayer--Vietoris sequence (see \cite[Proposition 2.7.15]{WY2020}) which is the desired one. The naturality can be checked directly. This argument is essentially same as the one given in \cite[Proposition 9.4.13]{WY2020}.

(5) We recall a proof of the homotopy invariance in detail for later use. We follow the argument in \cite[Proposition 6.4.14]{WY2020}. See also the proof of \cite[Proposition 3.7]{Yu1997}. For any locally compact space $X$ and for $ r \in[0 ,1]$, let 
\[
f_r\colon X\times [0, 1] \to X\times [0 ,1], \,\, (x, t)\to (x, (1-r)t).
\]
It suffices to show the following claim:
\[
\bD_\ast(f_1) = \bD_\ast(f_0) (=\mathrm{Id}).
\]
We let $Z=[0 ,1]\cap \mathbb{Q}$, regarded as a discrete set, and consider $l^2$-space $l^2(Z)$. We fix a separable infinite-dimensional Hilbert space $H$ with a decomposition
\[
H=\bigoplus_{z\in Z} H_z
\]
with each $H_z$ infinite-dimensional. Let $W_z\colon H\to H$ be an isometry with range $H_z$.

For an ample $X$-module $H_X$, we use the following ample $X\times [0, 1]$-module
\[
H_{X\times[0,1]}=H_X\otimes l^2(Z)\otimes H
\]
where $\phi \in C_0(X)$ acts as
\[
\phi (u\otimes \delta_z \otimes v)  = \phi u \otimes \delta_z\otimes v,
\]
and $\phi \in C[0, 1]$ acts as
\[
\phi (u\otimes \delta_z \otimes v)  = u \otimes  \phi(z) \delta_z\otimes v.
\]
For any $r\in Z$, we define an isometry $W(r)$ on $H_{X\times[0,1]}$ by
\[
W(r)\colon u\otimes \delta_z \otimes v \mapsto u\otimes \delta_{(1-r)z} \otimes W_zv.
\]
The isometries $W(r)$ satisfy the following:
\begin{enumerate}
\item The support $\mathrm{supp}(W(r))$ is the graph of $f_r\colon X\times [0, 1] \to X\times [0, 1]$. That is,
\[
\mathrm{supp}(W(r)) = \{ \, ((x, (1-r)s), (x, s) ) \mid (x, s) \in X\times[0, 1] \}.
\]
\item For $T\in \Linears(H_{X\times[0,1]})$, if $((x_1, s_1), (x_2, s_2)) \in \supp( W(r)TW(r)^\ast )$, then
\[
\bigg{\{} \begin{array}{cc} ((x_1, (1-r)^{-1}s_1), (x_2, (1-r)^{-1}s_2))  \in \supp(T) & (r\neq 1)   \\ ((x_1, s_0), (x_2, s'_0))  \in \supp(T)  \,\,\,  \text{for some $s_0, s_0'$ in $[0,1]$ and $s_1=s_2=0$} &  (r=1).   \end{array} 
\]
\end{enumerate}
For $n \in \mathbb{N}_{>0}\cup \{\infty\}$, we define a uniformly continuous family $(V_{n,t}\colon H_{X\times[0,1]} \to H_{X\times[0,1]})_{t\in[1, \infty)}$ of isometries. For $n=\infty$, we set
\[
V_{\infty,t}=W(0) \,\,\, (1\leq t< \infty).
\]
For $n\in \N$, we set
\[
V_{n,t}=\bigg{\{}  \begin{array}{cc}  W(0) & (0\leq t < n),   \\ W(1) &  (2n \leq  t < \infty).   \end{array}.
\]
For $j\in \{0, 1, \cdots, n-1\}$ and $n+j\leq t \leq n+j+1$, we set 
\[
V_{n,t}=  |\cos(\frac{\pi}{2}(t-n-j))|W(\frac{j}{n}) +  |\sin(\frac{\pi}{2}(t-n-j))|W(\frac{j+1}{n}) 
\]
on $u\otimes \delta_z \otimes v$, $(z\neq0)$ and 
\[
V_{n,t}(u\otimes \delta_0 \otimes v)= u\otimes \delta_0 \otimes W_0v.
\]
The isometries $V_{n,t}$ satisfy the following:
\begin{enumerate}
\item The conjugation by $(V_{1,t})_{1\leq t< \infty}$ on $RL_c^\ast(H_{X\times[0,1]})$ induces the same map as $\bD_\ast(f_1)$ on the K-theory groups.
\item The conjugation by $(V_{\infty,t})_{1\leq t< \infty}$ on $RL_c^\ast(H_{X\times[0,1]})$ induces the same map as $\bD_\ast(f_0)(=\mathrm{Id})$ on the K-theory groups.
\item For any $1\leq t< \infty$, $V_{n,t}=V_{\infty,t}=W(0)$ for almost all $n$.
\item For any metric $d_X$ on $X$ and for the standard metric $d_{[0 ,1]}$, define a metric $d_{X\times[0, 1]}=d_X + d_{[0, 1]}$ on $X\times [0,1]$. With respect to $d_{X\times[0 ,1]}$, the propagation $\prop(V_{n+1,t}V_{n,t}^\ast) \to 0$ uniformly in $n$ as $t\to \infty$.
\end{enumerate}

Let $\cA=RL_{c}^\ast(H_{X\times[0, 1]})$ and $\cA^\infty= RL_c^\ast(H_{X\times[0, 1]}\otimes l^2(\mathbb{N}))$ where $H_{X\times[0, 1]}\otimes l^2(\mathbb{N})$ is naturally viewed as an $X\times[0,1]$-module. Let $U_n$ be an isometry from $H_{X\times[0, 1]}$ to $H_{X\times[0, 1]}\otimes l^2(\mathbb{N})$ defined as $v\to v\otimes \delta_n$ for $v$ in $H_{X\times[0, 1]}$. It follows from the property (2) of $W(r)$ and from the definition of $V_{n,t}$ not only that for any $n\in \N_{>0}\cup \{\infty\}$, 
\[
\mathrm{Ad}_{V_{n,t}}\colon \cA \to \cA
\]
is well-defined but also that the diagonal map
\[
\alpha = \sum_{1\leq n < \infty} \mathrm{Ad}_{U_nV_{n,t}}  
\]
induces a $\ast$-homomorphism from $\cA$ to the multiplier algebra $M(\cA^\infty)$. Similarly, both
\[
\beta=  \sum_{1\leq n < \infty} \mathrm{Ad}_{U_nV_{n+1,t}}  , \,\,\, \gamma=   \sum_{1\leq n < \infty} \mathrm{Ad}_{U_nV_{\infty,t}}  
\]
map $\cA$ to $M(\cA^\infty)$. Since for any $1\leq t< +\infty$, $V_{n,t}=V_{\infty,t}$ for almost all $n$, for any $T_t$ in $\cA$, both pairs
\[
(\alpha, \gamma)(T_t) = ( \alpha(T_t), \gamma(T_t)   ), \,\, (\beta, \gamma)(T_t) = ( \beta(T_t), \gamma(T_t)   ), \,\, 
\]
of elements in $M(\cA^\infty)$ define elements in the double
\[
D=M(\cA^\infty)\oplus_{\cA^\infty}M(\cA^\infty) = \{ (a_1, a_2) \in M(\cA^\infty)\oplus M(\cA^\infty) \mid a_1-a_2 \in \cA^\infty  \}.
\]
Note that the pairs $(\alpha, \gamma)(T_t)$, $(\beta, \gamma)(T_t)$ define elements in the subalgebra
\[
C=\{ (a_1, a_2)\in D \mid  a_2= \sum_{1\leq n < \infty} \mathrm{Ad}_{U_nV_{\infty, t}}(T_t), T_t \in \cA  \}
\]
of $D$. The $\ast$-homomorphisms $(\alpha, \gamma)$ and $(\beta, \gamma)$ from $\cA$ to $C$ and hence as maps from $\cA$ to $D$, induce the same map on the K-theory groups. This is because $(\alpha, \gamma)$ and $(\beta, \gamma)$ are conjugate in $C$ by a partial isometry $w=(w_1, w_2)$ in the multiplier algebra $M(C)$ of $C$ where
\[
w_1= \sum_{1\leq n < \infty} U_n V_{n+1,t}V_{n,t}^\ast U_n^\ast,  \,\,w_2= \sum_{1\leq n < \infty} U_n V_{\infty, t}V_{\infty, t}^\ast U_n^\ast.
\]
Indeed, it can be directly checked that
\[
w(\alpha(a), \gamma(a))w^\ast = (\beta(a), \gamma(a)) \,\,\, (\alpha(a), \gamma(a))w^\ast w = (\alpha(a), \gamma(a))
\]
for $a\in \cA$. The fact that $w \in M(C)$ can be checked directly using $w_1, w_2 \in M(\cA^\infty)$ and 
\[
\sum_{1\leq n<\infty} U_n ( V_{n+1,t}V_{n,t}^\ast-V_{\infty,t} V_{\infty,t}^\ast ) T_t U_n^\ast  \in \cA^\infty
\]
for any $T_t$ in $\cA$. This follows from the properties (3) and (4) of $V_{n,t}$. Now we have 
\[
(\alpha, \gamma)_\ast = (\beta, \gamma)_\ast \colon K_\ast(\cA) \to K_\ast(D).
\]
The unilateral shift $U$ on $l^2(\N)$ defines an isometry $(U, U)$ in $M(D)$ and using this, we see that a $\ast$-homomorphism
\[
T_t \mapsto ( \mathrm{Ad}_{U_1V_{1,t}}(T_t),  \mathrm{Ad}_{U_1V_{\infty, t}}(T_t) )
\]
from $\cA$ to $D=M(\cA^\infty)\oplus_{\cA^\infty} M(\cA^\infty)$ is zero on the K-theory groups. From here, it is routine to see that the two maps $\mathrm{Ad}_{V_{1,t}}$, $\mathrm{Ad}_{V_{\infty,t}}$ from $\cA$ to $\cA$ induce the same map on the K-theory groups. Thus, we have $\bD_\ast(f_0)=\bD_\ast(f_1)$.
\end{proof}

\begin{remark} What is shown in \cite[Theorem 3.4]{WY2021} implies that the canonical inclusion induces an isomorphism
\[
K_\ast(RL^*_u(H_X)) \cong K_\ast(RL^*_c(H_X))
\]
for any ample $X$-module $H_X$. Alternatively, we can see this by noting that Theorem \ref{thm_homology} also holds if we use $RL^*_u(H_X)$ in place of $RL^*_c(H_X)$. The listed properties of two functors can be used to deduce the isomorphisms for general $X$ from the case when $X$ is the point. 
\end{remark}


\section{Crossed product of representable localization algebra}

Let $G$ be a second countable, locally compact group (locally compact group in short). We fix a left-Haar measure $\mu_G$ on $G$ and use it to define $L^1(G)$ and $L^2(G)$. We have for $f\in L^1(G)$,
\[
\int_{s\in G} f(s^{-1})\Delta(s)^{-1}d\mu_G(s)= \int_{s\in G} f(s)d\mu_G(s),
\]
\[
\int_{s\in G} f(st)\Delta(t)d\mu_G(s)= \int_{s\in G} f(s)d\mu_G(s),
\]
where $\Delta$ is the modular function.

A $G$-Hilbert space is a separable Hilbert space $H$ equipped with a unitary representation $g\mapsto u_g$ of $G$. We use $g$ for $u_g$ when there is no confusion. The representation is continuous with respect to the strong topology, i.e. $G\times H \ni (g, v)\mapsto gv \in H$ is continuous with respect to the norm topology on $H$.

A $C^*$-algebra $A$ equipped with an action of $G$ by $\ast$-automorphisms is called a $G$-$C^*$-algebra if all elements in $A$ are $G$-continuous. Here, $a\in A$ is $G$-continuous if $G \ni g\mapsto g(a) \in A$ is continuous in norm. A representation of a $G$-$C^*$-algebra is always assumed to be $G$-equivariant.

A proper $G$-space $X$ is a locally compact space $X$ equipped with a continuous action of $G$ by homeomorphisms such that for any compact subset $X_0$ of $X$,  $gX_0 \cap X_0 =\emptyset$ for any $g$ in $G$ outside a compact subset $K$ of $G$ (which depends on $X_0$). This is same as saying that the map $G\times X \ni (g, x) \mapsto (gx, x) \in X\times X$ is proper.

A $G$-$C_0(X)$-algebra is a $G$-$C^*$-algebra $A$ equipped with a non-degenerate, central representation of $G$-$C^*$-algebra $C_0(X)$ to the multiplier algebra $M(A)$ of $A$. A $G$-$C^*$-algebra is proper if it is a $G$-$C_0(X)$-algebra for some proper $G$-space $X$.

Let $X$ be a proper $G$-space. An $X$-$G$-module is a $G$-Hilbert space equipped with a non-degenerate representation of $G$-$C^*$-algebra $C_0(X)$.

Let $H_X$ be an $X$-$G$-module. There is a natural $G$-action on the algebras $RL_{0}^\ast(H_X)$, $RL^\ast_c(H_X)$, $RL^\ast_u(H_X)$, $RL_{c, Q}^\ast(H_X)$, $RL_{u, Q}^\ast(H_X)$ which is not necessarily continuous except on $RL_{0}^\ast(H_X)$. The subalgebra $RL^\ast_c(H_X)_{\Gcont}$ of $RL^\ast_c(H_X)$ consisting of $G$-continuous elements is naturally a $G$-$C^*$-algebra. For discrete $G$, there is no difference between the two. To make our notations clean, {\bf from now on, we will use the notation $RL^\ast_c(H_X)$ to denote the $G$-$C^*$-algebra $RL^\ast_c(H_X)_{\Gcont}$}. This definition implicitly depends on the group $G$ but it should be clear from the context. The same remark applies to the algebras $RL^\ast_u(H_X)$, $RL_{c, Q}^\ast(H_X)$, $RL_{u, Q}^\ast(H_X)$.

We shall mostly study $RL^\ast_c(H_X)$ but the entire discussion and results on $RL^\ast_c(H_X)$ have obvious analogues for $RL_{u}^\ast(H_X)$. 

We have the following  short exact sequence of $G$-$C^*$-algebras,
\begin{equation}\label{eq_seqG}
\xymatrix{
0  \ar[r]^-{}     &   RL_0^\ast(H_X)            \ar[r]^-{}         &  RL_c^\ast(H_X)           \ar[r]^-{} &               RL_{c, Q}^\ast(H_X) \ar[r]^-{}       &           0.
}
\end{equation}

The natural representation of $C_0(X)$ to $M(RL_{c}^\ast(H_X))$ and to $M(RL_{c, Q}^\ast(H_X))$ are $G$-equivariant.

\begin{proposition} The $G$-$C^*$-algebras $RL_{c, Q}^\ast(H_X)$ is naturally a $G$-$C_0(X)$-algebra. That is, the natural representation of the $G$-$C^*$-algebra $C_0(X)$ to the multiplier algebra $M(RL_{c, Q}^\ast(H_X))$ is non-degenerate and central. 
\end{proposition}

Since $RL_{c, Q}^\ast(H_X)$ is a proper $G$-$C^*$-algebra, the reduced crossed product $RL_{c, Q}^\ast(H_X)\rtimes_rG$ and the maximal crossed product $RL_{c, Q}^\ast(H_X)\rtimes_{\rmax}G$ coincide \cite[Theorem 5.3]{AD02}.

\begin{lemma}\label{lem_shortG} The short exact sequence \eqref{eq_seqG} of $G$-$C^*$-algebras descends to the short exact sequence of reduced crossed product algebras,
\[
\xymatrix{
0  \ar[r]^-{}     &   RL_0^\ast(H_X)\rtimes_rG            \ar[r]^-{}         &  RL_c^\ast(H_X)\rtimes_rG           \ar[r]^-{} &               RL_{c, Q}^\ast(H_X)\rtimes_rG \ar[r]^-{}       &           0.
}
\]
\end{lemma}
\begin{proof}
The exactness can be seen from the diagram 
\begin{equation*}
\xymatrix{
RL_c^\ast(H_X)\rtimes_{\rmax}G  /     RL_0^\ast(H_X)\rtimes_{\rmax}G     \ar[d]^-{}       \ar[r]^-{\cong} &               RL_{c, Q}^\ast(H_X)\rtimes_{\rmax}G \ar[d]^-{=}      \\
RL_c^\ast(H_X)\rtimes_rG  /     RL_0^\ast(H_X)\rtimes_rG            \ar[r]^-{} &               RL_{c, Q}^\ast(H_X)\rtimes_rG,
}
\end{equation*}
where the left vertical map and the bottom horizontal map are surjective from which their injectivity follows.
\end{proof}

\begin{proposition}\label{prop_quotient_isomG1} The quotient map from $RL_c^\ast(H_X)\rtimes_rG$ to $RL_{c, Q}^\ast(H_X)\rtimes_rG$ induces an isomorphism on the K-theory groups. The same holds for the maximal crossed product. 
\end{proposition}
\begin{proof} Note that
\[
RL_0^\ast(H_X)\rtimes_rG = C_0([1, \infty), \Compacts(H_X))\rtimes_rG \cong C_0([1, \infty), \Compacts(H_X))\otimes C^*_r(G)
\]
and 
\[
RL_0^\ast(H_X)\rtimes_{\rmax}G = C_0([1, \infty), \Compacts(H_X))\rtimes_{\rmax}G \cong C_0([1, \infty), \Compacts(H_X))\otimes_{\rmax} C^*_{\rmax}(G)
\]
since the $G$-action on $RL_0^\ast(H_X)=C_0([1, \infty), \Compacts(H_X))$ is inner. We see that both $RL_0^\ast(H_X)\rtimes_rG$ and $RL_0^\ast(H_X)\rtimes_{\rmax}G$ have zero K-theory groups. The claim follows from this and Lemma \ref{lem_shortG}.

\end{proof}

\begin{proposition}\label{prop_quotient_isomG} The quotient map from $RL_c^\ast(H_X)\rtimes_{\rmax}G$ to $RL_{c}^\ast(H_X)\rtimes_rG$ induces an isomorphism on the K-theory groups. 
\end{proposition}
\begin{proof}
This follows from the following commutative diagram 
\begin{equation*}
\xymatrix{
   K_\ast( RL_c^\ast(H_X)\rtimes_{\rmax}G )\ar[r]^-{\cong} \ar[d]^-{} &   K_\ast( RL_{c, Q}^\ast(H_X)\rtimes_{\rmax}G  ) \ar[d]^-{=}    \\
     K_\ast(  RL_c^\ast(H_X) \rtimes_rG )         \ar[r]^-{\cong} &          K_\ast(     RL_{c, Q}^\ast(H_X)\rtimes_rG ) .
}
\end{equation*}
\end{proof}

Let $H$ be an open subgroup of $G$. Let $Y$ be a proper $H$-space and consider the balanced product
\[
G\times _H Y = ( G\times Y ) / H
\]
where the right $H$-action on $G\times Y$ is given by $(g, y)\mapsto (gh, h^{-1}y)$. The balanced product $G\times _H Y $ is a proper $G$-space by the left-translation. The $H$-space $Y$ is naturally an $H$-invariant open subset of $G\times _H Y$. 

\begin{proposition}\label{prop_compactopen} Let $H$ be an open subgroup of $G$, $Y$ be a proper $H$-space and $X$ be the balanced product $G\times _H Y$. For any $X$-$G$-module $H_X$, let $H_Y=\chi_YH_X$ which we naturally view as an $H$-$Y$-module. Then, the natural inclusion
\[
RL^*_c(H_Y)\rtimes_rH \to RL^*_c(H_X)\rtimes_rG
\]
induces an isomorphism on K-theory groups.
\end{proposition} 
\begin{proof} Note that the natural inclusion $RL^*_c(H_Y)  \to RL^*_c(H_X)$ is well-defined since $H$ is open so $H$-continuous elements are automatically $G$-continuous. In view of Proposition \ref{prop_quotient_isomG1}, it is enough to show that the natural inclusion
\[
RL^*_{c, Q}(H_Y)\rtimes_rH  \to RL^*_{c, Q}(H_X)\rtimes_rG
\]
induces an isomorphism on K-theory groups. Note that $RL^*_{c, Q}(H_X)$ is a $G$-$C_0(G/H)$-algebra whose fiber at the coset $H$ is naturally the $H$-$C^*$-algebra $RL^*_{c, Q}(H_Y)$. Note that since $G/H$ is discrete, the fiber is not only a quotient but also a subalgebra. In general, for any $G$-$C_0(G/ H)$-algebra $A$, with fiber $A_0$ at $H$, the natural inclusion
\[
A_0\rtimes_rH \to A\rtimes_rG
\]
induces an isomorphism on K-theory since we have an isomorphism $A\rtimes_rG\cong (A_0\rtimes_rH)\otimes \Compacts(l^2(G/H))$ and the inclusion corresponds to the corner embedding of $A_0\rtimes_rH$ into $(A_0\rtimes_rH)\otimes \Compacts(l^2(G/H))$. This is just a special case of a more general Morita-equivalence between $(\mathrm{Ind}_H^GA_0)\rtimes_rG$ and $A_0\rtimes_rH$ which holds for any closed subgroup $H$ of $G$ and for any $H$-$C^*$-algebra $A_0$ \cite[Theorem 17]{Green78}.
\end{proof}


\section{Universal $X$-$G$-module}

For discrete $G$, there is a notion of an ample $X$-$G$-module. An $X$-$G$-module $H_X$ is ample (as an $X$-$G$-module) if it is ample as an $X$-module and if it is locally free in a sense that for any finite subgroup $F$ of $G$ and for any $F$-invariant Borel subset $E$ of $X$, $\chi_EH_X$ is $F$-equivariantly isomorphic to $l^2(F)\otimes H_0$ for some Hilbert space $H_0$ where $l^2(F)$ is equipped with the left-regular representation of $F$. For a locally compact group $G$, the exact analogue would not give us a good definition. For example, for $G=\R$ and $X=\R$, an $X$-$G$-module $L^2(\R)$ would be ample in such a definition, although it is rather the smallest $X$-$G$-module. In this article, we refrain from defining an ample $X$-$G$-module for locally compact $G$ but we will define some substitute, a universal $X$-$G$-module (see below). We also refer the reader to \cite[Section 2.1]{GHM} for a supplement to this section.

Let $X, Y$ be proper $G$-spaces and $H_X, H_Y$ be an $X$-$G$-module and a $Y$-$G$-module respectively.

For a $G$-equivariant continuous map $f\colon X\to Y$, a family of isometries $(V_t\colon H_X \to H_Y)_{t\in[1, \infty)}$ is called an equivariant continuous cover of $f$ (\cite[Definition 4.5.11]{WY2020}) if it is a continuous cover of $f$ (ignoring the $G$-actions) and if $V_t$ is $G$-equivariant for all $t\geq1$. For discrete $G$, such an equivariant continuous cover of $f$ exists whenever $H_Y$ is an ample $Y$-$G$-module \cite[Proposition 4.5.12]{WY2020}.  

\begin{lemma}\label{lem_coveridG} Let $f\colon X\to Y$ be a $G$-equivariant continuous map of proper $G$-spaces. A family of isometries $(V_t\colon H_X \to H_Y)_{t\in[1, \infty)}$ is an equivariant continuous cover of $f$ if and only if it is an equivariant continuous cover of the identity on $Y$ by regarding $H_X$ as a $Y$-$G$-module $(H_X)_Y$ via $f^*\colon C_0(Y)\to C_b(X)$. 
\end{lemma}
\begin{proof} The proof of Lemma \ref{lem_coverid} works verbatim.
\end{proof}

\begin{definition}\label{def_universalXG} We say that an $X$-$G$-module $H_X$ is a universal $X$-$G$-module if there is an equivariant continuous cover $(V_t\colon H^0_X \to H_X)_{t\in[1, \infty)}$ of the identity map on $X$ for any $X$-$G$-module $H^0_X$.
\end{definition}

Hence, for discrete $G$, any ample $X$-$G$-module is universal. 

\begin{lemma} Let $H_Y$ be a universal $Y$-$G$-module. Then, for any $G$-equivariant continuous map $f\colon X\to Y$ and for any $X$-$G$-module $H_X$, there is an equivariant continuous cover $(V_t\colon H_X \to H_Y)_{t\in[1, \infty)}$ of $f$.
\end{lemma}
\begin{proof} This follows from Lemma \ref{lem_coveridG} and from the definition of a universal $Y$-$G$-module.
\end{proof}

We will show that a universal $X$-$G$-module exists for any locally compact group $G$ and for any proper $G$-space $X$ but we first do some preparation.

\begin{definition}\label{def_amplify} Let $X$ be a proper $G$-space. Let $H_X$ be an $X$-module. We define an $X$-$G$-module 
\[
\bar H_X=H_X\otimes L^2(G) \cong L^2(G, H_X)
\]
where a representation of $G$ on $\bar H_X$ is given by the left-regular representation on $L^2(G)$,
\[
(gf)(h)=f(g^{-1}h)
\]
for $g\in G$ and $(f\colon G\ni h\mapsto f(h) \in H_X)\in L^2(G, H_X)$ and where the representation of $C_0(X)$ on $\bar H_X$ is given by
\[
(\phi f)(h)=h^{-1}(\phi) f(h)
\]
for $\phi \in C_0(X)$. 
\end{definition}

\begin{definition} We say that an $X$-$G$-module $H_X$ is very ample (as an $X$-$G$-module) if it is, up to isomorphisms, of the form $\bar H_X^0$ for some ample $X$-module $H_X^0$.
\end{definition}

\begin{lemma} Let $H_X$ be a very ample $X$-$G$-module. Let $Y$ be any $G$-invariant closed subset of $X$ which is the closure of an open subset of $X$. Then, $\chi_YH_X$ is a very ample $Y$-$G$-module. Similarly, if $U$ is any $G$-invariant open subset of $X$, then $\chi_UH_X$ is a very ample $U$-$G$-module. 
\end{lemma}
\begin{proof} If $H_X=\bar H_X^0=H^0_X\otimes L^2(G)$ for an ample $X$-module $H^0_X$, it is easy to see that $\chi_YH_X=\bar H^0_Y$ where $H^0_Y=\chi_YH^0_X$. Since $Y$ is the closure of an open subset of $X$, $H^0_Y$ is an ample $Y$-module. The claim now follows. The case of a $G$-invariant open subset is proved in the same way.
\end{proof}

Recall that for a proper $G$-space $X$, $c\in C_b(X)$ is a cut-off function on $X$ if $c\geq0$, $\int_Gg(c)^2d\mu_G(g)=1$, and for any compact subset $X_0$ of $X$, $\supp(c)\cap gX_0=\emptyset$ for $g$ outside a compact subset of $G$. When $X$ is $G$-compact, $c$ is (automatically) compactly supported. We fix any cut-off function $c$ on $X$. For any $X$-$G$-module $H_X$, an $X$-$G$-module $\bar H_X$ is defined as in Definition \ref{def_amplify} by regarding $H_X$ as an $X$-module. We define an isometry 
\[
V_c\colon H_X\to \bar H_X=H_X\otimes L^2(G)
\]
by sending $v\in H_X$ to 
\begin{equation}\label{eq_V_c}
V_c(v) (h) = cu_h^{-1}v\,\,\, (\text{$h\in G$}),
\end{equation}
where $u_h$ is the unitary operator on $H_X$ corresponding to $h\in G$.

\begin{lemma} The isometry $V_c\colon H_X \to \bar H_X$ is $G$-equivariant and intertwines the representations of $C_0(X)$. That is, $V_c$ is a strict cover of the identity map on $X$.
 \end{lemma}
 \begin{proof} This can be checked directly.
 \end{proof}

\begin{lemma}\label{lem_universal} Let $H_X$ be an ample $X$-module. Let $\bar H_X$ be the $X$-$G$-module as in Definition \ref{def_amplify}. Then, for any $X$-$G$-module $H^0_X$, and for any open neighborhood $U$ of the diagonal in $X^+\times X^+$, there is a $G$-equivariant isometry $V$ from $H^0_X$ to $\bar H_X$ such that $\supp(V) \subset U$.
\end{lemma} 
\begin{proof} Fix any metric $d$ on $X^+$. It is enough to find, for any compact subset $X_0$ of $X$ and $\epsilon>0$, a $G$-equivariant isometry $V$ from $H^0_X$ to $\bar H_X$ such that the support $\supp(V)$ is contained in 
\[
\{ (x, y) \in X_0\times X \mid d(x, y)<\epsilon   \} \cup \{ (x, y) \in X\times X_0 \mid d(x, y)<\epsilon   \}   \cup (X-X_0) \times (X-X_0).
\] 
We first explain how we can assume $X$ is $G$-compact for this. We let $U$ be any relatively compact, open subset of $X$ containing $X_0$ and let $Y$ be the closure of $GU$ which is a $G$-compact closed subspace of $X$. Suppose we get a $G$-equivariant isometry $V$ from $H^0_Y=\chi_YH^0_X$ to $\bar H_Y =\chi_Y\bar H_X$ (where $H_Y=\chi_YH_X$), then we obtain a desired isometry $V$ from $H^0_X$ to $\bar H_X$ by adding $V$ to any $G$-equivariant isometry from the complement $H^0_{X-Y}=\chi_{X-Y}H^0_X$ to $\chi_{X-Y}\bar H_{X}=\bar H_{X-Y}$ (where $H_{X-Y}=\chi_{X-Y}H_X$). Such an (arbitrary) $G$-equivariant isometry from $H^0_{X-Y}$ to $\bar H_{X-Y}$ exists, for example we can take the composition of $V_c\colon H^{0}_{X-Y} \to \bar H^0_{X-Y}$ and the amplification
\[
W\otimes1 \colon H^0_{X-Y}\otimes L^2(G) \to H_{X-Y}\otimes L^2(G)
\]
of any isometry $W\colon  H^0_{X-Y} \to H_{X-Y}$. 

Now, we assume $X$ is $G$-compact. Let $\epsilon>0$, $X_0$ be a compact subset of $X$ and $X_1$ be any compact neighborhood of $X_0$. Let $A$ be the support of $c$ which is compact (here we used $G$-compactness) and $B$ be any compact neighborhood of $A$ in $X$. We let $K$ be a compact subset of $G$ such that $X_1\cap gB=\emptyset$ for $g\in G\backslash K$. Let $X_2$ be any compact neighborhood of $KB$. Note that we have $X_2\supset KB \supset X_1 \supset X_0$. Since
\[
C= K^{-1} \{(x,y)\in X_2\times X_2 \mid d(x, y)\geq \epsilon/2 \}  
\]
\[
= \{(k^{-1}x,k^{-1}y)\in X \times X \mid k\in K, (x, y) \in X_2\times X_2, d(x, y)\geq \epsilon/2 \}
\]
is a compact subset of $K^{-1}X_2\times K^{-1}X_2$ which does not contain the diagonal, there is $\delta>0$ such that any $(x, y)\in C$ satisfies $d(x, y)\geq \delta$.

We let $V_0\colon H_X^0 \to H_X$ be an isometry so that the following is satisfied (here we used the ampleness of $H_X$):
\begin{enumerate}
\item $V_0\chi_A = \chi_BV_0\chi_A$.
\item $\prop(V_0)<\delta$.
\end{enumerate}

 Let $V_c\colon H^0_X \to \bar H^0_X$ be as in \eqref{eq_V_c} and let $\bar V_0\colon \bar H^0_X \to \bar H_X$ be given by
\[
\bar V_0=V_0\otimes 1\colon  H^0_X\otimes L^2(G)\to H_X\otimes L^2(G).
\] 
We show that a $G$-equivariant isometry $V=\bar V_0V_c\colon H_X^0 \to \bar H_X$ satisfies the required support condition. More explicitly, the isometry $V$ sends $v\in H_X^0$ to  
\[
(V(v)\colon h \mapsto V(v)(h)= V_0cu_h^{-1}v\in H_X)  \in L^2(G, H_X)
\]
where $u_h$ is the unitary on $H^0_X$ corresponding to $h\in G$. We see that for any Borel functions $\phi_1, \phi_2$ on $X$, $\phi_1 V \phi_2=0$ if 
\[
h^{-1}(\phi_1)V_0ch^{-1}(\phi_2)=0  \,\,\, \text{in $\Linears(H_X^0, H_X)$ for all $h\in G$}.
\]
We have
\[
h^{-1}(\phi_1)V_0ch^{-1}(\phi_2) = h^{-1}(\phi_1)V_0\chi_Ach^{-1}(\phi_2) = h^{-1}(\phi_1)\chi_BV_0\chi_Ach^{-1}(\phi_2),
\]
so if either of $\phi_1$, $\phi_2$ has support contained in $X_1$, this is zero unless $h\in K$. Moreover, if $\supp(\phi_1)\subset X_1$ (resp. $\supp(\phi_2)\subset X_1$), then this is zero for all $h\in G$ if $\supp(\phi_2)\cap KB=\emptyset$ (resp. $\supp(\phi_1)\cap KB=\emptyset$). In particular, $X_0\times (X-KB)$ and $(X-KB)\times X_0$ are disjoint from $\supp(V)$.

Now, let $(x, y)\in X_0\times X$ such that $d(x, y)\geq \epsilon$. We show that $(x, y) \notin \supp(V)$. The case when $y\in (X-KB)$ is already proved so we assume $y\in KB$. Then, there are open sets $U_x\ni x$ in $X_1$ and $U_y\ni y$ in $X_2$ such that the distance $d(U_x, U_y)\geq \epsilon/2$. It follows from the definition of $\delta>0$ that $d(k^{-1}U_x, k^{-1}U_y)\geq \delta$ for any $k\in K$. It follows $\chi_{U_x}V\chi_{U_y}=0$ i.e.
\[
\chi_{h^{-1}U_x}V_0c\chi_{h^{-1}U_y}=0  \,\,\, \text{for all $h\in G$}.
\]
For $h \in G\backslash K$, this is because $U_x\subset X_1$, and for $h\in K$, this is because $\prop(V_0)<\delta$. Similarly, we can show that $(x, y) \notin \supp(V)$ for any $(x, y)\in X\times X_0$ such that $d(x, y)\geq \epsilon$. 
\end{proof}

\begin{proposition}\label{prop_universal} For any second countable, locally compact group $G$ and for any proper $G$-space $X$, a universal $X$-$G$ module exists. In fact, for any ample $X$-$G$-module $H_X$, the $X$-$G$-module $\bar H_X$  (as in Definition \ref{def_amplify}) is a universal $X$-$G$ module, i.e. any very ample $X$-$G$-module is universal.
\end{proposition}
\begin{proof}  Let $H_X$ be any ample $X$-module and let $\bar H_X$ be the $X$-$G$-module as in Definition \ref{def_amplify}. For any $X$-$G$-module $H^0_X$, it is enough to find a uniformly continuous family $(V_t\colon H^0_X\to \bar H_X)_{t\in[1,\infty)}$ of $G$-equivariant isometries such that $\prop(V_t)\to 0$ as $t \to \infty$ for a fixed metric $d$ on $X^+$. Lemma \ref{lem_universal} gives us for any $n>0$, a $G$-equivariant isometry $V_n\colon H^0_X\to \bar H_X$ with $\prop(V_n)<1/n$. It is easy to see from the construction (using the ampleness of $H_X$) that we can arrange $V_n$ and $V_{n+1}$ to have mutually orthogonal ranges for any $n>0$. For $t\in (n, n+1)$, we let
\[
V_t= \sqrt{n+1-t} V_n + \sqrt{t-n} V_{n+1}.
\] 
The family $(V_t\colon H^0_X\to \bar H_X)_{t\in[1,\infty)}$ is an equivariant continuous cover of the identity.

\end{proof}


\section{Representable equivariant K-homology}

Let $X, Y$ be proper $G$-spaces and $H_X, H_Y$ be an $X$-$G$-module and a $Y$-$G$-module respectively. 
Given a $G$-equivariant continuous cover $(V_t\colon H_X \to H_Y)_{t\in[1, \infty)}$ of a $G$-equivariant continuous map $f\colon X\to Y$, the conjugation by $V_t$ defines a $G$-equivariant $\ast$-homomorphism
\[
\mathrm{Ad}_{V_t} \colon RL_{c}^\ast(H_X) \to RL_{c}^\ast(H_Y).
\]
Thus, it defines a $\ast$-homomorphism
\[
\mathrm{Ad}_{V_t} \colon RL_{c}^\ast(H_X)\rtimes_rG \to RL_{c}^\ast(H_Y)\rtimes_rG.
\]
This $\ast$-homomorphism depends on the cover $V_t$ of $f$ but the induced map on their K-theory groups is independent of the choice of $V_t$. This is because given two continuous covers $(V_{i, t}\colon H_X\to H_Y)_{t\in [1,\infty)}$ $(i=1,2)$ of $f$, the two maps
\[
a \mapsto \begin{bmatrix} \mathrm{Ad}_{V_{1,t}}(a) & 0  \\ 0 & 0  \end{bmatrix}, \,\,\,
a\mapsto \begin{bmatrix} 0 & 0  \\ 0 & \mathrm{Ad}_{V_{2,t}}(a)  \end{bmatrix}
\]
from $RL_{c}^\ast(H_X)\rtimes_rG$ to the matrix algebra $M_2(RL_{c}^\ast(H_Y)\rtimes_rG)$ are conjugate to each other by the unitary 
\[
 \begin{bmatrix} 1-V_{1,t}V_{1,t}^\ast & V_{1,t}V_{2,t}^\ast  \\ V_{2,t}V_{1,t}^\ast & 1-V_{2,t}V_{2,t}^\ast  \end{bmatrix}
\]
in the $2\times2$-matrix algebra $M_2( M(RL_{c}^\ast(H_Y))) \subset M_2( M(RL_{c}^\ast(H_Y)\rtimes_rG))$ of the multiplier algebra.

Note $\mathrm{Ad}_{V_t}$ induces a $\ast$-homomorphism on the quotients  
\[
\mathrm{Ad}_{V_t} \colon RL_{c, Q}^\ast(H_X)\rtimes_rG \to RL_{c, Q}^\ast(H_Y)\rtimes_rG
\]
as it maps the ideal $RL_{0}^\ast(H_X)\rtimes_rG$ into $RL_{0}^\ast(H_Y)\rtimes_rG$ and the induced map on their K-theory groups is again independent of the choice of a cover $V_t$ of $f$.

\begin{definition} Choose any universal $X$-$G$-module $H_X$ for each proper $G$-space $X$ and any $G$-equivariant continuous cover $(V^f_t\colon H_X\to H_Y)_{t\in [1,\infty)}$ for each $G$-equivariant continuous map $f\colon X\to Y$. A functor $\mathbb{D}^G_\ast$ from the category $\mathcal{PR}^G$ of (second countable, locally compact) proper $G$-spaces to the category $\mathcal{GA}$ of graded abelian groups is defined as
\[
\mathbb{D}^G_\ast(X)=K_\ast(RL^\ast_c(H_X)\rtimes_rG),
\]
\[
\mathbb{D}^G_\ast(f\colon X\to Y)=\mathrm{Ad}_{V^f_t \ast}\colon K_\ast(RL^\ast_c(H_X)\rtimes_rG) \to K_\ast(RL^\ast_c(H_Y)\rtimes_rG).
\]
\end{definition}

\begin{proposition}\label{prop_welldefG} The functor  $\mathbb{D}^G_\ast$ from  $\mathcal{PR}^G$ to $\mathcal{GA}$ is well-defined. The functor does not depend on the choice of universal $X$-$G$-modules $H_X$ up to canonical equivalence.
\end{proposition}
\begin{proof} As in the non-equivariant case (Proposition \ref{prop_welldef}), there is a technical point concerning the functoriality. This can be handled in the same manner as in the proof of Proposition \ref{prop_welldef}. A-priori, we only have the functoriality 
\[ 
\mathbb{D}^G_\ast(f_2) \circ  \mathbb{D}^G_\ast(f_1)   = \mathbb{D}^G_\ast(f_2\circ f_1)
\]
for $f_1\colon X\to Y$, $f_2\colon Y\to Z$ when $f_2$ is proper. Note that any $G$-equivariant continuous map from a $G$-compact proper $G$-space $X$ to a proper $G$-space $Y$ is proper. We reduce the functoriality to the case of proper maps. Thus, we first show the representability of $\mathbb{D}^G_\ast(X)$: if $(X_i)_{i \in I}$ is the net of $G$-compact, $G$-invariant closed subsets of $X$, ordered by the inclusion, the canonical maps $\bD_\ast(X_i) \to \bD_\ast(X)$ induce a natural isomorphism 
\begin{equation} \label{eq_representableG}
\varinjlim_{i \in I}\bD^G_\ast(X_i) \cong \bD^G_\ast(X).
\end{equation}
Note that the functoriality for proper maps is used for defining this inductive system. It suffices to show that if $U_n$ is an increasing sequence of relatively $G$-compact $G$-invariant open subsets of $X$ such that $\cup_n U_n=X$, then for their closures $X_n=\bar U_n$, the natural inclusions induce an isomorphism 
\[
\varinjlim_nK_\ast(RL_c^\ast(H_{X_n})\rtimes_rG)\cong K_\ast(RL_c^\ast(H_{X})\rtimes_rG).
\]
We may assume $H_X$ is very ample. Since $X_n=\bar U_n$, the subspace $\chi_{X_n}H_X$ is a very ample $X_n$-$G$-module and we may assume $H_{X_n}=\chi_{X_n}H_X$. Then, we see that $RL_c^\ast(H_{X_n})=\chi_{X_n}RL_c^\ast(H_{X})\chi_{X_n}$ and $RL_c^\ast(H_{X_n})$ is an increasing sequence of $G$-$C^*$-subalgebras of $RL_c^\ast(H_{X})$ whose union is dense in $RL_c^\ast(H_{X})$. The claim follows from the continuity of K-theory.

Note that this identification \eqref{eq_representableG} is compatible with $\mathbb{D}^G_\ast(f\colon X\to Y)$: if $(X_i)_{i \in I}$, $(Y_j)_{j\in J}$ are the nets of $G$-compact, $G$-invariant closed subsets of $X$ and $Y$ respectively, and if we consider the map
\[
\varinjlim_{i\in I} \bD^G_\ast(f\mid_{X_i})\colon  \varinjlim_{i \in I} \bD^G_\ast(X_i)   \to  \varinjlim_{j \in J}\bD^G_\ast(Y_j) 
\]
defined as the limit of 
\[
\bD^G_\ast(X_i)  \to \bD^G_\ast(f(X_i))    \to  \varinjlim_{j \in J}\bD^G_\ast(Y_j) 
\]
which is the composition of $\bD^G_\ast(f\mid_{X_i}\colon X_i\to f(X_i))$ and the natural map (note $f(X_i)$ is a $G$-compact, $G$-invariant closed subset of $Y$), the following diagram commutes
\begin{equation*}
\xymatrix{
\varinjlim_{i \in I}\bD^G_\ast(X_i) \ar[r]^-{\varinjlim_{i\in I} \bD^G_\ast(f\mid_{X_i})} \ar[d]^-{\cong} &     \varinjlim_{j \in J}\bD^G_\ast(Y_j) \ar[d]^-{\cong}  \\
K_\ast(RL_c^\ast(H_{X}))           \ar[r]^-{\mathbb{D}^G_\ast(f)}         &  K_\ast(RL_c^\ast(H_{Y})). 
}
\end{equation*}
This follows from the independence for $\mathbb{D}^G_\ast(f)$ of the choice of an equivariant continuous cover $V_t\colon H_X\to H_Y$ and from that for any $G$-compact $G$-invariant closed subset $X_0$ of $X$, there is a $G$-compact $G$-invariant closed subset $Y_0$ of $Y$ such that $V_t\chi_{X_0}=\chi_{Y_0}V_t\chi_{X_0}$ for large enough $t$.

The functoriality $\bD^G_\ast(f_2)\circ \bD^G_\ast(f_1) = \bD^G_\ast(f_2\circ f_1)$ for not necessarily proper maps $f_1\colon X\to Y$, $f_2\colon Y\to Z$ now follows.
\end{proof}

\begin{remark} Since $\mathbb{D}^G_\ast$ is representable as we explained in the proof of Proposition \ref{prop_welldefG}, we may extend this functor to any (not necessarily locally compact) proper $G$-spaces $X$ by defining 
\[
\mathbb{D}^G_\ast(X) = \varinjlim_{Y\subset X, \mathrm{Gcpt}} \mathbb{D}^G_\ast(Y)
\]
where the limit is over all $G$-compact, $G$-invariant subspaces $Y$ of $X$.
\end{remark}

Note that we have naturally
\[
\mathbb{D}^G_\ast(X)=K_\ast(RL^\ast_c(H_X)\rtimes_rG) \cong K_\ast(RL^\ast_{c, Q}(H_X)\rtimes_rG)
\]
by Proposition \ref{prop_quotient_isomG1} and 
\[
\mathbb{D}^G_\ast(X) = K_\ast(RL^*_c(H_X)\rtimes_rG) \cong K_\ast(RL^*_c(H_X)\rtimes_{\rmax}G)
\]
by Proposition \ref{prop_quotient_isomG}.

When $G$ is trivial, $\mathbb{D}^G_\ast$ is nothing but the representable K-homology $\mathbb{D}_\ast$.

\begin{theorem}\label{thm_Ghomology} The functor $\mathbb{D}^G_\ast$ satisfies the following: 
\begin{enumerate}
\item  $\mathbb{D}^G_\ast(\mathrm{empty\, set})\cong 0$.
\item Induction from an open subgroup: for any open subgroup $H$ of $G$ and any proper $H$-space $Y$, consider the balanced product $G\times_HY$ which is a proper $G$-space. We have a natural isomorphism $\mathbb{D}^G_\ast(G\times_H Y)\cong \mathbb{D}_\ast^H(Y)$. If $H$ is a compact open subgroup of $G$, we have $\mathbb{D}^G_\ast(G/H\times \Delta^k)\cong K_\ast(C^*_r(H))$ for any k-simplex (ball) $\Delta^k$ with the trivial $G$-action.

\item Representable: if $(X_i)_{i \in I}$ is the net of $G$-compact, $G$-invariant closed subsets of $X$, ordered by the inclusion, the canonical maps $\mathbb{D}^G_\ast(X_i) \to \mathbb{D}^G_\ast(X)$ induce a natural isomorphism 
\[
\varinjlim_{i \in I}\mathbb{D}^G_\ast(X_i) \cong \mathbb{D}^G_\ast(X).
\]
\item Mayer--Vietoris sequence for a $G$-invariant open cover: if $X=U\cup V$ for $G$-invariant open subsets $U$ and $V$ of $X$, we have a natural Meyer--Vietoris sequence
\[
\xymatrix{
 \mathbb{D}^G_0(U\cap V) \ar[r]^-{}&    \mathbb{D}^G_0(U) \oplus \mathbb{D}^G_0(V) \ar[r]^-{} &   \mathbb{D}^G_0(X) \ar[d]^-{}   \\
\mathbb{D}^G_1(X)     \ar[u]^-{}         &   \mathbb{D}^G_1(U) \oplus \mathbb{D}^G_1(V)        \ar[l]^-{} &            \mathbb{D}^G_1(U\cap V) \ar[l]^-{}.
}
\]
\item  Homotopy invariance: if $h\colon X\times[0 ,1]\to Y$ is a $G$-equivariant continuous homotopy between $f_0, f_1\colon X\to Y$, then $\mathbb{D}^G_\ast(f_0)=\mathbb{D}^G_\ast(f_1)$.
\end{enumerate}
\end{theorem}
\begin{proof} We prove (1)-(5)  following the argument in the non-equivariant case (see the proof of Theorem \ref{thm_homology}). We would like to thank Rufus Willett for telling the author that the non-equivariant proof of (5), the homotopy invariance, should extend to this equivariant setting.

(1) If $X$ is the empty set, $H_X=0$ and $RL_c^\ast(H_X)\rtimes_rG=0$.

(2) The first claim: Let $X=G\times_HY$. Here, $Y$ is a $H$-invariant open subset of $X$. We may assume that the chosen $X$-$G$-module $H_X$ is very ample and that the chosen $Y$-$H$-module $H_Y$ is $\chi_YH_X$ as this is a very ample $Y$-$H$-module. By Proposition \ref{prop_compactopen}, the inclusion
\[
RL^*_c(H_Y)\rtimes_rH \to RL^*_c(H_X)\rtimes_rG 
\]
induces an isomorphism on K-theory groups. This gives
\[
 \mathbb{D}_\ast^H(Y) \cong \mathbb{D}^G_\ast(G\times_H Y).
\]
Naturality (with respect to $Y$) can be checked directly. The second claim will be proved after we prove the homotopy invariance (5).

(3) We have already proved this in the proof of Proposition \ref{prop_welldefG}.

(4) As in the proof of the non-equivariant case, consider $G$-invariant ideals $RL_{c, Q}^\ast(H_X)_U$, $RL_{c, Q}^\ast(H_X)_V$, $RL_{c, Q}^\ast(H_X)_{U\cap V}$ of $RL_{c, Q}^\ast(H_X)$. These are naturally identified as $RL_{c, Q}^\ast(H_U)$, $RL_{c, Q}^\ast(H_V)$, $RL_{c, Q}^\ast(H_{U\cap V})$ respectively. Note that $RL_{c, Q}^\ast(H_X)\rtimes_rG$ is naturally a $C_0(X/G)$-algebra and let 
\[
(RL_{c, Q}^\ast(H_X)\rtimes_rG)_{U/G}, \,\, (RL_{c, Q}^\ast(H_X)\rtimes_rG)_{V/G},\,\, (RL_{c, Q}^\ast(H_X)\rtimes_rG)_{(U\cap V)/G}
\]
be its ideals corresponding to open subsets $U/G$, $V/G$ and $(U\cap V)/G$ of $X/G$. It is not hard to see that
\[
RL_{c, Q}^\ast(H_X)_U \rtimes_rG= (RL_{c, Q}^\ast(H_X)\rtimes_rG)_{U/G}
\]
inside $RL_{c, Q}^\ast(H_X) \rtimes_rG$ and similarly for other two. These ideals of $RL_{c, Q}^\ast(H_{X})\rtimes_rG$ satisfy
\[
(RL_{c, Q}^\ast(H_X)\rtimes_rG)_{U/G} +  (RL_{c, Q}^\ast(H_X)\rtimes_rG)_{V/G} = (RL_{c, Q}^\ast(H_X)\rtimes_rG),
\]
\[
(RL_{c, Q}^\ast(H_X)\rtimes_rG)_{U/G} \cap (RL_{c, Q}^\ast(H_X)\rtimes_rG)_{V/G} = (RL_{c, Q}^\ast(H_X)\rtimes_rG)_{(U\cap V)/G}.
\]
The corresponding Mayer--Vietoris sequence is the desired one. The naturality can be checked directly.

(5) For any proper $G$-space $X$ and for $ r \in[0 ,1]$, let 
\[
f_r\colon X\times [0, 1] \to X\times [0 ,1], \,\, (x, t)\to (x, (1-r)t).
\]
As in the proof of the non-equivariant case, it suffices to show the following claim:
\[
\bD^G_\ast(f_1) = \bD^G_\ast(f_0) (=\mathrm{Id}).
\]
We reuse the notations from the proof of the non-equivariant case except that now $H_X$ is a universal $X$-$G$-module. Note for example, the isometries $W_z$ for $z\in Z$ and $V_{n,t}$ are all $G$-equivariant and all the properties proved in the non-equivariant case are still valid in this context.
 
 As in the non-equivariant case, we let $\cA=RL_{c}^\ast(H_{X\times[0, 1]})$ and $\cA^\infty= RL_c^\ast(H_{X\times[0, 1]}\otimes l^2(\mathbb{N}))$. As before, we have $G$-equivariant $\ast$-homomorphisms
\[
\mathrm{Ad}_{V_{n,t}}\colon \cA \to \cA
\]
for $n\in \mathbb{N}_{>0}\cup\{\infty\}$ and 
\[
\alpha = \sum_{1\leq n < \infty} \mathrm{Ad}_{U_nV_{n,t}},\,\, \beta=  \sum_{1\leq n < \infty} \mathrm{Ad}_{U_nV_{n+1,t}}, \,\,\, \gamma=   \sum_{1\leq n < \infty} \mathrm{Ad}_{U_nV_{\infty,t}}
\]
 from $\cA$ to the multiplier algebra $M(\cA^\infty)$.
 
 We obtain $\ast$-homomorphisms
 \[
 \alpha \rtimes_r1, \beta\rtimes_r1, \gamma\rtimes_r1\colon \cA\rtimes_rG \to M(\cA^\infty)_\Gcont\rtimes_rG \to M(\cA^\infty\rtimes_rG)
 \]
where $M(\cA^\infty)_\Gcont$ consists of $G$-continuous elements of $M(\cA^\infty)$.

Each of the following pairs of $\ast$-homomorphisms
\[
( \alpha\rtimes_r1, \gamma\rtimes_r1), \,\,\, ( \beta\rtimes_r1, \gamma\rtimes_r1)
\]
defines a $\ast$-homomorphism from $\cA\rtimes_rG$ to the double
\[
D'=M(\cA^\infty\rtimes_rG)\oplus_{\cA^\infty\rtimes_rG}M(\cA^\infty\rtimes_rG).
\]
The two $\ast$-homomorphisms naturally factor through $C_{\mathrm{Gcont}}\rtimes_rG$ and $D_{\mathrm{Gcont}}\rtimes_rG$ as
\[
\cA\rtimes_rG \to C_{\mathrm{Gcont}}\rtimes_rG \to D_{\mathrm{Gcont}}\rtimes_rG \to D'
\]
where the first map is $(\alpha, \gamma)\rtimes_r1$ or $(\beta, \gamma)\rtimes_r1$. Here, we recall that
\[
D=M(\cA^\infty)\oplus_{\cA^\infty}M(\cA^\infty) = \{ (a_1, a_2) \in M(\cA^\infty)\oplus M(\cA^\infty) \mid a_1-a_2 \in \cA^\infty  \},
\]
and 
\[
C=\{ (a_1, a_2)\in D \mid  a_2= \sum_{1\leq n < \infty} \mathrm{Ad}_{U_nV_{\infty, t}}(T_t), T_t \in \cA  \}.
\]

The $\ast$-homomorphisms $(\alpha, \gamma)\rtimes_r1$ and $(\beta, \gamma)\rtimes_r1$ from $\cA\rtimes_rG$ to $C_{\mathrm{Gcont}}\rtimes_rG$  are conjugate in $C_{\mathrm{Gcont}}\rtimes_rG$ by a partial isometry $w=(w_1, w_2)$ in the multiplier algebra $M(C_{\mathrm{Gcont}})\subset M(C_{\mathrm{Gcont}}\rtimes_rG)$ (note $w$ is $G$-equivariant).

It follows $(\alpha\rtimes_r1, \gamma\rtimes_r1)$ and $( \beta\rtimes_r1, \gamma\rtimes_r1)$ from $\cA\rtimes_rG$ to $D'$ induce the same map on the K-theory groups. 

Now we have that 
\[
(\alpha\rtimes_r1, \gamma\rtimes_r1)_\ast = ( \beta\rtimes_r1, \gamma\rtimes_r1)_\ast \colon K_\ast(\cA\rtimes_rG) \to K_\ast(D').
\]

From here, by the same argument as in the non-equivariant case, we see that the two maps $\mathrm{Ad}_{V_{1,t}}\rtimes_r1$, $\mathrm{Ad}_{V_{\infty,t}}\rtimes_r1$ from $\cA\rtimes_rG$ to $\cA\rtimes_rG$ induce the same map on the K-theory groups. Thus, we have $D^G_\ast(f_0)=D^G_\ast(f_1)$.

(2) The second claim: by homotopy invariance, it is enough to show $\bD_\ast^{H}(\mathrm{point}) \cong K_\ast(C^*_r(H))$. This follows from the following proposition.
\end{proof}

\begin{proposition}\label{prop_pointcase} For any locally compact group $G$, for any $G$-$C^*$-algebra $B$, let
\[
\mathrm{ev}_1 \colon C_b([1, \infty), \Compacts(l^2(\N))\otimes B)_{\Gcont} \to  \Compacts(l^2(\N))\otimes B
\]
be the evaluation at $t=1$. Then, its reduced crossed product
\[
\mathrm{ev}_1 \colon C_b([1, \infty), \Compacts(l^2(\N))\otimes B)_{\Gcont}\rtimes_rG \to  \Compacts(l^2(\N))\otimes B\rtimes_rG
\]
induces an isomorphism on the K-theory groups. The same is true if we replace the reduced crossed product to the maximal one. More generally for any $G$-$C^*$-algebra $B_1$,  the evaluation map
\[
\mathrm{ev}_1\colon (C_b([1, \infty), \Compacts(l^2(\N))\otimes B)_{\Gcont}\otimes B_1) \rtimes_rG \to ( \Compacts(l^2(\N))\otimes B\otimes B_1) \rtimes_rG
\]
induces an isomorphism on the K-theory groups. The same is true if we replace either the minimal tensor product or the reduced crossed product or both to the maximal one.

\end{proposition}
\begin{proof} Let $I_u$ be the kernel of the evaluation map 
\[
\mathrm{ev}_1\colon  C_{b, u}([1, \infty), \Compacts(l^2(\N))\otimes B)_{\Gcont}\to \Compacts(l^2(\N))\otimes B
\]
at $t=1$. It is easy see that the Eilenberg swindle argument in \cite[Proposition 6.3.3]{WY2020} descends to the crossed product to show $K_\ast(I_u\rtimes_rG)=0$. Since the evaluation map admits a $G$-equivariant splitting,  it follows that its reduced crossed product
\[
\mathrm{ev}_1 \colon C_{b, u}([1, \infty), \Compacts(l^2(\N))\otimes B)_{\Gcont}\rtimes_rG \to  \Compacts(l^2(\N))\otimes B\rtimes_rG
\]
induces an isomorphism on the K-theory groups. Thus, it suffices to show the natural inclusion 
\[
C_{b, u}([1, \infty), \Compacts(l^2(\N))\otimes B)_{\Gcont}\rtimes_rG \to C_{b}([1, \infty), \Compacts(l^2(\N))\otimes B)_{\Gcont}\rtimes_rG
\]
induces an isomorphism on the $K$-theory groups. For this, we just need to check that the argument in the proof of \cite[Theorem 3.4]{WY2021} generalizes to this setting. Indeed, the relevant pushout (pullback) squares that appear in the proof \cite[Theorem 3.4]{WY2021} are preserved by the reduced crossed product. 
This is because for any $G$-$C^*$-algebra $B_1$, in the following pushout square 
\[
\xymatrix{
C_b([1, \infty), B_1)_{\Gcont}\ar[d]^-{}    \ar[r]^-{}   &   C_b(E, B_1)_{\Gcont}   \ar[d]^-{}      \\
C_b(O, B_1)_{\Gcont}    \ar[r]^-{}     &           C_b(\N, B_1)_{\Gcont} 
}
\]
for $E=\sqcup_{n\geq1}[2n, 2n+1]$ and $O=\sqcup_{n\geq1}[2n-1, 2n]$, each of the two vertical (and horizontal) surjections admits a $G$-equivariant c.c.p.\ splitting (by extending functions constantly and by multiplying a bump function). The uniformly continuous case is same. The rest of the argument in their proof works verbatim except a few points we have to be careful which we explain. We refer the reader to the proof of  \cite[Theorem 3.4]{WY2021} to follow the following explanations. Let
\[
 C^0_b(E) \subset C_b(E,  \Compacts(l^2(\N))\otimes B)_{\Gcont}  
\]
be the subalgebra of functions $f$ on $E$ such that $f(2n)=0$ for all $n$. Using the above pullback square and similar ones, it comes down to showing that the K-theory of $C^0_b(E) \rtimes_rG$ is zero. In fact, the K-theory of $(C^0_b(E)\otimes B_1)\rtimes_rG$ is zero for any $G$-$C^*$-algebra $B_1$. In this generality, it is enough to consider $K_0$ (taking $B_1=C_0(\R)$). We consider 
\[
x = [p] - [1_k] \in K_0( (C^0_b(E)\otimes B_1) \rtimes_rG)
\] 
for $p$ in the matrix algebra of the unitization of $(C^0_b(E)\otimes B_1)\rtimes_rG$. The first point we explain is that using the inclusion (this is well-defined)
\[
(C^0_b(E)\otimes B_1) \rtimes_rG \to C_b(E,     (\Compacts(l^2(\N))\otimes B \otimes B_ 1)\rtimes_rG),
\]
we view $p$ as a function on $E$ with values in the matrix algebra of the unitization of $(\Compacts(l^2(\N))\otimes B \otimes B_ 1)\rtimes_rG$. Using the uniform continuity of $p$ on each interval $[2n, 2n+1]$, we choose $r_n\in ( 0, 1)$ as in their proof. The second point we explain is that the element
\[
x_\infty= [\sum_{l=0}^\infty s_l p^{(l)}s_l^* ] - [\sum_{l=0}^{\infty} s_l1_k s_l^\ast]
\]
is well-defined in our context, that is $x_\infty \in K_0((C^0_b(E)\otimes B_1) \rtimes_rG )$. Here, $s_l$ are the isometries on $l^2(\N)$ such that $\sum_{l=0}^\infty s_l s_l^\ast=1$. This is because the function
\[
\sum_{l=0}^\infty s_l p^{(l)}s_l^*
\]
on $E$ is defined by the reduced crossed product of the tensor product of the identity on $B_1$ with a $G$-equivariant $\ast$-homomorphism on $C^0_b(E)$ which sends $T \in C^0_b(E)$ to
\[
\sum_{l=0}^\infty s_l T^{(l)}s_l^*
\]
where 
\[
T^{(l)}_t=  T_{2n+(t-2n)(r_n)^l}\,\,\,\, \text{for $t\in [2n, 2n+1]$}.
\]
Only $r_n\in (0 ,1)$ is used for this map to send $C^0_b(E)$ to itself. From here, their proof of $x_\infty +x = x_\infty$ works verbatim. The case of maximal crossed product and the case with an extra coefficient  $B_1$ can be proved in the same way.
\end{proof}

In fact, it is not hard to see that the argument in the proof \cite[Theorem 3.4]{WY2021} generalizes to show the following:

\begin{proposition}\label{prop_ucsame} For any universal $X$-$G$-module $H_X$ or for any $X$-$G$-module of the form $H_X^0\otimes l^2(\N)$ for an $X$-$G$-module $H_X^0$, the natural inclusion induces an isomorphism
\[
K_\ast(RL^*_u(H_X)\rtimes_rG) \cong K_\ast(RL^*_c(H_X)\rtimes_rG).
\]
The same is true if we replace the reduced crossed product to the maximal one. More generally, for any $G$-$C^*$-algebra $B$, the natural inclusion 
induces an isomorphism
\[
K_\ast((RL^*_u(H_X)\otimes B) \rtimes_rG) \cong K_\ast((RL^*_c(H_X)\otimes B)\rtimes_rG).
\]
The same is true if we replace either the minimal tensor product or the reduced crossed product or both to the maximal one.
\end{proposition} 
\begin{proof} The proof of  \cite[Theorem 3.4]{WY2021} generalizes to this setting just as we explained in the proof of Proposition \ref{prop_pointcase}. 
\end{proof}


\section{With coefficients}
Let $B$ be a separable $G$-$C^*$-algebra $B$. We will generalize our previous results on $G$-equivariant $K$-homology to the case with coefficient $B$. This will be quite straightforward so it will be brief.

\begin{definition} For any proper $G$-space $X$ and for any $X$-$G$-module $H_X$, we define a $G$-$C^*$-algebra $RL^\ast_c(H_X\otimes B)$ as the norm completion of the $G$-$\ast$-subalgebra $RL^{\mathrm{alg}}_c(H_X\otimes B)$ of $C_b([1, \infty), \Compacts(H_X)\otimes B)$ that consists of bounded, $G$-continuous, norm-continuous $\Compacts(H_X)\otimes B$-valued functions $T_t$ on $[1, \infty)$, such that
\begin{enumerate}
\item $T_t$ has uniform compact support in a sense that there is a compact subset $K$ of $X$ such that $T_t=\chi_KT_t\chi_K$ for all $t\geq1$, and 
\item for any $\phi$ in $C_0(X)$, we have
\[
 \lim_{t\to \infty }\lVert[\phi, T_t]\rVert=\lim_{t\to \infty}\lVert\phi T_t-T_t\phi\rVert = 0.
\]
\end{enumerate}
A $G$-$C^*$-algebra $RL^\ast_u(H_X\otimes B)$ is similarly defined in $C_{b, u}([1, \infty), \Compacts(H_X)\otimes B)$.
\end{definition}

Both $RL^\ast_c(H_X\otimes B)$ and $RL^\ast_u(H_X\otimes B)$ contain a $G$-invariant ideal $RL^\ast_{0}(H_X\otimes B)=C_0([1,\infty), \Compacts(H_X)\otimes B)$. We denote by $RL^\ast_{c, Q}(H_X\otimes B)$, resp. $RL^\ast_{u, Q}(H_X\otimes B)$, the quotient of $RL^\ast_c(H_X\otimes B)$, resp. $RL^\ast_u(H_X\otimes B)$ by the ideal $RL^\ast_{0}(H_X\otimes B)$. These quotients are naturally $G$-$C_0(X)$-algebras.

Let $X, Y$ be proper $G$-spaces and $H_X, H_Y$ be an $X$-$G$-module and a $Y$-$G$-module respectively. Given a $G$-equivariant continuous cover $(V_t\colon H_X \to H_Y)_{t\in[1, \infty)}$ of a $G$-equivariant continuous map $f\colon X\to Y$, the conjugation by $V_t$ defines a $G$-equivariant $\ast$-homomorphism
\[
\mathrm{Ad}_{V_t} \colon RL_{c}^\ast(H_X\otimes B) \to RL_{c}^\ast(H_Y\otimes B).
\]
Thus, it defines a $\ast$-homomorphism
\[
\mathrm{Ad}_{V_t}\colon RL_{c}^\ast(H_X\otimes B)\rtimes_rG \to RL_{c}^\ast(H_Y\otimes B)\rtimes_rG.
\]
This $\ast$-homomorphism depends on the cover $V_t$ of $f$ but the induced map on their K-theory groups is independent of the choice of $V_t$.

\begin{definition}\label{def_DBG} Choose any universal $X$-$G$-module $H_X$ for each proper $G$-space $X$ and any $G$-equivariant continuous cover $(V^f_t\colon H_X\to H_Y)_{t\in [1,\infty)}$ for each $G$-equivariant continuous map $f\colon X\to Y$. A functor $\mathbb{D}^{B, G}_\ast$ from the category $\mathcal{PR}^G$ of (second countable, locally compact) proper $G$-spaces to the category $\mathcal{GA}$ of graded abelian groups is defined as
\[
\mathbb{D}^{B, G}_\ast(X) = K_\ast(RL^\ast_c(H_X\otimes B)\rtimes_rG),
\]
\[
\mathbb{D}^{B, G}_\ast(f\colon X\to Y)=\mathrm{Ad}_{V^f_t \ast} \colon K_\ast(RL^\ast_c(H_X\otimes B)\rtimes_rG) \to K_\ast(RL^\ast_c(H_Y\otimes B)\rtimes_rG).
\]
\end{definition}

\begin{proposition}The functor  $\mathbb{D}^{B, G}_\ast$ from  $\mathcal{PR}^G$ to $\mathcal{GA}$ is well-defined. The functor does not depend on the choice of universal $X$-$G$-modules $H_X$ up to canonical equivalence.
\end{proposition}
\begin{proof} The proof of \ref{prop_welldefG} works verbatim.
\end{proof}

\begin{theorem}\label{thm_coeff} The functor $\mathbb{D}^{B, G}_\ast$ satisfies the following: 
\begin{enumerate}
\item  $\mathbb{D}^{B, G}_\ast(\mathrm{empty\, set})\cong 0$.
\item  Induction from an open subgroup: for any open subgroup $H$ of $G$ and any proper $H$-space $Y$, we have a natural isomorphism $\mathbb{D}^{B, G}_\ast(G\times_H Y)\cong \mathbb{D}_\ast^{B, H}(Y)$. If $H$ is a compact open subgroup of $G$, we have $\mathbb{D}^{B, G}_\ast(G/H\times \Delta^k)\cong K_\ast(B\rtimes_rH)$ for any k-simplex (ball) $\Delta^k$ with the trivial $G$-action.

\item Representable: if $(X_i)_{i \in I}$ is the net of $G$-compact, $G$-invariant closed subsets of $X$, ordered by the inclusion, the canonical maps $\bD^{B, G}_\ast(X_i) \to \bD^{B, G}_\ast(X)$ induce a natural isomorphism 
\[
\varinjlim_{i \in I}\bD^{B, G}_\ast(X_i) \cong \bD^{B, G}_\ast(X).
\]
\item Mayer--Vietoris sequence for a $G$-invariant open cover: if $X=U\cup V$ for open $G$-invariant subsets $U$ and $V$ of $X$, we have a natural Meyer--Vietoris sequence
\[
\xymatrix{
 \bD^{B, G}_0(U\cap V) \ar[r]^-{}&    \bD^{B, G}_0(U) \oplus \bD^{B, G}_0(V) \ar[r]^-{} &   \bD^{B, G}_0(X) \ar[d]^-{}   \\
\bD^{B, G}_1(X)     \ar[u]^-{}         &   \bD^{B, G}_1(U) \oplus \bD^{B, G}_1(V)        \ar[l]^-{} &            \bD^{B, G}_1(U\cap V) \ar[l]^-{}.
}
\]
\item  Homotopy invariance: if $h\colon X\times[0 ,1]\to Y$ is a $G$-equivariant continuous homotopy between $f_0, f_1\colon X\to Y$, then $\bD^{B, G}_\ast(f_0)=\bD^{B, G}_\ast(f_1)$. 
\end{enumerate}
\end{theorem}
\begin{proof} It is straightforward to see that the proof of Theorem \ref{thm_Ghomology} generalizes to prove all the assertions except the second assertion of (2). Using the homotopy invariance (5) and the first assertion of (2), it is enough to show
\[
\mathbb{D}^{B, H}_\ast(\mathrm{point}) = K_\ast(C_b([1, \infty), \Compacts(H_0\otimes B))_{\mathrm{Hcont}}\rtimes_rH)\cong K_\ast(B\rtimes_rH),
\]
for a separable infinite-dimensional $H$-Hilbert space $H_0=l^2(\N)\otimes L^2(H)$. This follows from Proposition \ref{prop_pointcase}.
\end{proof}

We have the following analogue of Proposition \ref{prop_ucsameB}

\begin{proposition}\label{prop_ucsameB} For any universal $X$-$G$-module $H_X$ or for any $X$-$G$-module of the form $H_X^0\otimes l^2(\N)$ for an $X$-$G$-module $H_X^0$, the natural inclusion induces an isomorphism
\[
K_\ast(RL^*_u(H_X\otimes B)\rtimes_rG) \cong K_\ast(RL^*_c(H_X\otimes B)\rtimes_rG).
\]
The same is true if we replace the reduced crossed product to the maximal one. More generally, for any $G$-$C^*$-algebra $B_1$, the natural inclusion 
induces an isomorphism
\[
K_\ast((RL^*_u(H_X\otimes B)\otimes B_1) \rtimes_rG) \cong K_\ast((RL^*_c(H_X\otimes B)\otimes B_1)\rtimes_rG).
\]
The same is true if we replace either the minimal tensor product or the reduced crossed product or both to the maximal one.
\end{proposition} 
\begin{proof} Again, the proof of  \cite[Theorem 3.4]{WY2021} generalizes to this setting just as we explained in the proof of Proposition \ref{prop_pointcase}. 
\end{proof}

Any $G$-equivariant $\ast$-homomorphism $\pi \colon B_1\to B_2$ induces 
\[
\pi\colon RL^*_c(H_X\otimes B_1)\rtimes_rG \to RL^*_c(H_X\otimes B_2)\rtimes_rG.
\]
This defines a natural transformation
\[
\pi_\ast\colon \bD^{B_1, G}_\ast(X) \to  \bD_\ast^{B_2, G}(X)
\]
of the functors from $\mathcal{PR}^G$ to $\mathcal{GA}$.

We may also define a variant of representable $G$-equivariant K-homology with coefficient $B$ in the following way. 

\begin{definition} Choose any universal $X$-$G$-module $H_X$ for each proper $G$-space $X$ and any $G$-equivariant continuous cover $(V^f_t\colon H_X\to H_Y)_{t\in [1,\infty)}$ for each $G$-equivariant continuous map $f\colon X\to Y$. A functor $\mathbb{D}^{\otimes B, G}_\ast$ from $\mathcal{PR}^G$ to $\mathcal{GA}$ is defined as
\[
\mathbb{D}^{\otimes B, G}_\ast(X) = K_\ast((RL^\ast_c(H_X)\otimes B )\rtimes_rG),
\]
\[
\mathbb{D}^{\otimes B, G}_\ast(f\colon X\to Y)=\mathrm{Ad}_{V^f_t \ast} \colon K_\ast((RL^\ast_c(H_X)\otimes B )\rtimes_rG) \to K_\ast((RL^\ast_c(H_Y)\otimes B )\rtimes_rG).
\]
\end{definition}

Although the following sequence may not be exact in general,
\begin{equation*}
\xymatrix{
0  \ar[r]^-{}     &   RL_0^\ast(H_X)\otimes B          \ar[r]^-{}         &  RL_c^\ast(H_X)\otimes B          \ar[r]^-{} &               RL_{c, Q}^\ast(H_X)\otimes B \ar[r]^-{}       &           0,
}
\end{equation*}
this would not be a problem. The quotient of $RL_c^\ast(H_X)\otimes B$ by $RL_0^\ast(H_X)\otimes B$ is a $G$-$C_0(X)$-algebra which can be used as an enough substitute for $RL_{c, Q}^\ast(H_X\otimes B)$ for our purpose.

\begin{proposition}The functor  $\mathbb{D}^{\otimes B, G}_\ast$ from  $\mathcal{PR}^G$ to $\mathcal{GA}$ is well-defined. The functor does not depend on the choice of universal $X$-$G$-modules $H_X$ up to canonical equivalence.
\end{proposition}
\begin{proof} The proof of \ref{prop_welldefG} works verbatim.
\end{proof}

\begin{theorem}\label{thm_coeff2} The functor $\mathbb{D}^{\otimes B, G}_\ast$ satisfies the following: 
\begin{enumerate}
\item  $\mathbb{D}^{\otimes B, G}_\ast(\mathrm{empty\, set})\cong 0$.
\item Induction from an open subgroup: for any open subgroup $H$ of $G$ and any proper $H$-space $Y$, we have a natural isomorphism $\mathbb{D}^{\otimes B, G}_\ast(G\times_H Y)\cong \mathbb{D}_\ast^{B, H}(Y)$. If $H$ is a compact open subgroup, we have $\mathbb{D}^{\otimes B, G}_\ast(G/H\times \Delta^k)\cong K_\ast(B\rtimes_rH)$ for any k-simplex (ball) $\Delta^k$ with the trivial $G$-action.

\item Representable: if $(X_i)_{i \in I}$ is the net of $G$-compact, $G$-invariant closed subsets of $X$, ordered by the inclusion, the canonical maps $\bD^{\otimes B, G}_\ast(X_i) \to \bD^{\otimes B, G}_\ast(X)$ induce a natural isomorphism 
\[
\varinjlim_{i \in I}\bD^{\otimes B, G}_\ast(X_i) \cong \bD^{\otimes B, G}_\ast(X).
\]
\item Mayer--Vietoris sequence for a $G$-invariant open cover: if $X=U\cup V$ for open $G$-invariant subsets $U$ and $V$ of $X$, we have a natural Meyer--Vietoris sequence
\[
\xymatrix{
 \bD^{\otimes B, G}_0(U\cap V) \ar[r]^-{}&    \bD^{\otimes B, G}_0(U) \oplus \bD^{\otimes B, G}_0(V) \ar[r]^-{} &   \bD^{\otimes B, G}_0(X) \ar[d]^-{}   \\
\bD^{\otimes B, G}_1(X)     \ar[u]^-{}         &   \bD^{\otimes B, G}_1(U) \oplus \bD^{\otimes B, G}_1(V)        \ar[l]^-{} &            \bD^{\otimes B, G}_1(U\cap V) \ar[l]^-{}.
}
\]
\item  Homotopy invariance: if $h\colon X\times[0 ,1]\to Y$ is a $G$-equivariant continuous homotopy between $f_0, f_1\colon X\to Y$, then $\bD^{\otimes B, G}_\ast(f_0)=\bD^{\otimes B, G}_\ast(f_1)$. 
\end{enumerate}
\end{theorem}
\begin{proof} Again, the proof of Theorem \ref{thm_Ghomology} works in this setting in a straightforward way except a few points which we explain. The second assertion of the (2) is proven by using the homotopy invariance (5), the first assertion of (2) and $\mathbb{D}^{\otimes B, H}_\ast(\mathrm{point})\cong K_\ast(B\rtimes_rH)$ which follows from Proposition \ref{prop_pointcase} with $B=\Compacts(L^2(H))$ and $B_1=B$ (for $G=H$). For the Mayer--Vietoris sequence (4), the previous argument works verbatim if we use, instead of $RL_{c, Q}^*(H_X)\otimes B$, the quotient of
$RL_{c}^*(H_X)\otimes B$ by $RL_0^*(H_X)\otimes B$ which is a $G$-$C_0(X)$-algebra.
\end{proof}

We have a natural inclusion
\[
(RL_c^*(H_X)\otimes B ) \rtimes_rG \to RL_c^*(H_X\otimes B)\rtimes_rG
\]
induced from the inclusion 
\[
RL_c^*(H_X)\otimes B  \to  RL_c^*(H_X\otimes B).
\]
The latter inclusion is well-defined since the inclusion 
\[
C_b([1 ,\infty), \Compacts(H_X)) \otimes B \to C_b([1 ,\infty), \Compacts(H_X)\otimes B) 
\]
is well-defined. This follows, for example, from that in general, the inclusion
\[
\Linears(\E) \otimes B \to \Linears(\E\otimes B)
\]
is well-defined for any Hilbert $A$-module $\E$ for a $C^*$-algebra $A$. The inclusion induces a natural transformation 
\[
\mathbb{D}^{\otimes B, G}_\ast(X)  \to    \mathbb{D}^{B, G}_\ast(X)
\]
of functors from $\mathcal{PR}^G$ to $\mathcal{GA}$.

We refer the reader to \cite[Appendix]{Valette02} for the definition of a proper $G$-CW complex for a discrete group $G$.

\begin{theorem}\label{thm_discrete_natural_isom} Let $G$ be a countable discrete group, $B$ be a separable $G$-$C^*$-algebra and $X$ be a proper $G$-space $X$ which is $G$-equivariantly homotopy equivalent to a proper $G$-CW complex. Then, for any universal $X$-$G$-module $H_X$, the natural inclusion
\[
(RL_c^*(H_X)\otimes B ) \rtimes_rG \to RL_c^*(H_X\otimes B)\rtimes_rG
\]
induces an isomorphism on $K$-theory groups. In other words, the natural transformation 
\[
\mathbb{D}^{\otimes B, G}_\ast(X) \to   \mathbb{D}^{B, G}_\ast(X)
\]
is an isomorphism if $X$ is $G$-equivariantly homotopy equivalent to a proper $G$-CW complex.
\end{theorem}
\begin{proof} The natural transformation $\mathbb{D}^{\otimes B, G}_\ast \to \mathbb{D}^{B, G}_\ast$ is an isomorphism when $X=G/H$ for any finite subgroup $H$. The claim follows from Theorem \ref{thm_coeff} and Theorem \ref{thm_coeff2}.
\end{proof}


\section{Comparison with the localized Roe algebra}
In this section, we compare the crossed product algebra $RL^*_u(H_X)\rtimes_rG$ with the localized equivariant Roe algebra / the equivariant localization algebra $C^*_L(X)^G$ for any discrete group  $G$ and for any proper $G$-space $X$. We refer the reader to \cite[Section 2.2]{GHM} for a supplement to this section.

For later purpose, most of the results in this section will be proved in a general setting where $G$ is any locally compact group and $X$ is any proper $G$-space. We also let $B$ be any separable $G$-$C^*$-algebra which will be used  as a coefficient.

For any $X$-$G$-module $H_X$, we consider $H_X\otimes B$ as a $G$-Hilbert $B$-module equipped with the diagonal $G$-action. Any Borel function $\phi$ on $X$ is represented as $\phi\otimes 1\in \Linears(H_X\otimes B)$ and we write $\phi=\phi\otimes 1$ as long as there is no confusion.

Recall that $T \in \Linears(H_X\otimes B)$ for an $X$-module $H_X\otimes B$ is locally compact if $\phi T, T\phi \in \Compacts(H_X\otimes B)$ for any $\phi \in C_0(X)$.

\begin{definition}(c.f.\ \cite[Definition 5.2.1]{WY2020}) For any $X$-$G$-module $H_X$, we define the equivariant Roe algebra $C^*(H_X\otimes B)^G_{\mathrm{Gcpt}}$ with $G$-compact support as the norm completion of the $\ast$-algebra $\bC(H_X\otimes B)_{\mathrm{Gcpt}}^G$ consisting of $T \in \Linears(H_X\otimes B)$ satisfying:
\begin{enumerate}
\item $T$ has $G$-compact support, i.e. $\mathrm{supp}(T)\subset Y\times Y$ for some $G$-compact subset $Y$ of $X$.
\item $T$ is $G$-equivariant.
\item $T$ is locally compact.
\item $T$ is properly supported, i.e. for any compact subset $A$, there is a compact subset $B$ so that $T\chi_A=\chi_BT\chi_A$ and $\chi_AT=\chi_AT\chi_B$.
\end{enumerate}
If $X$ is $G$-compact, we write $C^*(H_X\otimes B)^G=C^*(H_X\otimes B)^G_{\mathrm{Gcpt}}$ and call it the equivariant Roe algebra.
\end{definition}
\begin{remark}\label{rem_same} If $X$ has a $G$-invariant proper metric $d$ and if $T$ is $G$-equivariant and has $G$-compact support, the condition (4) is same as saying $T$ has finite propagation with respect $d$. Indeed, if $Y=GY_0$ is the support of $T$ for a compact subset $Y_0$ of $X$ and if we let $Y_1$ be any compact neighborhood of $Y_0$, we have $T\chi_{Y_1}=\chi_{Y_2}T\chi_{Y_1}$ and $\chi_{Y_1}T=\chi_{Y_1}T\chi_{Y_2}$  for some compact subset $Y_2$ of $Y$. This condition already implies that if $(x, y)\in \supp(T)$ and if either $x\in gY_0$ or resp. $y \in gY_0$, then $y \in gY_2$, resp $x\in gY_2$. Hence, $\prop(T)\leq d(Y_0, Y_2)$. The converse is obvious.
\end{remark}

\begin{proposition} Let $X$ be a $G$-compact proper $G$-space with a $G$-invariant proper metric $d$ and $H_X$ be an $X$-$G$-module. Then, the equivariant Roe algebra $C^*(H_X\otimes B)^G$ is same as the norm completion of the $\ast$-algebra $\bC(H_X\otimes B)^G$ consisting of $G$-equivariant, locally compact operator in $\Linears(H_X\otimes B)$ with finite propagation with respect to $d$. 
\end{proposition}
\begin{proof} This follows from Remark \ref{rem_same}. 
\end{proof}

In general, $C^*(H_X\otimes B)^G_{\mathrm{Gcpt}}$ is just the inductive limit (union) of $C^*(H_{X_n}\otimes B)^G$ where $X_n$ are $G$-compact, $G$-invariant closed subsets of $X$ with $\cup X_n=X$ and $H_{X_n}=\chi_{X_n}H_X$.

\begin{definition}\label{def_tildeH} Let $H_X$ be an $X$-$G$ module. We define an $X$-$G$-module
\[
\tilde H_X = H_X\otimes L^2(G)
\]
equipped with the following covariant representation of $G$ and $C_0(X)$: for $\phi \in C_0(X)$ 
\[
\phi \mapsto  \phi\otimes 1 \in \Linears(H_X\otimes L^2(G)),
\] 
and for $g \in G$
\[
g \mapsto u_g\otimes \lambda_g \in \Linears(H_X\otimes L^2(G))
\] 
where $u_g$ is the unitary on $H_X$ corresponding to $g$ and $\lambda_g$ is the left-translation by $g$.
\end{definition}

\begin{definition}\label{def_rightreg} The right-regular representation
\[
\rho\colon\Linears(H_X\otimes B)_{\mathrm{Gcont}}\rtimes_rG \to  \Linears(\tilde H_X\otimes B)
\] 
of $\Linears(H_X\otimes B)_{\mathrm{Gcont}}\rtimes_rG$ is defined for $T \in \Linears(H_X\otimes B)$, 
\[
(\rho(T)f)(h)=h(T)f(h) 
\]
for $(f\colon h\mapsto f(h)\in H_X \otimes B) \in L^2(G, H_X\otimes B)$ and for $g\in G$,
\[
\rho(g)= 1\otimes 1\otimes \rho_g \in  \Linears(H_X\otimes B \otimes L^2(G))
\]
where $\rho_g$ is the right-regular representation of $g\in G$, 
\[
\rho_gf(s) = f(sg)\Delta^{1/2}(g) 
\]
for $f \in L^2(G)$ where $\Delta$ is the modular function.
\end{definition}
 
\begin{lemma} For any locally compact group $G$, for any proper $G$-space $X$ and for any $X$-$G$-module $H_X$, the right-regular representation $\rho$ maps $\Compacts(H_X\otimes B)\rtimes_rG$ into the equivariant Roe algebra $C^*(\tilde H_X\otimes B)_{\mathrm{Gcpt}}^G \subset \Linears(\tilde H_X\otimes B)$ with $G$-compact support.
\end{lemma}
\begin{proof}
It is easy to see that the image of $\rho$ consists of $G$-equivariant operators on $\tilde H_X\otimes B$ and that $\rho(g)$ for any $g$ in $G$ commutes with $C_0(X)$. We show that for any $T \in \Compacts(H_X\otimes B)$ with compact support $X_0$ of $X$, and for any function $f \in C_c(G)$, $\rho(T)\rho(f)\in \bC(\tilde H_X\otimes B)_{\mathrm{Gcpt}}^G$. This implies our assertion. The operator $\rho(T)$ is of the form
\[
(h \mapsto h(T)) \in C_b(G, \Compacts(H_X\otimes B)) 
\]
acting on $L^2(G, H_X\otimes B)$ in an obvious way and it is easy to see that $\rho(T)$ has $G$-compact support $GX_0$. Moreover, for any compact subset $A$ of $X$,
\[
\chi_Ah(T)=h(T)\chi_A=0
\]
for all $h$ in $G$ outside a compact subset $K$. In particular, by letting $B=KX_0$ which is compact, we have
\[
\chi_Ah(T)= \chi_Ah(T)\chi_B,  h(T)\chi_A=\chi_Bh(T)\chi_A
\] 
for all $h$ in $G$. This shows $\rho(T)$ is properly supported. Finally, for any $\phi\in C_c(X)$, the operator $\phi \rho(T)$ is of the form
\[
(h \mapsto \phi h(T)) \in C_c(G, \Compacts(H_X\otimes B)) 
\]
acting on $L^2(G, H_X\otimes B)$ in an obvious way. It follows $\phi \rho(T)\rho(f) \in \Compacts(\tilde H_X\otimes B)$ and similarly $\rho(T)\rho(f)\phi  \in \Compacts(\tilde H_X\otimes B)$. We see that $\rho(T)\rho(f)$ is locally compact so we are done.
\end{proof}

\begin{proposition}\label{prop_isom}(c.f. \cite[Theorem 2.11]{GHM}) For any locally compact group $G$, for any proper $G$-space $X$ and for any $X$-$G$-module $H_X$, the right-regular representation 
\[
\rho\colon \Compacts(H_X\otimes B)\rtimes_rG \to  \Linears(\tilde H_X\otimes B)
\]
induces an isomorphism 
\[
\Compacts(H_X\otimes B)\rtimes_rG  \cong   C^*(\tilde H_X\otimes B)_{\mathrm{Gcpt}}^G.
\]
\end{proposition}

Before giving a proof of this, we recall a useful lemma. 

\begin{lemma}\label{lem_sum}\cite[Lemma 2.5, 2.6]{Nishikawa19}  Let $X$ be a proper $G$-space and $H_X$ be an $X$-$G$-module. For any $\phi_0,\phi_1 \in C_c(X)$, there is a constant $C>0$ such that the following holds. 
Let $(T_g)_{g \in G}$ be a uniformly bounded family of operators on $H_X$ which defines a bounded operator on $L^2(G, H_X)$. Then, the map
\[
v\mapsto \int_{g\in G}g(\phi_0)T_g g(\phi_1)vd\mu_G(g)
\]
on $H_X$ defines the bounded operator $\int_{g\in G}g(\phi_0)T_g g(\phi_1)d\mu_G(g)$ on $H_X$, and we have
\[
\lVert\int_{g\in G}g(\phi_0)T_g g(\phi_1)d\mu_G(g)\rVert \leq C\sup_{g\in G}\lVert T_g\rVert.
\]
Moreover, if the function $g\mapsto T_g$ is in $C_0(G, \Compacts(H_X))$, the integral converges in norm to a compact operator on $H_X$.  More generally, the same assertion holds if we consider, in place of $H_X$, $\E_X$ a $G$-Hilbert $B$-module equipped with a non-degenerate representation of $G$-$C^*$-algebra $C_0(X)$ by adjointable operators (in this case, operators are all assumed to be adjointable).
\end{lemma}
\begin{proof} The proof in \cite[Lemma 2.5, 2.6]{Nishikawa19} easily generalizes to the case of a $G$-Hilbert $B$-module $\E_X$.
\end{proof}

\begin{proof}[Proof of Proposition~\ref{prop_isom}]

For any $G$-invariant closed subspace $Y$ of $X$, let $H_Y=\chi_YH_X$. Then, it is easy to see from the definition of $\tilde H_X$ and $\rho$, the following diagram commutes
\[
\xymatrix{
\Compacts(H_X\otimes B)\rtimes_rG \ar[r]^-{\rho}  &    C^*(\tilde H_X\otimes B)_{\mathrm{Gcpt}}^G    \\
 \Compacts(H_Y\otimes B)\rtimes_rG  \ar[u]^-{}  \ar[r]^-{\rho} &       C^*(\tilde H_Y\otimes B)_{\mathrm{Gcpt}}^G.     \ar[u]^-{}
}
\]
where the vertical maps are natural inclusions. Since both $\Compacts(H_X\otimes B)\rtimes_rG$ and resp. $C^*(\tilde H_X\otimes B)_{\mathrm{Gcpt}}^G$, are the inductive limit (the union) of  $\Compacts(H_Y\otimes B)\rtimes_rG$, resp.$C^*(\tilde H_Y\otimes B)_{\mathrm{Gcpt}}^G$ for $G$-compact $G$-invariant closed subspaces $Y$ of $X$, it is enough to show the isomorphism when $X$ is $G$-compact. 

Now, suppose $X$ is $G$-compact. Let $T \in \bC(\tilde H_X\otimes B)^G_{\mathrm{Gcpt}}$. We show that for any $\epsilon>0$, there is $(S\colon h\mapsto S_h) \in C_c(G, \Compacts(H_X\otimes B)) \subset \Compacts(H_X\otimes B)\rtimes_rG$ such that 
\[
\lVert T - \rho(S) \rVert < \epsilon. 
\]
Let $c \in C_c(X)$ be a cut-off function on $X$. Since $T$ is locally compact and properly supported, $cT$ is a compact operator whose support is contained in $X_0\times X_0\subset X\times X$ for some compact subset $X_0$. We have
\[
T = \int_{g\in G} g(c)^2T d\mu_G(g) = \int_{g\in G} g(c) g(cT) d\mu_g(g)=  \int_{g\in G} g(c) g(cT) g(\chi)d\mu_G(g)
\]
where $\chi \in C_c(X)$ is fixed so that $\chi=1$ on $X_0$. These integrals converge weakly. Here we used $g(T)=T$.

 Let $C$ be as in Lemma \ref{lem_sum} for $\phi_0=c$, $\phi_1=\chi$. We approximate the compact operator $cT$ on $H_X\otimes B \otimes L^2(G)$ by any 
 \[
 S'\colon h_1\mapsto (S'_{h_1}\colon h_2\mapsto S'_{h_1}(h_2) ) 
 \]
 in 
 \[
 C_c(G, C_c(G, \Compacts(H_X\otimes B)) \subset C_0(G, \Compacts(H_X\otimes B))\rtimes_rG\cong \Compacts(H_X\otimes B \otimes L^2(G))
 \] so that $\lVert cT-S'\rVert<\epsilon/C$. Here, the identification 
 \[
 C_0(G, \Compacts(H_X\otimes B))\rtimes_rG\cong \Compacts(H_X\otimes B \otimes L^2(G))
 \]
 is given by letting $C_0(G, \Compacts(H_X\otimes B))$ act on $H_X\otimes B \otimes L^2(G)$ in an obvious way and by letting the group $G$ act on $H_X\otimes B \otimes L^2(G)$ by $\rho$ i.e. the right-regular representation.

 We can arrange $\supp( S'_{h_1}(h_2) )\subset X_0\times X_0$ for all $h_1, h_2$, by multiplying $\chi_{X_0}$. This is possible because $\rho(g)$ commutes with functions $C_0(X)$.
   Let
\[
\tilde S' = \int_{g \in G} g(c) g(S') d\mu_G(g) = \int_{g \in G} g(c) g(S') g(\chi) d\mu_G(g)
\]
which converges weakly on $\tilde H_X\otimes B$.  By Lemma \ref{lem_sum} and from $\lVert cT - S'\rVert < \epsilon/C$, we have
\[
\lVert T - \tilde S'\rVert = \lVert\int_{g\in G} g(c)( g(cT) - g(S')) g(\chi) d\mu_g(g) \rVert < C \cdot \epsilon/C =\epsilon.
\]
It is not hard to see that $\tilde S' $ is a locally compact, $G$-equivariant, properly supported operator and it is of the form $\rho(S)$ where 
\[
(S\colon h\mapsto S_h) \in C_c(G, \Compacts(H_X\otimes B)) \subset \Compacts(H_X\otimes B)\rtimes_rG,
\]
\[
S_h = \int_{g\in G} g(c) g((S'_h)(g^{-1})) g(\chi) d\mu_G(g) \in \Compacts(H_X\otimes B).
\]
Note that the integral is norm convergent as $S'_h(g^{-1})=0$ for a.e. $g\in G$. That is we obtained $S\in \Compacts(H_X\otimes B)\rtimes_rG$ such that $\lVert T-\rho(S)\rVert<\epsilon$ as desired.
\end{proof}

\begin{remark} This provides us an alternative proof of the well-known fact \cite{Roe96} that for any discrete group $G$, for any $G$-compact proper $G$-space $X$ and for any ample $X$-$G$-module $H_X$, $C^*(H_X)^G\cong \Compacts(l^2(\N))\otimes C^*_r(G)$ at least when $H_X$ is of the form $\tilde H^0_X$. Of course, the general case follows from this if use the fact that the equivariant Roe algebra $C^*(H_X)^G$ is independent of an ample $X$-$G$-module $H_X$ up to isomorphisms.
\end{remark}

\begin{definition}(c.f.\  \cite[Definition 6.6.1]{WY2020}) \label{def_localized} For any $X$-$G$-module $H_X$, we define the localized Roe algebra $C_{L, u}^*(H_X\otimes B)^G_{\mathrm{Gcpt}}$ with $G$-compact support as the norm completion of the $\ast$-algebra $\bC_{L, u}(H_X\otimes B)_{\mathrm{Gcpt}}^G$ consisting of uniformly norm-continuous $\bC(H_X\otimes B)_{\mathrm{Gcpt}}^G$-valued functions $T\colon t\mapsto T_t$ on $[1, \infty)$ satisfying:
\begin{enumerate}
\item $T$ has uniform $G$-compact support, i.e. $\mathrm{supp}(T_t)\subset Y\times Y$ for some $G$-compact subset $Y$ of $X$.
\item For any open neighborhood $U$ of the diagonal in $X^+\times X^+$, there is $t_U\geq1$ such that for all $t> t_U$, $\supp(T_t)\subset U$.
\end{enumerate}
Note that the second condition is same as saying that $\prop(T_t)\to 0$ with respect to a (any) fixed metric on $X^+$. Similarly, $C_{L, c}^*(H_X\otimes B)^G_{\mathrm{Gcpt}}$ is defined by replacing uniformly norm-continuous to norm-continuous. If $X$ is $G$-compact, we write $C_{L, u}^*(H_X\otimes B)^G=C_{L, u}^*(H_X\otimes B)^G_{\mathrm{Gcpt}}$, $C_{L, c}^*(H_X\otimes B)^G=C_{L, c}^*(H_X\otimes B)^G_{\mathrm{Gcpt}}$ and call them the localized equivariant Roe algebras.
\end{definition}

\begin{remark}\label{rem_same2} If $X$ has a $G$-invariant proper metric $d$, the condition (2) is same as saying $\prop(T_t)\to 0$ as $t\to \infty$ with respect $d$ when $T$ is $G$-equivariant and has uniform $G$-compact support. Indeed, if $Y=GY_0$ is the support of $T$ for a compact subset $Y_0$ of $X$ and if $Y_1$ is any compact neighborhood of $Y_0$, (2) implies $\prop(\chi_{Y_1}T_t) \to 0$ with respect to $d$. Using $G$-equivariance, we can easily see $\prop(T_t)\to 0$ with respect to $d$. The converse is easier.
\end{remark}

\begin{remark} We used the terminology, the localized Roe algebra, from \cite[Section 6.6]{WY2020} where the terminology, the equivariant localization algebra, is reserved for some other slightly bigger algebra. 
\end{remark}

\begin{proposition} Let $X$ be a $G$-compact proper $G$-space with a $G$-invariant proper metric $d$. For any $X$-$G$-module $H_X$, the localized equivariant Roe algebra $C^*_{L, u}(H_X\otimes B)^G$ is same as the norm completion of the $\ast$-algebra $\bC_{L, u}(H_X\otimes B)^G$ of uniformly norm-continuous $\bC(H_X\otimes B)^G$-valued functions $T\colon t\mapsto T_t$ on $[1, \infty)$ with $\prop(T_t)\to 0$ as $t\to \infty$.
\end{proposition}
\begin{proof} This follows from Remark \ref{rem_same2}. 
\end{proof}

Let $H_X$ be an $X$-$G$-module and let $\tilde H_X$ be the $X$-$G$-module as in Definition \ref{def_tildeH}. Recall that the right-regular representation (Definition \ref{def_rightreg}) 
\[
\rho\colon \Linears(H_X\otimes B)_{\mathrm{Gcont}}\rtimes_rG \to  \Linears(\tilde H_X\otimes B)
\]
restricts to an isomorphism 
\[
\rho\colon \Compacts(H_X\otimes B)\rtimes_rG \cong C^*(\tilde H_X\otimes B)^G_{\mathrm{Gcpt}}.
\]

Applying the right-regular representation $\rho$ for each $t\in [1, \infty)$, $\rho$ extends to a $\ast$-homomorphism 
\[
\rho\colon C_{b}([1, \infty), \Linears(H_X\otimes B))_{\mathrm{Gcont}}\rtimes_rG \to C_b([1, \infty), \Linears(\tilde H_X\otimes B)).
\]

\begin{proposition}\label{prop_rightreg_localized} For any locally compact group $G$, for any proper $G$-space $X$ and for any $X$-$G$-module $H_X$, the right-regular representation 
\[
\rho\colon C_{b}([1, \infty), \Linears(H_X\otimes B))_{\mathrm{Gcont}}\rtimes_rG \to C_b([1, \infty), \Linears(\tilde H_X\otimes B))
\]
maps $RL^*_u(H_X\otimes B)\rtimes_rG\subset C_b([1, \infty), \Compacts(H_X\otimes B))_{\mathrm{Gcont}}\rtimes_rG$ into the localized equivariant Roe algebra $C_{L, u}^*(\tilde H_X\otimes B)_{\mathrm{Gcpt}}^G$ with $G$-compact support. Similarly, $\rho$ maps $RL^*_c(H_X\otimes B)\rtimes_rG$ into $C_{L, c}^*(\tilde H_X\otimes B)_{\mathrm{Gcpt}}^G$.
\end{proposition}
\begin{proof} Fix any metric $d$ on $X^+$. It is enough to show that for any $T \in RL^*_u(H_X)$ with uniform compact support $X_0$ such that $\prop(T_t)\to 0$ with respect to $d$, and for any $f \in C_c(G)$, 
\[
\rho(T)\rho(f) \in C_{L, u}^*(\tilde H_X\otimes B)_{\mathrm{Gcpt}}^G.
\]
We already know that $\rho(T_t)\rho(f) \in   \bC(\tilde H_X\otimes B)_{\mathrm{Gcpt}}^G$ for any $t\geq1$. Moreover it has uniform $G$-compact support $GX_0$. Note that $\rho(f)$ commutes with $C_0(X)$ so we just need to show the condition (2) in Definition \ref{def_localized} is satisfied for $\rho(T)$. The condition (2) for $\rho(T)$ is equivalent to that for any compact subset $X_1$ of $X$, $\prop(\rho(T_t)\chi_{X_1}), \prop(\chi_{X_1}\rho(T_t)) \to 0$ as $t\to \infty$ with respect to $d$. On the other hand, $\chi_{X_1}h(T_t)=h(T_t)\chi_{X_1}=0$ for $h \in G\backslash K$ for some compact subset $K$ of $G$ since $T_t$ has uniform compact support $X_0$. Thus, it is enough to show that $\prop(h(T_t))\to 0$ uniformly in $h\in K$. This can be proved as follows. For any $\epsilon>0$, let
\[
A=\{(x, y)\in X^+ \times X^+ \mid d(x, y)\geq \epsilon \}.
\]
Then, $K^{-1}A$ is a compact subset of $X^+\times X^+$ which does not contain the diagonal, so there is $\delta>0$ such that all $(x, y)\in K^{-1}A$ satisfy $d(x, y)\geq \delta$. It follows for any $(x, y)\in X_0\times X_0$ with $d(x, y)<\delta$, $d(kx, ky)<\epsilon$ for all $k\in K$. We see that $\prop(h(T_t))\to 0$ uniformly in $h\in K$ as desired. The continuous case is proved in the same way.
\end{proof}

\begin{remark} The inclusion $\rho\colon  RL^*_u(H_X\otimes B)\rtimes_rG \to C_{L, u}^*(\tilde H_X\otimes B)_{\mathrm{Gcpt}}^G$ is never surjective unless $G$ is finite. For example, if $G$ is compact and if $X$ is a point, then the inclusion is identified as the inclusion 
\[
C_{b, u}([1, \infty), \Compacts(H_X)\otimes B)\rtimes_rG \to C_{b, u}([1, \infty), (\Compacts(H_X)\otimes B)\rtimes_rG)
\]
which is not surjective unless $G$ is finite even if $H_X=\bC$ and $B=\bC$.
\end{remark}

Now we go back to the classical setting when $G$ is discrete. Here is the promised comparison between the crossed product algebra $RL^*_u(H_X)\rtimes_rG$ and the localized equivariant Roe algebra $C^*_{L, u}(H_X)^G$. A generalization with coefficient $B$ is possible but we only consider $B=\bC$ here. The continuous case can be shown too but with an extra effort.

\begin{theorem} Let $G$ be a countable discrete group, $X$ be a proper $G$-space which is $G$-equivariantly homotopic to a $G$-$CW$ complex and $H_X$ be any ample $X$-$G$-module. Then, the inclusion $\rho\colon  RL^*_u(H_X)\rtimes_rG \to C_{L, u}^*(\tilde H_X)_{\mathrm{Gcpt}}^G$ induces an isomorphism 
\[
\rho_\ast\colon K_\ast( RL^*_u(H_X)\rtimes_rG) \cong K_\ast(C_{L, u}^*(\tilde H_X)_{\mathrm{Gcpt}}^G).
\]
In particular, when $X$ is $G$-compact, the inclusion $\rho\colon  RL^*_u(H_X)\rtimes_rG \to C_{L, u}^*(\tilde H_X)^G$ induces  an isomorphism
\[
\rho_\ast\colon K_\ast( RL^*_u(H_X)\rtimes_rG) \cong K_\ast(C_{L, u}^*(\tilde H_X)^G).
\]

\end{theorem}
\begin{proof}  It suffices to show the case when $X$ is $G$-compact. For any $X$-$G$-module $H_X$, let  $L^*(H_X)^G$ be the equivariant localization algebra as defined in \cite[Definition 6.5.1]{WY2020}. The $C^*$-algebra $L^*(H_X)^G$ is the norm completion of the $\ast$-algebra $\mathbb{L}[H_X]^G$ of all bounded functions $T_t$ on $[1 ,\infty)$ to $\Linears(H_X)$ such that
\begin{enumerate}
\item $T_t$ is $G$-equivariant,
\item for any compact subset $K$ of $X$, there exists $t_K\geq1$ such that for all $t\geq t_k$, $\chi_KT_t$, $T_t\chi_K$ are compact, and the functions $t\mapsto \chi_KT_t$, $t\mapsto T_t\chi_K$ are uniformly norm-continuous when restricted to $[t_k, \infty)$,
\item  for any open neighborhood of the diagonal in $X^+\times X^+$, there exists $t_U\geq1$ such that for all $t>t_U$, $\supp(T_t)\subset U$.
\end{enumerate}
 It is proved in \cite[Proposition 6.6.2]{WY2020} that the natural inclusion induces an isomorphism
\[
 K_\ast(C_{L, u}^*(H_X)^G) \cong  K_\ast(L^*(H_X)^G)
 \]
whenever $H_X$ is ample. Thus, to prove our claim, it suffices to show that when $X$ is $G$-compact, for any ample $X$-$G$-module $H_X$, the natural inclusion
\[
\rho\colon RL^*_u(H_X)\rtimes_rG \to C_{L, u}^*(\tilde H_X)^G \to L^*(\tilde H_X)^G
\]
induces an isomorphism on K-theory. For this, we consider more generally, for any not-necessary $G$-compact $X$, the equivariant localization algebra $L^*(H_X)^G_{\mathrm{Gcpt}}$ with $G$-compact support to be the completion of $\mathbb{L}[H_X]_{\mathrm{Gcpt}}^G$ which is the subalgebra of $\mathbb{L}[H_X]^G$ consisting of $T_t$ that has eventually uniform $G$-compact support, that is there is $t_0\geq1$ and a $G$-compact subset $X_0$ of $X$ such that for all $t\geq t_0$, $\supp(T_t)\subset X_0\times X_0$. Then, our claim is that the natural inclusion 
\[
\rho\colon RL^*_u(H_X)\rtimes_rG \to C_{L, u}^*(\tilde H_X)_{\mathrm{Gcpt}}^G \to L^*(\tilde H_X)^G_{\mathrm{Gcpt}}
\]
induces an isomorphism on K-theory. Now, the point is that just as in the case of our functor $\bD_\ast^G$, the assignment
\[
X \mapsto RK_\ast^G(X)=K_\ast(L^*(H_X)^G_{\mathrm{Gcpt}}), \,\,\,f \mapsto  RK_\ast^G(f)= \mathrm{Ad}_{V^f_t \ast}
\]
becomes a functor from $\mathcal{PR}^G$ to $\mathcal{GA}$, where $H_X$ is a chosen ample $X$-$G$-module for a proper $G$-space $X$ and $V^f_t$ is a chosen $G$-equivariant continuous cover of a $G$-equivariant continuous map $f\colon X\to Y$. As in the case of $\bD^G_\ast$ (see the proof of Proposition \ref{prop_welldefG}), one shows the composition law for not-necessarily proper maps by reducing it to the $G$-compact setting after showing the representability of the functor $RK^G_\ast$. Moreover, the functor $RK_\ast^G$ satisfies the five listed properties in Theorem \ref{thm_Ghomology}. Since this does not involve any new idea, we explain this very briefly. The first part of the second property (2) (induction from a finite subgroup) follows from Proposition \cite[Proposition 6.5.13]{WY2020} and from (3) the representability of $RK^G_\ast$. The fourth property (4) (the Mayer--Vietoris sequence) can be shown just as in the case of $\bD_\ast^G$ by using the quotient $L_Q^*(H_X)^G_{\mathrm{Gcpt}}$ of $L^*(H_X)^G_{\mathrm{Gcpt}}$ by the ideal $L_0^*(H_X)^G_{\mathrm{Gcpt}}$ which is the completion of the subalgebra of $\mathbb{L}[H_X]_{\mathrm{Gcpt}}^G$ consisting of $T_t$ such that for any compact subset $K$ of $X$, $\chi_KT_t, T_t\chi_K=0$ eventually. The quotient $L_Q^*(H_X)^G_{\mathrm{Gcpt}}$ is a $C_0(X/G)$-algebra (see \cite[Lemma 6.4.18]{WY2020}) and the quotient map induces an isomorphism
\[
K_\ast(L^*(H_X)^G_{\mathrm{Gcpt}}) \cong K_\ast(L_Q^*(H_X)^G_{\mathrm{Gcpt}})
\]
because $K_\ast(L_0^*(H_X)^G_{\mathrm{Gcpt}})=0$ (see \cite[Lemma 6.5.12]{WY2020}). For the fifth property (5) (the homotopy invariance), the proof for $\bD_\ast^G$ works verbatim. It remains to prove the second part of (2), that is
\[
RK^H_\ast(\mathrm{point}) \cong K_\ast(C^*_r(H))
\]
for any finite group $H$.  Using $K_\ast(C_{L, u}^*(H_X)^G) \cong  K_\ast(L^*(H_X)^G)$, we just need to show
\[
K_\ast(C_{b, u}([1, \infty), \Compacts(l^2(\N)\otimes l^2(H)))^H) \cong K_\ast(C^*_r(H)),
\]
but since $\Compacts(l^2(\N)\otimes l^2(H))^H \cong \Compacts(l^2(\N))\otimes C^*_r(H)$, this follows because the evaluation 
\[
\mathrm{ev}_{1}\colon C_{b, u}([1,\infty),   \Compacts(l^2(\N))\otimes C^*_r(H)) \to \Compacts(l^2(\N))\otimes C^*_r(H)
\]  
at $t=1$ is an isomorphism on K-theory.  

Now, the inclusion 
\[
\rho\colon  RL^*_u(H_X)\rtimes_rG  \to L^*(\tilde H_X)_{\mathrm{Gcpt}}^G
\]
induces a natural transformation of the functors $\bD^G_\ast$ to $RK_\ast^G$. We can see that this is an isomorphism when $X=G/H$ for any finite subgroup $H$ of $G$, by considering
\[
\rho\colon  RL^*_u(H_X)\rtimes_rG  \to C_{L, u}^*(\tilde H_X)^G
\]
instead. Indeed, at the level of K-theory, this inclusion is isomorphic to the inclusion 
\[
C_{b, u}([1 ,\infty), \mathrm{Ind}_H^G\Compacts(H_Y))\rtimes_rG \to C_{b, u}([1 ,\infty), \mathrm{Ind}_H^G\Compacts(H_Y)\rtimes_rG)
\] 
where $\mathrm{Ind}_H^G\Compacts(H_Y)$ is the $G$-$C_0(G/H)$-algebra with fiber $\Compacts(H_Y)$ at the coset $H$, or more precisely
\[
\mathrm{Ind}_H^G\Compacts(H_Y) = \{ f \in C_0(G, \Compacts(H_Y)) \mid f(gh)=h^{-1}(f(g)) \,\, \text{for $h \in H$} \}
\]
equipped with the left translation $G$-action. It is now easy to see that the above inclusion is an isomorphism on K-theory because both are, at the level of K-theory, isomorphic to $\mathrm{Ind}_H^G\Compacts(H_Y)\rtimes_rG$ via the evaluation at $t=1$. Of course, we have $K_\ast(\mathrm{Ind}_H^G\Compacts(H_Y)\rtimes_rG) \cong K_\ast(C^*_r(H))$. It follows since both functors $\bD^G_\ast$, $RK_\ast^G$
 satisfy properties (1) - (5) in Theorem \ref{thm_Ghomology}, the natural transformation 
\[
\rho\colon \bD^G_\ast(X) \to RK^G_\ast(X)
\]
is an isomorphism for all $X$ which is $G$-equivariantly homotopic to a $G$-CW complex. Going back to the original question, we just showed that the inclusion $\rho$ induces an isomorphism
\[
K_\ast(RL^*_u(H_X)\rtimes_rG) \cong K_\ast(C_{L, u}^*(\tilde H_X)_{\mathrm{Gcpt}})
\]
for all such $X$.
 \end{proof}



\section{The forget-control map and the Baum--Connes assembly map} \label{sec_forget}

Let $G$ be a (second countable) locally compact group and $B$ be a separable $G$-$C^*$-algebra. Let $X$ be a proper $G$-space and $H_X$ be an $X$-$G$-module.

The evaluation map
\[
\mathrm{ev}_1\colon RL^\ast_c(H_X\otimes B) \to \Compacts(H_X)\otimes B 
\]
at $1$, which we call the forget-control map, induces a $\ast$-homomorphism 
\[
\mathrm{ev}_1\colon RL^\ast_c(H_X\otimes B)\rtimes_rG \to (\Compacts(H_X)\otimes B ) \rtimes_rG \cong  \Compacts(H_X)\otimes (B\rtimes_rG).
\]
Here, the isomorphism on the right is obtained by trivializing the inner $G$-action on $\Compacts(H_X)$.  

It induces a group homomorphism (the forget-control map)
\begin{equation}\label{eq_forgetK}
\cf = \mathrm{ev}_{1\ast}\colon K_\ast(RL^\ast_c(H_X\otimes B)\rtimes_rG) \to K_\ast(B\rtimes_rG).
\end{equation}

\begin{proposition} The forget-control map $\cf$ \eqref{eq_forgetK} is functorial in $X$ and in $B$ in a sense that the following diagrams commute. For any $G$-equivariant continuous map $f\colon X\to Y$ and for any (if exists) $G$-equivariant continuous cover $(V_t\colon H_X\to H_Y)_{t\in[1,\infty)}$ of $f$,
\[
\xymatrix{
K_\ast(RL^\ast_c(H_X\otimes B)\rtimes_rG)  \ar[r]^-{\cf}  \ar[d]^-{\mathrm{Ad}_{V_t\ast}}  &   K_\ast(B\rtimes_rG)   \ar[d]^-{=}       \\
K_\ast(RL^\ast_c(H_Y\otimes B)\rtimes_rG)    \ar[r]^-{\cf}         &   K_\ast(B\rtimes_rG),
}
\]
and for any $G$-equivariant $\ast$-homomorphism $\pi\colon B_1\to B_2$,
\[
\xymatrix{
 K_\ast(RL^\ast_c(H_X\otimes B_1)\rtimes_rG)   \ar[r]^-{\cf}  \ar[d]^-{\pi_\ast}    & K_\ast(B_1\rtimes_rG)   \ar[d]^-{\pi\rtimes_r1_\ast}   \\
K_\ast(RL^\ast_c(H_X\otimes B_2)\rtimes_rG)    \ar[r]^-{\cf}         &  K_\ast(B_2\rtimes_rG). 
}
\]
In particular, $\cf$ $\eqref{eq_forgetK}$ induces a group homomorphism
\begin{equation}\label{eq_forgetD}
\cf \colon \bD^{B, G}_\ast(X) \to  K_\ast(B\rtimes_rG)
\end{equation}
which is natural in $X$ and in $B$.
\end{proposition} 
\begin{proof} The first diagram commutes since $\mathrm{Ad}_{V_t}$ on $(\Compacts(H_X)\otimes B)\rtimes_rG$ is the identity on K-theory. The second diagram commutes at the level of $\ast$-homomorphisms.
\end{proof}

\begin{definition} We call the group homomorphism $\cf$ in \eqref{eq_forgetD}, the forget-control map for the functor $\bD_\ast^{B, G}$.
\end{definition}

In the rest of this section, our goal is to show that the forget-control map $\cf$ naturally factors through the Baum--Connes assembly map (see \cite{BCH93}, \cite{Valette02}, \cite{GJV19})
\begin{equation}\label{eq_BCKK}
\mu_X^{B, G}\colon \varinjlim_{Y\subset X, \mathrm{Gcpt}}KK_\ast^G(C_0(Y), B) \to K_\ast(B\rtimes_rG)
\end{equation}
via a group homomorphism
\[
\rho_X\colon \bD_\ast^{B, G}(X) \to  \varinjlim_{Y\subset X, \mathrm{Gcpt}} KK_\ast^G(C_0(Y), B)
\]
which we will define first. We will also show that $\rho_X$ is an isomorphism for any $B$ when $G$ is a discrete group and $X$ is $G$-equivariantly homotopy equivalent to a $G$-CW complex.

To obtain these, for technical reasons, we will use equivariant $E$-theory $E^G$ of Guentner, Higson and Trout (\cite{GHT}, see also \cite[Chapter 2]{HG04}) in place of equivariant $KK$-theory $KK^G$ of Kasparov \cite{Kasparov88}. For our purpose, this is no problem because the canonical natural transformation (\cite[Appendix]{KasparovSkandalis03}, see also \cite[Definition 7.2]{HigsonKasparov})
\begin{equation}\label{eq_KKE}
KK_\ast^G(A, B) \to E_\ast^G(A, B)
\end{equation}
is an isomorphism when $A$ is a proper, nuclear $G$-$C^*$-algebra (more generally when $A\mapsto KK_\ast^G(A, B)$ is half-exact), in particular when $A=C_0(X)$ for a proper $G$-space $X$ (\cite[Corollary A.3, A.4]{KasparovSkandalis03}) and because the Baum--Connes assembly map in equivariant $E$-theory (see \cite{GHT})
\begin{equation}\label{eq_BCE}
\mu^{B, G}_X\colon  \varinjlim_{Y\subset X, \mathrm{Gcpt}}E_\ast^G(C_0(Y), B) \to  K_\ast(B\rtimes_rG)
\end{equation}
is known to be equivalent to the one \eqref{eq_BCKK} in $KK^G$ via \eqref{eq_KKE} (see \cite[Remark A.5]{KasparovSkandalis03}).

We first recall some materials from \cite{GHT}. For any (not necessarily separable) $G$-$C^*$-algebras $A$ and  $B$, we let
\[
\mathfrak{T}(B)= C_b([1, \infty), B)_{\mathrm{Gcont}},  \,\,\, \mathfrak{T}_0(B) = C_0([1, \infty), B).
\]
The asymptotic algebra of $B$ is defined as 
\[
\mathfrak{A}(B)=  \mathfrak{T}(B) / \mathfrak{T}_0(B).
\]
An (equivariant) asymptotic morphism from $A$ to $B$ is an equivariant $\ast$-homomorphism from $A$ to $\mathfrak{A}(B)$. A homotopy of asymptotic morphisms is given by an asymptotic morphism from $A$ to $BI=B\otimes C[0 ,1]$. The set of homotopy equivalence classes of asymptotic morphisms from $A$ to $B$ is denoted by $[[A, B]]_1$. More generally, $[[A, B]]_n$ is the set of $n$-homotopy classes of equivariant $\ast$-homomorphisms from $A$ to $\mathfrak{A}^n(B)$ where $\mathfrak{A}^n$ is the $n$-fold composition of the functor $\mathfrak{A}$ with itself and an $n$-homotopy is given by an equivariant $\ast$-homomorphism from $A$ to $\mathfrak{A}^n(BI)$ (see \cite [Definition 2.6]{GHT}). The set $[[A, B]]$ is defined as the natural inductive limit of $[[A, B]]_n$ (see \cite[Definition 2.7]{GHT}). If $A$ is separable, the set $[[A, B]]$ can be naturally identified as $[[A, B]]_1$ (\cite[Theorem 2.16]{GHT}). 

For any $G$-$C^*$-algebras $A, B, C$, we have a well-defined associative composition law
\[
[[A, B]]\times [[B, C]] \to [[A, C]]
\]
\cite[Proposition 2.12]{GHT}. If $A$ is separable, the composition of asymptotic morphisms $(\phi_t)_{t\in [1, \infty)}\colon A\to \as(B)$ and $(\psi_t)_{t\in [1, \infty)}\colon B \to \as(C)$ can be represented by an asymptotic morphism $\psi_{s(t)} \circ \phi_t \colon A\to \as(C)$ where $(t\mapsto s(t))$ is an increasing function on $[1 ,\infty)$ such that $s(t)\to \infty$ as $t\to \infty$ sufficiently fast \cite{ConnesHigson} or alternatively by an asymptotic morphism $\psi_{t} \circ \phi_{r(t)}\colon A\to \as(C)$ where $(t\mapsto r(t))$ is a continuous function on $[1 ,\infty)$ such that $r(t)\to \infty$ sufficiently slowly \cite[Lemma 2.17, Claim 2.18]{GHT} (both ways of representing the compositions are homotopic).

For any $G$-$C^*$-algebras $A, B, C$, we have a well-defined functor given by the maximal tensor product with the identity 
 \[
\otimes_{\rmax} \mathrm{id}_C\colon [[A, B]] \to [[A\otimes_{\rmax} C, B\otimes_{\rmax}C]]
 \]
\cite[Proposition 4.4]{GHT} and the maximal descent 
\[
[[A, B]] \to [[A\rtimes_{\rmax}G,  B\rtimes_{\rmax}G]]
\]
\cite[Theorem 4.12]{GHT}. 

Let $\hill_G=l^2(\N)\otimes L^2(G)$ and $\Sigma= C_0(\R)\otimes$ be the suspension functor. For any (not necessarily separable) $G$-$C^*$-algebras $A , B$, the equivariant $E$-theory group $E^G(A, B)$ \cite[Definition 6.8]{GHT} is defined as
\[
E^G(A, B)=[[\Sigma A\otimes \Compacts(\hill_G), \Sigma B\otimes \Compacts(\hill_G) ]].
\]
If $A, B$ are separable, we define for $i=0, 1$,
\[
E_i^G(A, B)=E^G(\Sigma^{i}A, B).
\]
For any $G$-$C^*$-algebras $A , B$, $E^G(A, B)$ is an abelian group and the composition law 
\[
E^G(A, B)\times E^G(B, C) \to E^G(A, C)
\]
is bilinear. In this way, we have the additive category $E^G$ whose objects are $G$-$C^*$-algebras and the morphisms groups are $E^G(A, B)$ \cite[Theorem 6.9]{GHT}. There are natural isomorphisms $K_0(B) \cong E(\bC, B)$, $K_1(B) \cong E(\Sigma, B)$ for any $B$ (with trivial $G$-action). \cite[Theorem 6.24]{GHT}.

The Bott periodicity theorem and the half-exactness of the bi-functor $(A, B)\mapsto E^G(A, B)$ are only proved for separable $A ,B$. On the other hand, as we recalled above, the composition law, the (maximal) tensor product and the (maximal) crossed product are defined for general $A, B$. We have a bilinear map
\[
E^G(A, D_1\otimes_{\rmax}B)\times E^G(B\otimes_{\rmax} D_2, C) \to E^G(A\otimes_{\rmax}D_2, D_1\otimes_{\rmax}C)
\]
given by the maximal tensor product with the identity on $\mathrm{id}_{D_2}$ on the first slot and with the identity on $\mathrm{id}_{D_1}$ on the second slot, followed by the composition law. The maximal decent defines a group homomorphism 
\[
j^G_{\rmax}\colon E^G(A, B) \to E(A\rtimes_{\rmax}G,B\rtimes_{\rmax}G)
\]
which is functorial in both variables. When $A$ is a proper algebra (more generally, if $A\rtimes_rG=A\rtimes_{\rmax}G$), we define
\[
j^G_{r}\colon E^G(A, B) \to E(A\rtimes_{r}G,B\rtimes_{r}G)
\]
by the composition of $j^G_{\rmax}$ and the map $E(A\rtimes_{r}G, B\rtimes_{\rmax}G) \to E(A\rtimes_{r}G, B\rtimes_{r}G)$ induced by the quotient map $B\rtimes_{\rmax}G \to B\rtimes_rG$.

For any separable $G$-Hilbert space $H$, any asymptotic morphism $\phi\colon A \to \as(B\otimes \Compacts(H))$ defines an element in $E^G(A, B)$. This is given by the suspension of the tensor product
\[
\phi\otimes \mathrm{id}_{\Compacts(\hill_G)}\colon   A\otimes  \Compacts(\hill_G)\to  \as(B \otimes \Compacts(H)\otimes  \Compacts(\hill_G)) \cong \as(B\otimes \Compacts(\hill_G))
\]
using any isomorphism $\hill\otimes \hill_G\cong \hill_G$.

With these in mind, we do some preparation for constructing a group homomorphism
\[
\rho_X\colon \bD_\ast^{B, G}(X) \to \varinjlim_{Y\subset X, \mathrm{Gcpt}}E_\ast^G(C_0(Y), B).
\]

\begin{lemma} For any $X$-$G$-module $H_X$ and for any $T\in C_{L, c}^*(H_X\otimes B)_{\mathrm{Gcpt}}^G$
\[
\phi T \in C_b([1, \infty), \Compacts(H_X\otimes B)), \,\, [\phi, T] \in C_0([1, \infty), \Compacts(H_X\otimes B))
\]
holds for any $\phi \in C_0(X)$.
\end{lemma}
\begin{proof} The first condition follows since $T_t$ is locally compact for each $t$. The second one holds because for any $T$ in the dense subalgebra $\bC_{L, c}(H_X\otimes B)_{\mathrm{Gcpt}}^G$ which satisfies the second condition (2) in Definition \ref{def_localized}, we have $\lVert[\phi, T_t]\rVert \to 0$ as $t\to \infty$. This follows from Lemma \ref{lem_commutator}.
\end{proof}

An important example of asymptotic morphisms is obtained by the (maximal) tensor product of two asymptotically commuting $\ast$-homomorphisms:

\begin{lemma} Let $A_1, A_2$ and $B$ be $G$-$C^*$-algebras and let 
\[
\phi_i\colon A_i \to C_b([1,\infty), M(B))
\]
(i=1, 2) be equivariant $\ast$-homomorphisms such that 
\[
\phi_1(a_1)\phi_2(a_2) \in C_b([1, \infty), B), \,\, [\phi(a_1), \phi(a_2)]\in C_0([1, \infty), B)
\]
for any $a_1\in A$ and $a_2 \in A_2$. Then, there is a (unique) equivariant asymptotic morphism 
\[
\phi_1\otimes \phi_2\colon A_1\otimes_{\rmax}A_2 \to \mathfrak{A}(B)
\]
such that the image of $a_1\otimes a_2$ is represented by 
\[
\phi_1(a_1)\phi_2(a_2) \in C_b([1, \infty), B). 
\]
\end{lemma}
\begin{proof} This is trivial. We just note here that the image of any equivariant $\ast$-homomorphism from a $G$-$C^*$-algebra consists of $G$-continuous elements.
\end{proof}

\begin{lemma} (c.f.\ \cite[Section 5]{DWW18}, \cite[Construction 6.7.5]{WY2020}) For any $X$-$G$-module $H_X$, the natural map
\[
\pi_X\colon C_0(X) \to C_b([1,\infty), \Linears(H_X\otimes B)) 
\]
and the inclusion 
\[
\iota \colon C_{L, c}^*(H_X\otimes B)_{\mathrm{Gcpt}}^G \subset C_b([1,\infty), \Linears(H_X\otimes B) )
\]
induce an asymptotic morphism
\[
\pi_X\otimes \iota \colon C_0(X)\otimes C_{L, c}^*(H_X\otimes B)_{\mathrm{Gcpt}}^G \to \mathfrak{A}(\Compacts(H_X\otimes B))
\]
such that the image of $\phi \otimes T$ is represented by 
\[
\phi T \in C_b([1, \infty), \Compacts(H_X\otimes B)). 
\]
\end{lemma}
\begin{proof} This follows from the previous two lemmas.
\end{proof}

We recall from the previous section that the right-regular representation
\[
\rho\colon  C_{b}([1, \infty), \Linears(H_X\otimes B))_{\mathrm{Gcont}}\rtimes_rG \to C_{b}([1, \infty), \Linears(\tilde H_X\otimes B))
\]
restricts to 
\[
\rho\colon RL^*_c(H_X\otimes B)\rtimes_rG \to C_{L, c}^*(\tilde H_X\otimes B)_{\mathrm{Gcpt}}^G.
\]

Thus, we obtain the following asymptotic morphism.

\begin{proposition}\label{prop_asympXG} For any $X$-$G$-module $H_X$, the natural map
\[
\pi_X\colon C_0(X) \to C_b([1,\infty), \Linears(\tilde H_X\otimes B)) 
\]
and the right-regular representation 
\[
\rho \colon RL^*_c(H_X\otimes B)\rtimes_rG  \to C_b([1,\infty), \Linears(\tilde H_X\otimes B)) 
\]
induce an asymptotic morphism
\[
\pi_X\otimes \rho \colon C_0(X)\otimes (RL^*_c(H_X\otimes B)\rtimes_rG )       \to \mathfrak{A}(\Compacts(\tilde H_X\otimes B))
\]
such that the image of $\phi \otimes T$ is represented by 
\[
\phi \rho(T) \in C_b([1, \infty), \Compacts(\tilde H_X\otimes B)). 
\]
\end{proposition}

\begin{definition} For any $X$-$G$-module $H_X$, the element
\[
[\pi_X\otimes \rho]  \in E^G(C_0(X)\otimes(RL^*_c(H_X\otimes B)\rtimes_rG ), B)
\]
is defined by the asymptotic morphism $\pi_X\otimes \rho$ in Proposition \ref{prop_asympXG}.
\end{definition}

\begin{definition} We define a group homomorphism 
\[
\rho_X \colon K_i(RL^*_c(H_X\otimes B)\rtimes_rG) \to E_i^G(C_0(X), B)
\]
for $i=0, 1$, by sending 
\[
K_i(RL^*_c(H_X\otimes B)\rtimes_rG) \cong  E(\Sigma^i , RL^*_c(H_X\otimes B)\rtimes_rG)  
\]
to $E_i^G(C_0(X), B)$ by the composition with the class $[\pi_X\otimes \rho]$ under the composition law
\[
E^G(\Sigma^i , RL^*_c(H_X\otimes B)\rtimes_rG) \times  E^G(C_0(X)\otimes(RL^*_c(H_X\otimes B)\rtimes_rG ), B)
\]
\[
\to E^G(\Sigma^iC_0(X), B).
\]
\end{definition}

\begin{lemma}\label{lem_functorialB} The group homomorphism $\rho_X$ is natural with respect to any $G$-equivariant $\ast$-homomorphism $\pi\colon B_1\to B_2$ in a sense that the following diagram commutes
\begin{equation*}
\xymatrix{ 
K_\ast(RL^*_c(H_X\otimes B_1)\rtimes_rG)  \ar[d]^{\pi_\ast}  \ar[r]^-{\rho_X } &  E_\ast^G(C_0(X), B_1)   \ar[d]^{\pi_\ast}  \\
K_\ast(RL^*_c(H_X\otimes B_2)\rtimes_rG) \ar[r]^-{\rho_X } & E_\ast^G(C_0(X), B_2).    
   }
\end{equation*}
\end{lemma}
\begin{proof} This follows because the following diagram commutes
\begin{equation*}
\xymatrix{ 
  C_0(X)\otimes ( RL^*_c(H_X\otimes B_1)\rtimes_rG)  \ar[r]^-{ \pi_X\otimes \rho}  \ar[d]^-{\mathrm{id}_{C_0(X)}\otimes \pi} & \mathfrak{A}(\Compacts(\tilde H_X\otimes B_1))  \ar[d]^{\pi}  \\
 C_0(X)\otimes ( RL^*_c(H_X\otimes B_2)\rtimes_rG )  \ar[r]^-{ \pi_X\otimes \rho}    &  \mathfrak{A}(\Compacts(\tilde H_X\otimes B_2)).
   }
\end{equation*}
\end{proof}

Now suppose that a $G$-equivariant continuous map $f\colon X\to Y$ is proper, so that we have $f^\ast\colon C_0(Y) \to C_0(X)$. It defines
\[
f_\ast = (f^\ast)^\ast \colon E_\ast^G(C_0(X), B) \to E_\ast^G(C_0(Y), B).
\]
Let $H_X$ and $H_Y$ be an $X$-$G$-module and a $Y$-$G$-module respectively and suppose that there is an equivariant continuous cover $(V_t\colon H_X\to H_Y)_{t\in [1\infty)}$ of $f$. 

\begin{lemma}\label{lem_functorialX} The following diagram commutes for any proper $f\colon X\to Y$ and for any equivariant continuous cover $(V_t\colon H_X\to H_Y)_{t\in [1\infty)}$ of $f$:
\begin{equation*}
\xymatrix{ 
K_\ast(RL^*_c(H_X\otimes B)\rtimes_rG)    \ar[d]^{\mathrm{Ad}_{V_t, \ast}}  \ar[r]^-{\rho_X } &   E_\ast^G(C_0(X), B) \ar[d]^{f_\ast}    \\
K_\ast(RL^*_c(H_Y\otimes B)\rtimes_rG)     \ar[r]^-{\rho_Y} &   E_\ast^G(C_0(Y), B).
   }
\end{equation*}
\end{lemma}
\begin{proof} Recall that $V_t\colon H_X\to H_Y$ is a cover of $f$ if and only if it is a cover of the identity when we view $H_X$ as a $Y$-module $(H_X)_Y$ via $f^\ast$. It is clear from the definition of $\rho$ that the diagram commutes when $V_t$ is the identity map $H_X \to (H_X)_Y$. Thus, it is enough to show the claim when $X=Y$ and $(V_t\colon H^0_X\to H^1_X)_{t\in [1\infty)}$ is a cover of the identity map on $X$. For $i=0,1$, we have
\[
\pi_X^i\otimes \rho^i \colon C_0(X)\otimes RL^*_c(H^i_X\otimes B)\rtimes_rG    \to \as(\Compacts(\tilde H_X^i\otimes B)),  
\]
defined by the two asymptotically commuting $\ast$-homomorphisms
\[
\pi^i_X\colon C_0(X) \to C_b([1,\infty), \Linears(\tilde H^i_X\otimes B)),
\]
\[
\rho^i \colon RL^*_c(H^i_X\otimes B)\rtimes_rG  \to C_b([1,\infty), \Linears(\tilde H^i_X\otimes B)). 
\]
We consider $H_X^0\oplus H_X^1$ and $\tilde H_X^0\oplus \tilde H_X^1$. We show that two asymptotic morphisms
\[
\pi_X^0\otimes \rho^0, \, \pi_X^1\otimes (\rho^1\circ \mathrm{Ad}_{V_t}) \colon C_0(X)\otimes RL^*_c(H^0_X\otimes B)\rtimes_rG    \to \as(\Compacts((\tilde H_X^0\otimes B \oplus \tilde H_X^1\otimes B))
\]
are homotopic. Note that $V_t\colon H^0_X\to H^1_X$ defines an $G$-equivariant isometry $\tilde V_t=V_t\otimes 1 \colon \tilde H^0_X\to \tilde H^1_X$ in $C_b([1,\infty),   \Linears(\tilde H^0_X\otimes B \oplus \tilde H^1_X\otimes B ))$ which conjugates $\rho^0$ to $\rho^1\circ \mathrm{Ad}_{V_t}$. Furthermore, since $V_t$ is a cover of the identity on $X$, it satisfies 
\[
\pi_X^1(\phi) \tilde V_t - \tilde V_t\pi_X^0(\phi)  \in C_0([1,\infty),   \Linears(\tilde H^0_X\otimes B \oplus \tilde H^1_X\otimes B )),
\]
and so
\[
 \pi_X^1(\phi) \tilde V_t\tilde V_t^\ast - \tilde V_t\pi_X^0(\phi)\tilde V_t^\ast  \in C_0([1,\infty),   \Linears(\tilde H^0_X\otimes B \oplus \tilde H^1_X\otimes B )).
\]
From this, we see that the isometry $\tilde V_t$ asymptotically conjugates $\pi_X^0\otimes \rho^0$ to $ \pi_X^1\otimes (\rho^1\circ \mathrm{Ad}_{V_t})=  \pi_X^1\otimes (\tilde V_t\rho^0 \tilde V_t^\ast)$ in a sense that 
\[
\mathrm{Ad}_{\tilde V_t} ( \pi_X^0\otimes \rho^0   ) =  \pi_X^1\otimes (\rho^1\circ \mathrm{Ad}_{V_t})
\]
in $\as(\Compacts((\tilde H_X^0\otimes B \oplus \tilde H_X^1\otimes B))$. It follows that the two asymptotic morphisms are homotopic. The claim follows from this.

\end{proof}

Thanks to the previous lemma, and since $X\to E^G(C_0(X), B)$ is functorial for proper $f\colon X\to Y$, the following definition is well-defined.

\begin{definition} For any $G$-compact proper $G$-space $X$, we define a group homomorphism 
\[
\rho_X \colon \bD^{B, G}_\ast(X) \to E_\ast^G(C_0(X), B)
\]
by
\[
\rho_X\colon  K_\ast(RL^*_c(H_X\otimes B)\rtimes_rG ) \to E_\ast^G(C_0(X), B)
\]
where $H_X$ is the chosen universal $X$-$G$-module. More generally for any proper $G$-space $X$, we define a group homomorphism 
\[
\rho_X \colon \bD^{B, G}_\ast(X) \to \varinjlim_{Y\subset X, \mathrm{Gcpt}}E_\ast^G(C_0(Y), B)
\]
so that the following diagram commutes:
\begin{equation*}
\xymatrix{ 
\bD^{B, G}_\ast(X) \cong  \varinjlim_{Y\subset X, \mathrm{Gcpt}}\bD^{B, G}(Y)  \ar[r]^{\rho_X } &   \varinjlim_{Y\subset X, \mathrm{Gcpt}}E_\ast^G(C_0(Y), B)    \\
\bD^{B, G}_\ast(Y)  \ar[u]^{}  \ar[r]^{\rho_Y } &   E_\ast^G(C_0(Y), B) \ar[u]^{}.
   }
\end{equation*}
\end{definition}

\begin{theorem} The group homomorphisms 
\[
\rho_X \colon \bD^{B, G}_\ast(X) \to \varinjlim_{Y\subset X, \mathrm{Gcpt}}E_\ast^G(C_0(Y), B)
\] 
define a natural transformation of functors from the category $\mathcal{PR}^G$ of (second countable, locally compact) proper $G$-spaces to the category $\mathcal{GA}$ of graded abelian groups. Furthermore, the transformation is natural with respect to a  $G$-equivariant $\ast$-homomorphism $\pi\colon B_1\to B_2$ in a sense that the following diagram commutes
\begin{equation*}
\xymatrix{ 
\bD^{B_1, G}_\ast(X) \ar[d]^{\pi_\ast}  \ar[r]^-{\rho_X} &   \varinjlim_{Y\subset X, \mathrm{Gcpt}}E_\ast^G(C_0(Y), B_1)    \ar[d]^{\pi_\ast}  \\
\bD^{B_2, G}_\ast(X) \ar[r]^-{\rho_X } &   \varinjlim_{Y\subset X, \mathrm{Gcpt}}E_\ast^G(C_0(Y), B_2).
   }
\end{equation*}
\end{theorem}
\begin{proof} The first assertion follows from Lemma \ref{lem_functorialX}. The second assertion follows from Lemma \ref{lem_functorialB}.
\end{proof}

The next goal is to show the following:
\begin{theorem}\label{thm_forget_factor} The forget-control map $\cf\colon \bD_\ast^{B, G}(X) \to K_\ast(B\rtimes_rG)$ factors through the Baum--Connes assembly map $\mu_X^{B, G}$ via $\rho_X$ in a sense that the following diagram commutes:
\begin{equation*}
\xymatrix{ \bD_\ast^{B, G}(X)   \ar[dr]^{\rho_X}  \ar[rr]^{\cf} & &  K_\ast(B\rtimes_rG)  \\
   &  \varinjlim_{Y\subset X, \mathrm{Gcpt}}E_\ast^G(C_0(Y), B). \ar[ur]^{\mu^{B, G}_X}  & 
   }
   \end{equation*}
\end{theorem}

For this, we need some preparation. Recall that for any $G$-Hilbert $B$-module $\E$, a Hilbert $B\rtimes_rG$-module $\E\rtimes_rG$ is defined \cite[Defintion 3.8]{Kasparov88}. For simplicity (for our purpose, this is enough) we will only consider $\E$ of the form $H\otimes B$ where $H$ is a $G$-Hilbert space and in this case $\E\rtimes_rG=H\otimes B \rtimes_rG $ is canonically identified as the Hilbert $B\rtimes_rG$-module $H\otimes (B\rtimes_rG)$. Kasparov defined for any $G$-equivariant $\ast$-homomorphism $\pi\colon A\to \Linears(\E)$, its crossed product $\pi\rtimes_r1\colon A\rtimes_rG \to \Linears(\E\rtimes_rG)$. In the case when $\E=H\otimes B$, the $\ast$-homomorphism $\pi\rtimes_r1$ is defined by sending $a \in A$ to
\[
\pi(a) \in  \Linears(\E) \subset  \Linears(\E\rtimes_rG)
\]
and $g \in G$ to 
\[
( g\otimes u_g\colon v\otimes f \mapsto gv\otimes u_gf  )  \in \Linears(H\otimes B \rtimes_rG)
\]
where $u_g\in M(B\rtimes_rG)$ is the unitary corresponding to $g$.

\begin{lemma}\label{lem_isom_Uc}(c.f.\ \cite[Proposition 5.2]{Nishikawa19}) Let $X$ be a $G$-compact proper $G$-space. Let $H_X$ be an $X$-$G$-module and let $\tilde H_X$ be the $X$-$G$-module as in Definition \ref{def_tildeH}. Consider 
\[
\pi_X\colon C_0(X) \to \Linears(\tilde H_X\otimes B),
\]
the structure map for the $X$-$G$-module $\tilde H_X\otimes B$ and consider its Kasparov's crossed product (see the explanation right before the lemma)
\[
\pi_X\rtimes_r1 \colon C_0(X)\rtimes_rG \to \Linears(\tilde H_X\otimes B\rtimes_rG). 
\]
Let $p_c(g)=\Delta(g)^{-1/2}cg(c)$ be the cut-off projection in $C_0(X)\rtimes_rG$ for a cut-off function $c \in C_c(X)$ and let
\[
\rho\colon \Linears(H_X\otimes B)_{\mathrm{Gcont}}\rtimes_rG \to \Linears(\tilde H_X\otimes B) \subset  \Linears(\tilde H_X\otimes B\rtimes_rG)
\]
be the right-regular representation as in Definition \ref{def_rightreg}. There is an adjointable isometry
\[
U_c\colon H_X\otimes B\rtimes_rG \to  \tilde H_X\otimes B\rtimes_rG
\]
of Hilbert $B\rtimes_rG$-modules such that 
\[
U_cU_c^\ast =  (\pi_X\rtimes_r1)(p_c) 
\]
and such that the c.c.p.\ map
\[
U_c^\ast \rho  U_c \colon \Linears(H_X\otimes B)_{\mathrm{Gcont}}\rtimes_rG \to \Linears(H_X\otimes B\rtimes_rG)
\]
is identified as the c.c.p.\ map on $\Linears(H_X\otimes B)_{\mathrm{Gcont}}\rtimes_rG$ (naturally viewed as a subalgebra of $\Linears(H_X\otimes B\rtimes_rG)$ via Kasparov's crossed product) defined by the $G$-equivariant c.c.p.\ map
\[
 \Linears(H_X\otimes B)\ni T  \mapsto T'= \int_{g\in G} g(c)Tg(c) d\mu_G(g)  \in  \Linears(H_X\otimes B).
\]
\end{lemma}
\begin{proof} An isometry $U_c\colon H_X\otimes B\rtimes_rG \to \tilde H_X\otimes B\rtimes_rG$ is defined by sending $v \in H_X\otimes B\rtimes_rG$ to 
\[
(G\ni h\mapsto  \Delta(h)^{-1/2} (\pi_X\rtimes_r1)(c)(h\otimes u_h) v  \in H_X\otimes B\rtimes_rG  )
\]
 in $L^2(G)\otimes H_X\otimes B\rtimes_rG$. Its adjoint $U_c^\ast$ is defined by sending 
\[
(G\ni h \mapsto  v_h \in H_X\otimes B\rtimes_rG  ) \in L^2(G)\otimes H_X\otimes B\rtimes_rG
\]
to
 \[
 \int_{h \in G}  \Delta(h)^{-1/2} (h^{-1}\otimes u_{h^{-1}})((\pi_X\rtimes_r1)(c) v_h)  d\mu_G(h)
 \]
 which converges weakly in $H_X\otimes B\rtimes_rG$. We can first see that $U_c$ is well-defined and $\lVert U_c\rVert=1$ (as an a-priori not necessarily adjointable map) and from which the weak convergence of the formula of $U_c^\ast$ can be deduced. The equalities $U_c^\ast U_c=1$ and $U_cU_c^\ast=(\pi_X\rtimes_r1)(p_c)$ can be checked directly.
 
 Now, the following equalities can be checked directly,
 \[
 U_c^\ast \rho(T)U_c = \int_{h\in G} h(c)Th(c) d\mu_G(h) \in \Linears(H_X\otimes B) \subset \Linears(H_X\otimes B\rtimes_rG)
 \]
for any $T\in \Linears(H_X\otimes B)$ and 
\[
 U_c^\ast \rho_g U_c = g\otimes u_g \in \Linears(H_X\otimes B\rtimes_rG)
\]
for any $g\in G$.
\end{proof}

\begin{lemma}\label{lem_compute} Let $X$ be a $G$-compact proper $G$-space. Let $H_X$ be any $X$-$G$-module. Consider the element
\[
[\pi_X\otimes \rho]  \in E^G(C_0(X)\otimes(RL^*_c(H_X\otimes B)\rtimes_rG ), B).
\]
Let 
\[
j^G_r([\pi_X\otimes \rho]) \in    E(C_0(X)\rtimes_rG \otimes(RL^*_c(H_X\otimes B)\rtimes_rG ), B\rtimes_rG)
\]
be its reduced crossed product (which is defined since the first variable is a proper $G$-$C^*$-algebra). Let 
\[
[p_c] \in E(\bC, C_0(X)\rtimes_rG)
\]
be the element corresponding to the cut-off projection $p_c$ in $C_0(X)\rtimes_rG$ for a cut-off function $c$ on $X$. Then, the class
\[
j^G_r([\pi_X\otimes \rho]) \circ [p_c]  \in   E(RL^*_c(H_X\otimes B)\rtimes_rG, B\rtimes_rG)
\]
is equal to the class associated to the natural inclusion map
\[
RL^*_c(H_X\otimes B)\rtimes_rG \to C_b([1, \infty), \Compacts(H_X\otimes B)\rtimes_rG).
\]
\end{lemma}
\begin{proof} The crossed product $j^G_r([\pi_X\otimes \rho])$ is represented by the product of the two asymptotically commuting representations
\[
\pi_X\rtimes_r1 \colon C_0(X)\rtimes_rG \to \Linears(\tilde H_X\otimes B\rtimes_rG) \subset C_b([1, \infty),  \Linears(\tilde H_X\otimes B\rtimes_rG))
\]
and
\[
\rho\colon  RL^*_c(H_X\otimes B)\rtimes_rG  \to C_b([1, \infty),  \Linears(\tilde H_X\otimes B)) \subset C_b([1, \infty),  \Linears(\tilde H_X\otimes B\rtimes_rG)).
\]
The class $j^G_r([\pi_X\otimes \rho])\circ [p_c]$ is represented by the asymptotic morphism from $RL^*_c(H_X\otimes B)\rtimes_rG$ to $\as(  \Compacts(\tilde H_X\otimes B\rtimes_rG))$ which is represented by the c.c.p.\ map
\[
\pi_X\rtimes_r1 (p_c) \rho \pi_X\rtimes_r1 (p_c) \colon  RL^*_c(H_X\otimes B)\rtimes_rG  \to C_b([1, \infty),  \Compacts(\tilde H_X\otimes B\rtimes_rG)).
\]
In view of Lemma \ref{lem_isom_Uc}, this map can be identified as the c.c.p.\ map 
\[
RL^*_c(H_X\otimes B)\rtimes_rG  \to C_b([1, \infty),  \Compacts(H_X\otimes B\rtimes_rG))
\]
defined as the natural inclusion 
\[
RL^*_c(H_X\otimes B)\rtimes_rG  \to C_b([1, \infty),  \Compacts(H_X\otimes B\rtimes_rG))
\]
preceded by the c.c.p.\ map on $RL^*_c(H_X\otimes B)\rtimes_rG$ induced by the $G$-equivariant c.c.p.\ map 
\[
T_t \mapsto  \int_{h\in G} h(c)T_th(c) d\mu_G(h)
\]
on $RL_c^\ast(H_X\otimes B)$. Thus, to prove our claim, we just need to show that for any $T \in RL^*_c(H_X\otimes B)$, 
\[
\lVert T_t -  \int_{h\in G} h(c)T_th(c) d\mu_G(h)\rVert \to 0 
\] 
as $t\to \infty$. We may assume $T$ has uniform compact support in $X$. In this case, $h(c)T_t=0$ for $h\in G\backslash K$ for some compact subset $K$ of $G$ and $\lVert[h(c), T_t]\rVert\to 0$ uniformly in $h\in K$ as $t\to \infty$. We see the last assertion holds so we are done. 
\end{proof}

\begin{lemma}\label{lem_factor_Gcompact} Let $X$ be a $G$-compact proper $G$-space and $H_X$ be an $X$-$G$-module. The forget-control map $\cf\colon K_\ast(RL^*_c(H_X\otimes B)\rtimes_rG) \to K_\ast(B\rtimes_rG)$ factors through $E_\ast^G(C_0(X), B)$ via $\rho_X$. That is, the following diagram commutes,
\begin{equation*}
\xymatrix{ K_\ast(RL^*_c(H_X\otimes B)\rtimes_rG)    \ar[dr]^{\rho_X}  \ar[rr]^{\cf} & &  K_\ast(B\rtimes_rG)  \\
   & E_\ast^G(C_0(X), B) \ar[ur]^{\mu^G_X}.   & 
   }
\end{equation*}
\end{lemma}
\begin{proof} Let $[\phi] \in K_i(RL^*_c(H_X\otimes B)\rtimes_rG) = E(\Sigma^{i}, RL^*_c(H_X\otimes B)\rtimes_rG)$. Using the functoriality of the crossed product functor and the composition law, we see that the assembly map $\mu^G_X$ sends $\rho_X([\phi])$ to the composition of $[\phi]$ with the element
 \[
j^G_r([\pi_X\otimes \rho])\circ [p_c] \in   E(RL^*_c(H_X\otimes B)\rtimes_rG, B\rtimes_rG)
 \]
 under the composition law
 \[
 E(\Sigma^i, RL^*_c(H_X\otimes B)\rtimes_rG) \times  E(RL^*_c(H_X\otimes B)\rtimes_rG, B\rtimes_rG) \to  E(\Sigma^i, B\rtimes_rG).
 \]
 By Lemma \ref{lem_compute},  the element $j^G_r([\pi_X\otimes \rho])\circ [p_c]$ is represented by the natural inclusion 
 \[
 RL^*_c(H_X\otimes B)\rtimes_rG \to C_b([1, \infty), \Compacts(H_X\otimes B)\rtimes_rG).
 \]
That is $j^G_r([\pi_X\otimes \rho])\circ [p_c]$ is represented by a continuous family $(\mathrm{ev}_{t})_{t\in [1, \infty)}$ of $\ast$-homomorphisms (evaluation at $t$). On the other hand, such a continuous family of $\ast$-homomorphisms is homotopic (as an asymptotic morphism) to the constant one $\mathrm{ev}_{1}$. It is now clear that the element  $\mu^G_X \circ \rho_X ([\phi])$ coincides in $E_\ast(\bC, B\rtimes_rG)$ with the one represented by the composition of $\phi$ with the evaluation map $RL^*_c(H_X\otimes B)\rtimes_rG \to \Compacts(H_X)\otimes B\rtimes_rG$ at $t=1$ so we are done.
\end{proof}

\begin{proof}[Proof of Theorem \ref{thm_forget_factor}]
The theorem follows from Lemma \ref{lem_factor_Gcompact}. 
\end{proof}

As the last thing in this section, we prove that the natural transformation
\[
\rho_X\colon \bD_\ast^{B, G}(X)  \to \varinjlim_{Y\subset X, \mathrm{Gcpt}}E_\ast^G(C_0(Y), B)
\]
is an isomorphism when $G$ is discrete and if $X$ is $G$-equivariantly homotopic to a $G$-CW complex. Recall that for any open subgroup $H$ of $G$ and for any proper $H$-space $Y$, if $X$ is the balanced product $G\times_HY$, we have a natural isomorphism (see Theorem \ref{thm_coeff} (2))
\[
\bD_\ast^{B, G}(X) \cong \bD_\ast^{B, H}(Y).
\]
We also have a natural isomorphism (\cite[Lemma 12.11]{GHT}, \cite[Proposition 5.14]{CE01})
\[
E_\ast^G(C_0(X), B) \cong E_\ast^H(C_0(Y), B).
\]
The rightward map is obtained by the restriction to the $H$-$C^*$-subalgebra $C_0(Y) \subset C_0(X)$. 

\begin{lemma} The natural transformation $\rho_X$ commutes with the induction. That is, for any open subgroup $H$ of $G$, for any proper $H$-space $Y$ and for $X=G\times_HY$, the following diagram commutes.
\begin{equation*}
\xymatrix{ \bD_\ast^{B, G}(X)\ar[d]^-{\cong} \   \ar[r]^-{\rho_X}   &  E^G(C_0(X), B) \ar[d]^-{\cong} \\
\bD_\ast^{B, H}(Y)   \ar[r]^-{\rho_Y}   &  E^H(C_0(Y), B). \\
   }
   \end{equation*}
\end{lemma}
\begin{proof} Recall that the isomorphism $\bD_\ast^{B, H}(Y) \cong  \bD_\ast^{B, G}(X)$ is obtained by the canonical inclusion
\[
RL^*_c(H_Y\otimes B)\rtimes_rH \to RL^*_c(H_X\otimes B)\rtimes_rG.
\]
With this in mind, the claim is direct to check. 
\end{proof}

\begin{theorem}\label{thm_discrete_isom} Let $G$ be a countable discrete group and $X$ be a proper $G$-space which is $G$-equivariantly homotopy equivalent to a $G$-CW complex. Then, the group homomorphism 
\[
\rho_X\colon \bD_\ast^{B, G}(X)  \to \varinjlim_{Y\subset X, \mathrm{Gcpt}}E_\ast^G(C_0(Y), B) \cong  \varinjlim_{Y\subset X, \mathrm{Gcpt}}KK_\ast^G(C_0(Y), B)
\]
is an isomorphism for any separable $G$-$C^*$-algebra $B$.
\end{theorem}
\begin{proof} Since both functors satisfy the axioms (1)-(5) in Theorem \ref{thm_coeff} and since $\rho_X$ is a natural transformation which commutes with the induction from a finite subgroup (more generally from an open subgroup), it is enough to prove that $\rho_X$ is an isomorphism when $G$ is a finite group $H$ and $X$ is a point. On the other hand, we have the following commutative diagram
\begin{equation*}
\xymatrix{ \bD_\ast^{B, H}(\mathrm{point})   \ar[dr]^{\rho_X}  \ar[rr]^{\cf} & &  K_\ast(B\rtimes_rH)  \\
   & E_\ast^H(\bC, B). \ar[ur]^{\mu_r^{B, H}}  & 
   }
   \end{equation*}
   We know that the assembly map $\mu_r^{B, H}$ is an isomorphism for any finite group $H$. We also know that the forget-control map is an isomorphism for any finite group $H$ when $X$ is a point (see Proposition \ref{prop_pointcase}). Thus, $\rho_X$ is an isomorphism. Of course, we may also directly show that $\rho_X$ is an isomorphism.
\end{proof}

\begin{theorem} Let $G$ be a countable discrete group $G$ and $X$ be a proper $G$-space which is $G$-equivariantly homotopy equivalent to a $G$-CW complex. The forget-control map $\cf\colon \bD_\ast^{B, G}(X) \to K_\ast(B\rtimes_rG)$ is naturally equivalent to the Baum--Connes assembly map $\mu_{X}^{B, G}$ for any separable $G$-$C^*$-algebra $B$.
\end{theorem}
\begin{proof} This follows from Theorem \ref{thm_forget_factor} and Theorem \ref{thm_discrete_isom}. 
\end{proof}

 Let $RL^0_c(H_X\otimes B)$ be the kernel of the evaluation map $\mathrm{ev}_1$ on $RL^*_c(H_X\otimes B)$. The short exact sequence
 \[
 0 \to RL^0_c(H_X\otimes B) \to RL^*_c(H_X\otimes B) \to \Compacts(H_X)\otimes B \to 0
 \] 
 admits a $G$-equivariant c.c.p.\ splitting (by extending constantly and by multiplying a bump function) and thus it descends to the short exact sequence
 \[
  0 \to RL^0_c(H_X\otimes B)\rtimes_rG \to RL^*_c(H_X\otimes B)\rtimes_rG \to (\Compacts(H_X)\otimes B)\rtimes_rG \to 0.
  \]
   
 \begin{corollary} For any countable discrete group $G$, the Baum--Connes assembly map $\mu^{B, G}_r$ is an isomorphism if and only if
 \[
 K_\ast(RL^0_c(H_X\otimes B)\rtimes_rG)=0
 \]
for a universal (or ample) $X$-$G$-module $H_X$ for $X=\EG$.
 \end{corollary}

We also have the following short exact sequence
 \[
  0 \to (RL^0_c(H_X)\otimes B)\rtimes_rG \to( RL^*_c(H_X)\otimes B)\rtimes_rG \to (\Compacts(H_X)\otimes B)\rtimes_rG \to 0.
  \]
 Recall that the natural transformation $\mathbb{D}^{\otimes B, G}_\ast(X) \to   \mathbb{D}^{B, G}_\ast(X)$ is an isomorphism if $G$ is discrete and $X$ is $G$-equivariantly homotopy equivalent to a proper $G$-CW complex (Theorem \ref{thm_discrete_natural_isom}). Hence, we also have the following:
 
 \begin{theorem}\label{thm_otimesB} Let $G$ be a countable discrete group $G$ and $X$ be a proper $G$-space which is $G$-equivariantly homotopy equivalent to a $G$-CW complex. The forget-control map $\cf\colon \bD_\ast^{\otimes B, G}(X) \to K_\ast(B\rtimes_rG)$ is naturally equivalent to the Baum--Connes assembly map $\mu_{X}^{B, G}$ for any separable $G$-$C^*$-algebra $B$.
\end{theorem}

 \begin{corollary}\label{cor_discrete_N} For any countable discrete group $G$, the Baum--Connes assembly map $\mu^{B, G}_r$ is an isomorphism if and only if
 \[
 K_\ast((RL^0_c(H_X)\otimes B)\rtimes_rG)=0
 \]
for a universal (or ample) $X$-$G$-module $H_X$ for $X=\EG$.
 \end{corollary}
 


\section{$\rho_X$ is an isomorphism, part I}

We begin our proof of the following:

\begin{theorem*} Let $G$ be a locally compact group. The natural transformation
\[
\rho_X\colon  \bD_\ast^{B, G}(X)  \to \varinjlim_{Y\subset X, \mathrm{Gcpt}}E_\ast^G(C_0(Y), B)  \cong  \varinjlim_{Y\subset X, \mathrm{Gcpt}}KK_\ast^G(C_0(Y), B) 
\]
is an isomorphism for any proper $G$-space $X$ and for any separable $G$-$C^*$-algebra $B$.
\end{theorem*}

The proof will be given over the next four sections. We will entirely focus on the case when $X$ is $G$-compact. Note that the general case follows from this. Here is a rough idea of the proof. Let $X$ be a $G$-compact proper $G$-space. We first show (in this section) that the right-regular representation induces an isomorphism
\[
\rho \colon K_\ast(RL^*_u(H_X\otimes B)\rtimes_rG) \cong  K_\ast(C^*_{L, u}(\tilde H_X\otimes B)^G)
\] 
whenever $H_X$ is a universal $X$-$G$-module. Then, we prove an isomorphism 
\[
K_\ast(C^*_{L, u}(\tilde H_X\otimes B)^G) \cong   E_\ast^G(C_0(X), B)  \cong KK_\ast^G(C_0(X), B) 
\]
following the idea of \cite{DWW18}.

Let $X$ be a $G$-compact proper $G$-space and $H_X$ be an $X$-$G$-module. We recall that the equivariant Roe algebra $C^*(H_X\otimes B)^G=C^*(H_X\otimes B)^G_{\mathrm{Gcpt}}$ is defined as the completion of the $\ast$-algebra $\bC(H_X\otimes B)^G$ consisting of $G$-equivariant, locally compact operators in $\Linears(H_X\otimes B)$ that are properly supported. Note if $X_0$ is any compact subset of $X$ such that $GX_0=X$ then a $G$-equivariant operator $T$ on $H_X\otimes B$ is properly supported if and only if there is a compact subset $X_1$ of $X$ so that $\chi_{X_0}T=\chi_{X_0}T\chi_{X_1}$ and  $T\chi_{X_0}=\chi_{X_1}T\chi_{X_0}$. We also recall that the localized equivariant Roe algebra $C^*_{L, u}(H_X\otimes B)^G=C^*_{L, u}(H_X\otimes B)^G_{\mathrm{Gcpt}}$ is the norm completion of the $\ast$-subalgebra $\bC_{L, u}(H_X\otimes B)^G$ of $C_{b, u}([1, \infty), \Linears(H_X\otimes B))$ consisting of $\bC(H_X\otimes B)^G$-valued functions $T_t$ with $\prop(T_t)\to 0$ with respect to a (any) fixed metric $d$ on $X^+$.

Let $H_X$ be any $X$-$G$-module. We introduce several $C^*$-subalgebras of $C_{b, u}([1, \infty), \Linears(H_X\otimes B))$ containing the localized equivariant Roe algebra. Let 
\[
\pi\colon C_0(X) \to  \Linears(H_X\otimes B)
\] 
be the structure map for the $X$-$G$-module $H_X\otimes B$. We let $\cC(\pi)^G$ to be the $C^*$-subalgebra of $\Linears(H_X\otimes B)$ consisting of $G$-equivariant, locally compact operators. We have an inclusion 
\[
C^*(H_X\otimes B)^G \subset \cC(\pi)^G.
\]
\begin{itemize}
\item $\cC_L(H_X\otimes B)^G$ is the $C^*$-subalgebra of $C_{b, u}([1, \infty), C^*(H_X\otimes B)^G)$ consisting of $T_t$ such that $\limt\lVert[\phi, T_t]\rVert= 0$ for any $\phi \in C_0(X)$.
\item $\cC_L(\pi)^G$ is the $C^*$-subalgebra of $C_{b, u}([1, \infty), \cC(\pi)^G)$ consisting of functions $T_t$ such that $\limt\lVert[\phi, T_t]\rVert= 0$ for any $\phi \in C_0(X)$.
\end{itemize}

We have inclusions
\[
C^*_{L, u}(H_X\otimes B)^G \subset \cC_L(H_X\otimes B)^G \subset  \cC_L(\pi)^G.
\]
Let $c$ be a cut-off function on $X$. A $G$-equivariant u.c.p.\ map $\psi_c$ on $\Linears(H_X\otimes B)$ defined as
\[
\psi_c\colon T\mapsto \int_{g\in G} g(c)T g(c) d\mu_G(g)
\]
extends to a $G$-equivariant u.c.p.\ map on $C_{b, u}([1, \infty), \Linears(H_X\otimes B))$. Note that $\psi_c$ sends any $G$-equivariant, locally compact operator to a $G$-equivariant, locally compact operator which is properly supported. Thus, $\psi_c$ maps $\cC(\pi)^G$ to $\bC(H_X\otimes B)^G\subset C^*(H_X\otimes B)^G$. Moreover, $\psi_c$ sends $\cC_L(\pi)^G$ to $\cC_L(H_X\otimes B)^G$.

\begin{lemma}\label{lem_properly} Let $T \in \cC_L(\pi)^G$. The following are equivalent: 
\begin{enumerate}
\item $\lim_{t\to \infty} \lVert T_t - \psi_c(T_t)\rVert = 0$ for any cut-off function $c$ on $X$.
\item $\lim_{t\to \infty} \lVert T_t - \psi_c(T_t)\rVert = 0$ for some cut-off function $c$ on $X$.
\item There is $S$ in $\cC_L(H_X\otimes B)^G$ such that $\lim_{t\to \infty}\lVert T_t-S_t\rVert= 0$ and $S$ is properly supported in a sense that for any compact subset $A$ of $X$, there is a compact subset $B$ of $X$ such that $\chi_AS=\chi_A S\chi_B$ and $S\chi_A=\chi_B S\chi_A$.
\item There is $S$ in $\cC_L(\pi)^G$ such that $\lim_{t\to \infty} \lVert T_t-S_t\rVert= 0$ and $S$ is properly supported.
\end{enumerate}
\end{lemma} 
\begin{proof} (1) $\implies$ (2): Obvious. (2) $\implies$ (3): This is because $\psi_c(T)$ is properly supported for any $T$. 
(3) $\implies$ (4): Obvious. (4) $\implies$ (1): It is enough to show that if $S$ in $\cC_L(\pi)^G$ is properly supported, then $\limt\lVert S_t - \psi_c(S_t)\rVert= 0$ for any cut-off function $c$ on $X$.  Since $S$ is properly supported, there is $\chi\in C_c(X)$ such that $c=c\chi$ and $cS=cS\chi$. We have
\[
S_t - \psi_c(S_t) = \int_{g\in G} g(c)^2S_t - g(c)S_tg(c) d\mu_G(g) 
\]
\[
=  \int_{g\in G} g(c)^2S_tg(\chi) - g(c)S_tg(c)g(\chi) d\mu_G(g)  =   \int_{g\in G} g(c)[g(c), S_t]g(\chi) d\mu_G(g).
\]
We have $\limt\lVert[g(c), S_t]\rVert = \limt\lVert[c, S_t]\rVert = 0$ (uniformly in $g$ in $G$). It follows $\limt\lVert S_t - \psi_c(S_t)\rVert=0$.
\end{proof}

It is easy to see that the conditions (3) and (4) are preserved by taking sums, products and adjoint and that  the conditions (1) and (2) pass to the norm limit. We define some $C^*$-subalgebras of $\cC_L(H_X\otimes B)^G$ and $\cC_L(\pi)^G$ as follows:
\begin{itemize}
\item $\cC_L(H_X\otimes B)^G_{\proper}$ is the $C^*$-subalgebra of $\cC_L(H_X\otimes B)^G$ consisting of $T$ satisfying the four equivalent conditions in Lemma \ref{lem_properly}.
\item $\cC_L(\pi)^G_{\proper}$ is the $C^*$-subalgebra of $\cC_L(\pi)^G$ consisting of $T$ satisfying the four equivalent conditions in Lemma \ref{lem_properly}.
\end{itemize}
We have inclusions
\[
\xymatrix{
 \cC_L(H_X\otimes B)^G_{\proper}  \ar[r]^-{} \ar[d]^-{}   & \cC_L(H_X\otimes B)^G \ar[d]^-{}   \\
\cC_L(\pi)^G_{\proper}   \ar[r]^-{}         &       \cC_L(\pi)^G.
}
\]
We have the following commutative diagram of short exact sequences
\[
\xymatrix{
0   \ar[r]^-{}         &       C_0([1, \infty), C^*(H_X\otimes B)^G)  \ar[r]^-{} \ar[d]^-{}   &   \cC_L(H_X\otimes B)^G_{\proper}  \ar[r]^-{} \ar[d]^-{}   & \cC_{L, Q}(H_X\otimes B)^G_{\proper} \ar[d]^-{\cong}   \ar[r]^-{}         &      0  \\
0   \ar[r]^-{}         &        C_0([1, \infty), \cC(\pi)^G) \ar[r]^-{}   &  \cC_L(\pi)^G_{\proper}   \ar[r]^-{}         &       \cC_{L, Q}(\pi)^G_{\proper}    \ar[r]^-{}         &      0.
}
\]
In particular, the inclusion induces an isomorphism
\[
K_\ast( \cC_L(H_X\otimes B)^G_{\proper} ) \cong K_\ast(  \cC_{L}(\pi)^G_{\proper}).
\]

\begin{proposition} The two $C^*$-subalgebras $C^*_{L, u}(H_X\otimes B)^G$ and $\cC_L(H_X\otimes B)^G_{\proper}$ in $C_{b, u}([1, \infty), \Linears(H_X\otimes B))$ are identical.
\end{proposition} 
\begin{proof} $C^*_{L, u}(H_X\otimes B)^G\subset \cC_L(H_X\otimes B)^G_{\proper}$: Take any $T\in C_{b, u}([1, \infty), \bC(H_X\otimes B)^G)$ such that $\prop(T_t)\to 0$ with respect to a fixed metric $d$ on $X^+$. Let $X_0$ be a compact subset of $X$ such that $GX_0=X$ and $X_1$ be any compact neighborhood of $X_0$ in $X$. Since $\prop(T_t)\to 0$, there is $t_0\geq1$ such that for any $t>t_0$, $\chi_{X_0}T_t=\chi_{X_0}T_t\chi_{X_1}$ and $T_t\chi_{X_0}=\chi_{X_1}T_t\chi_{X_0}$. This implies that $T$ satisfies the condition (3) of Lemma \ref{lem_properly}. We already know $C^*_{L, u}(H_X\otimes B)^G\subset \cC_L(H_X\otimes B)^G$ so we have $C^*_{L, u}(H_X\otimes B)^G\subset \cC_L(H_X\otimes B)^G_{\proper}$. 

$\cC_L(H_X\otimes B)^G_{\proper} \subset C^*_{L, u}(H_X\otimes B)^G$: Let $T\in \cC_L(H_X\otimes B)^G_{\proper}$. It is enough to show that $\psi_c(T)\in C^*_{L, u}(H_X\otimes B)^G$ for a cut-off function $c$ on $X$ since both algebras contain the common ideal $C_0([1, \infty), C^*(H_X\otimes B)^G)$. Since $\lim_{t\to \infty}\lVert[T_t, \phi]\rVert= 0$ for any $\phi \in C_0(X)$, there is $S \in C_{b, u}([1, \infty), \Linears(H_X))$ such that $\lim_{t\to \infty}\lVert T_t -S_t\rVert=0$ and such that $\prop(S_t)\to 0$ as $t\to \infty$ for some (any) fixed metric $d$ on $X^+$. We can arrange it so that $S_t$  a $G$-continuous, locally compact operator in $\Linears(H_X)$ for each $t\geq1$. In this case, since $\lim_{t\to \infty}\lVert T_t -S_t\rVert=0$ with $T$ $G$-equivariant, we see that $S$ is $G$-continuous in $C_{b, u}([1, \infty), \Linears(H_X))$. Now consider $\tilde S \in C_{b, u}([1, \infty), \Linears(H_X\otimes B))$ defined (weakly) by
\[
\tilde S_t = \int_{g\in G} g(c)g(S_t) g(c) d\mu_G(g).
\]
Note this is uniformly continuous in $t$. Indeed $S \mapsto \tilde S$ is a u.c.p.\ map on $C_{b, u}([1, \infty), \Linears(H_X\otimes B))$ ($S$ does not have to be $G$-continuous for this to be well-defined). We have
\[
\psi_c(T_t) - \tilde S_t  = \int_{g\in G} g(c)(T_t- g(S_t)) g(c) d\mu_G(g).
\]
Together with $\limt\lVert T_t- g(S_t)\rVert = \limt\lVert T_t- S_t\rVert = 0$ (uniformly in $g$ in $G$), we see that $\lim_{t\to \infty}\lVert\psi_c(T_t) - \tilde S_t\rVert=0$. Since $\tilde S_t \in \bC(H_X\otimes B)^G \subset C^*(H_X\otimes B)^G$ for any $t\geq1$,  to show $\psi_c(T) \in C^*_{L, u}(H_X\otimes B)^G$ it is enough to show $\tilde S \in C^*_{L, u}(H_X\otimes B)^G$. For this, it suffices to show $\prop(\tilde S_t)\to 0$ as $t \to \infty$ with respect to the fixed metric $d$ on $X^+$. It suffices to show $\prop(\chi_{X_0}\tilde S_t)\to 0$ and $\prop(\tilde S_t\chi_{X_0})\to 0$ as $t\to \infty$ for any compact subset $X_0$ of $X$. On the other hand,
\[
\chi_{X_0}\tilde S_t = \int_{g\in G} \chi_{X_0}g(c)g(S_t) g(c) d\mu_G(g)
\]
and $\chi_{X_0}g(c)=0$ for any $g\in G\backslash K$ for some compact subset $K$ of $G$. The claim $\prop(\chi_{X_0}\tilde S_t)\to 0$ now follows since $\prop(g(S_t)) \to 0$ uniformly in $g \in K$. We also have $\prop(\tilde S_t\chi_{X_0})\to 0$ for the same reason, so we are done.
\end{proof}

Note we have a natural inclusion ($\ast$-homomorphism)
\[
 \iota\colon  \cC_L(\pi)^G \to M(RL^*_u(H_X\otimes B)) \subset  M(RL^*_u(H_X\otimes B)\rtimes_rG)
\]
just because $T\in  \cC_L(\pi)^G$ satisfies $\lim_{t\to \infty}\lVert[\phi, T_t]\rVert = 0$. Moreover, for any $T\in  \cC_L(\pi)^G$, $\phi$ in $C_0(X)$, we have
\[
\phi T \in RL^*_u(H_X\otimes B)
\]
because $T_t$ is locally compact and $\lim_{t\to \infty}\lVert[\phi, T_t]\rVert = 0$. We let
\[
\pi\rtimes_r1\colon C_0(X)\rtimes_rG \to M(RL^*_u(H_X\otimes B)\rtimes_rG)
\]
be the crossed product of the structure map $\pi\colon C_0(X) \to M(RL^*_u(H_X\otimes B))$. Let $c \in C_c(X)$ be a cut-off function on $X$ and $p_c\in C_0(X)\rtimes_rG$ be the associated cut-off projection. We have a c.c.p.\
 map \[
(\pi\rtimes_r1)(p_c) \iota (\pi\rtimes_r1)(p_c) \colon   \cC_L(\pi)^G \to RL^*_u(H_X\otimes B)\rtimes_rG
\] 
which sends $T \in  \cC_L(\pi)^G$ to the compression $(\pi\rtimes_r1)(p_c) \iota(T) (\pi\rtimes_r1)(p_c) $ in $RL^*_u(H_X\otimes B)\rtimes_rG$. For any $T \in  \cC_L(\pi)^G$, we have
\[
[\iota(T), (\pi\rtimes_r1)(p_c)] \in RL^*_0(H_X\otimes B)\rtimes_rG.
\]
Thus, after passing to the quotient $RL^*_{u,Q}(H_X\otimes B)\rtimes_rG$, the c.c.p.\ map $(\pi\rtimes_r1)(p_c) \iota (\pi\rtimes_r1)(p_c)$ becomes a $\ast$-homomorphism
\[
(\pi\rtimes_r1)(p_c) \iota (\pi\rtimes_r1)(p_c) \colon  \cC_L(\pi)^G \to RL^*_{u, Q}(H_X\otimes B)\rtimes_rG.
\]
Recall that the quotient map induces an isomorphism 
\[
K_\ast(RL^*_{u}(H_X\otimes B)\rtimes_rG) \cong  K_\ast(RL^*_{u, Q}(H_X\otimes B)\rtimes_rG).
\]
\begin{definition} We define a group homomorphism
\[
\iota_{p_c} \colon  K_\ast(\cC_L(\pi)^G) \to  K_\ast(RL^*_{u}(H_X)\rtimes_rG)
\]
as the composition of the group homomorphism 
\[
(\pi\rtimes_r1)(p_c) \iota (\pi\rtimes_r1)(p_c)_\ast \colon  K_\ast(\cC_L(\pi)^G) \to K_\ast(RL^*_{u, Q}(H_X)\rtimes_rG)
\]
and the inverse of the isomorphism 
\[
K_\ast(RL^*_{u}(H_X\otimes B)\rtimes_rG) \cong  K_\ast(RL^*_{u, Q}(H_X\otimes B)\rtimes_rG)
\]
induced by the quotient map. If $A$ is any of the subalgebras 
\[
\cC_L(H_X\otimes B)^G_{\proper}, \cC_L(\pi)^G_{\proper}, \cC_L(H_X\otimes B)^G
\]
of $\cC_L(\pi)^G$, we set
\[
\iota_{p_c} \colon  K_\ast(A) \to  K_\ast(RL^*_{u}(H_X\otimes B)\rtimes_rG)
\]
be the restriction of $\iota_{p_c}$ to the subalgebra $A$.
\end{definition}

Now, for an $X$-$G$-module $H_X$,  let $\tilde H_X=H_X\otimes L^2(G)$ be the $X$-$G$-module as before (see Definition \ref{def_tildeH}). Let
\[
\tilde\pi\colon C_0(X) \to  \Linears(\tilde H_X\otimes B)
\]
be the structure map. We have the right-regular representation (see Proposition \ref{prop_rightreg_localized})
\[
\rho \colon  RL^*_u(H_X\otimes B)\rtimes_rG \to C^*_{L, u}(\tilde H_X\otimes B)^G
\]
and the inclusions 
\[
\xymatrix{
 C^*_{L, u}(\tilde H_X\otimes B)^G = \cC_L(\tilde H_X\otimes B)^G_{\proper}  \ar[r]^-{} \ar[d]^-{}   & \cC_L(\tilde H_X\otimes B)^G \ar[d]^-{}   \\
\cC_L(\tilde \pi)^G_{\proper}   \ar[r]^-{}         &       \cC_L(\tilde\pi)^G.
}
\]

\begin{definition} If $A$ is any of the algebras 
\[
\cC_L(\tilde H_X\otimes B)^G_{\proper}, \cC_L(\tilde\pi)^G_{\proper}, \cC_L(\tilde H_X\otimes B)^G, \cC_L(\tilde \pi)^G,
\]
we set
\[
\rho_\ast \colon  K_\ast(RL^*_{u}(H_X\otimes B)\rtimes_rG) \to K_\ast(A)
\]
to be the group homomorphism induced by the right-regular representation 
\[
\rho\colon RL^*_u(H_X\otimes B)\rtimes_rG \to C^*_{L, u}(\tilde H_X\otimes B)^G
\]
followed by the inclusion from $C^*_{L, u}(\tilde H_X\otimes B)^G$ to $A$.
\end{definition}

We are now interested in computing the compositions
\[
 \iota_{p_c}  \circ \rho_\ast\colon K_\ast(RL^*_u(H_X\otimes B)\rtimes_rG) \to K_\ast(C^*_{L, u}(\tilde H_X\otimes B)^G) \to  K_\ast(RL^*_u(\tilde H_X\otimes B)\rtimes_rG)
\]
and
\[
\rho_\ast \circ \iota_{p_c}  \colon K_\ast(C^*_{L, u}(H_X\otimes B)^G) \to K_\ast(RL^*_u(H_X\otimes B)\rtimes_rG) \to  K_\ast(C^*_{L, u}(\tilde H_X\otimes B)^G).
\] 
Let us first compute the first one. We remind that a proper $G$-space $X$ has been assumed to be $G$-compact from the beginning of this section.

\begin{lemma}\label{lem_genfix} For any (not-necessarily separable) $G$-$C^*$-algebra $A$ equipped with a non-degenerate representation of the $G$-$C^*$-algebra $C_0(X)$ to $M(A)$, the right regular representation
\[
\rho\colon A\rtimes_rG \to M(A\otimes \Compacts(L^2(G))), \,\, A\ni a\mapsto (g(a))_{g\in G}, \,\, G\ni g\mapsto \rho_g, 
\]
is an isomorphism onto the $C^*$-algebra $M(A\otimes \Compacts(L^2(G)))^{G, c}$ defined as the completion of the $\ast$-algebra consisting of $G$-equivariant, locally compact, properly supported operators in $M(A\otimes \Compacts(L^2(G)))\cong \Linears(A\otimes L^2(G))$.
\end{lemma}
\begin{proof} The proof of Proposition \ref{prop_isom} works verbatim.
\end{proof}
\begin{remark} The algebra $M(A\otimes \Compacts(L^2(G)))^{G, c}$ is known as a generalized fixed point algebra and this lemma is proved in a more general setting (\cite[Corollary 3.25, Remark 3.26]{BE15} see also \cite{BE14} \cite{BE16} for a more general theory).
\end{remark}

\begin{lemma} The composition $ \iota_{p_c}  \circ \rho_\ast\colon K_\ast(RL^*_u(H_X\otimes B)\rtimes_rG) \to K_\ast(RL^*_u(\tilde H_X\otimes B)\rtimes_rG)$ coincides with the map induced by the strict cover $V_c\colon H_X\to  \tilde H_X$ of the identity map on $X$ defined as
\[
V_c\colon H_X \ni v\mapsto (g\mapsto g(c)v) \in \tilde H_X =H_X\otimes L^2(G).
\]
\end{lemma}
\begin{proof} 

A c.c.p.\ map 
\[
\iota^1_{p_c}\colon RL^*_u(H_X\otimes B)\rtimes_rG  \to (RL^*_u(H_X\otimes B)\otimes \Compacts(L^2(G)) ) \rtimes_rG
\]
is defined as the compression of the right-regular representation 
\[
 RL^*_u(H_X\otimes B)\rtimes_rG \ni T \to \rho(T)  \in M(RL^*_u(H_X\otimes B)\otimes \Compacts(L^2(G)))
\]
by the projection $p^{(1)}_c$ which is the image of the cut-off projection $p_c$ in $C_0(X)\rtimes_rG$ by the map
\[
\pi^1\rtimes_r1 \colon C_0(X)\rtimes_rG \to M( (RL^*_u(H_X\otimes B)\otimes \Compacts(L^2(G)) ) \rtimes_rG)
\]
induced by the natural $G$-equivariant representation 
\[
\pi^{1}\colon C_0(X) \to M(RL^*_u(H_X)) \to  M( RL^*_u(H_X\otimes B)\otimes \Compacts(L^2(G))).
\]
Passing to the quotient, the c.c.p.\ map $\iota^1_{p_c}$ becomes a $\ast$-homomorphism
\[
\iota^1_{p_c}\colon RL^*_u(H_X\otimes B)\rtimes_rG \to (RL^*_{u, Q}(H_X\otimes B)\otimes \Compacts(L^2(G)) ) \rtimes_rG.
\]
Composing its induced map on K-theory with the inverse of the isomorphism (the quotient map)
\[
K_\ast((RL^*_{u}(H_X\otimes B)\otimes \Compacts(L^2(G)) ) \rtimes_rG) \cong  K_\ast( (RL^*_{u, Q}(H_X\otimes B)\otimes \Compacts(L^2(G)) ) \rtimes_rG),
\]
we obtain a group homomorphism
\[
\iota^1_{p_c}\colon K_\ast(RL^*_u(H_X\otimes B)\rtimes_rG) \to K_\ast( (RL^*_u(H_X\otimes B)\otimes \Compacts(L^2(G)) ) \rtimes_rG).
\]
Here, we used that the reduced crossed product preserves the exact sequence 
\[
0 \to RL^*_{0}(H_X\otimes B)\otimes \Compacts(L^2(G))  \to RL^*_{u}(H_X\otimes B)\otimes \Compacts(L^2(G)) 
\]
\[
 \to RL^*_{u, Q}(H_X\otimes B)\otimes \Compacts(L^2(G))  \to 0
\]
since the quotient is a proper $G$-$C^*$-algebra.

The composition $ \iota_{p_c}  \circ \rho_\ast$ factors through the natural inclusion
\[
(RL^*_u(H_X\otimes B)\otimes \Compacts(L^2(G))) \rtimes_rG \to RL^*_u(\tilde H_X\otimes B)\rtimes_rG
\]
via the group homomorphism $\iota^1_{p_c}$.

We now consider
\[
\mathrm{Ad}_{V_c \ast}\colon  K_\ast(RL^*_u(H_X\otimes B)\rtimes_rG) \to K_\ast(RL^*_u(\tilde H_X\otimes B)\rtimes_rG).
\]
Let $p_c^{(2)}=V_cV_c^\ast \in \Linears(\tilde H_X) \subset  \Linears(\tilde H_X\otimes B)$. The $G$-equivariant projection $p^{(2)}_c$ is the image of $p_c \in C_0(X)\rtimes_rG$ by the right-regular representation
\[
C_0(X)\rtimes_rG \to \Linears(\tilde H_X)= \Linears(H_X\otimes L^2(G)), C_0(X)\ni \phi \mapsto (g(\phi))_{g\in G}, \,\, G\ni g\mapsto \rho_g.
\] 
We claim that the $\ast$-homomorphism
\[
\mathrm{Ad}_{V_c}\colon RL^*_u(H_X\otimes B) \ni T \mapsto V_cTV_c^\ast \in RL^*_u(\tilde H_X\otimes B)
\]
and the $G$-equivariant c.c.p.\ map
\[
\iota^2_{p_c}\colon  RL^*_u(H_X\otimes B) \ni T \mapsto p^{(2)}_c(T)_{g\in G} p^{(2)}_c \in  RL^*_u(\tilde H_X\otimes B)
\]
coincide after passing to the quotient $RL^*_{u, Q}(\tilde H_X\otimes B)$. Indeed, for $S$ in $\Linears(H_X\otimes B)$, we have
\[
V_c^\ast (S)_{g\in G}V_c = \int_{g\in G}g(c)S g(c) d\mu_g(g).
\]
Thus, for any $T \in RL^*_u(H_X\otimes B)$, we have
\[
\lim_{t\to \infty}\lVert T_t - V_c^\ast (T_t)_{g\in G}V_c    \rVert = 0.
\] 
We can see this for example by considering $T$ which has uniform compact support. 

The $G$-equivariant c.c.p.\ map $\iota^2_{p_c}$ naturally factors through the inclusion
\[
RL^*_u(H_X\otimes B)\otimes \Compacts(L^2(G)) \to RL^*_u(\tilde H_X\otimes B).
\]
Overall, in order to show  $\iota_{p_c}  \circ \rho_\ast = \mathrm{Ad}_{V_c \ast}$ on  $K_\ast(RL^*_u(H_X\otimes B)\rtimes_rG)$, it suffices to show that the two group homomorphisms
\[
\iota^1_{p_c}, \iota^2_{p_c}\colon K_\ast(RL^*_u(H_X\otimes B)\rtimes_rG) \to K_\ast( (RL^*_u(H_X\otimes B)\otimes \Compacts(L^2(G)) ) \rtimes_rG)
\]
induced by the c.c.p.\ maps $\iota^1_{p_c}, \iota^2_{p_c}$ (which become $\ast$-homomorphisms after passing to the quotient) coincide. On the other hand, by viewing 
\[
(RL^*_u(H_X\otimes B)\otimes \Compacts(L^2(G)) ) \rtimes_rG \cong M(RL^*_u(H_X\otimes B)\otimes \Compacts(L^2(G)) \otimes \Compacts(L^2(G)) )^{G, c}
\]
as in Lemma \ref{lem_genfix}, we can see that the two map $\iota^1_{p_c}$ and $\iota^2_{p_c}$ are conjugate by the unitary $U$ in the multiplier algebra of 
\[
M(RL^*_u(H_X\otimes B)\otimes \Compacts(L^2(G)) \otimes \Compacts(L^2(G)) )^{G, c}
\]
defined as the flip on $G\times G$, i.e. $U \in \Linears(L^2(G\times G))$ defined by $Uf(g_1, g_2)= f(g_2, g_1)$ for $f\in L^2(G\times G)$ so we are done.
\end{proof}

\begin{corollary} \label{cor_first_composition} The composition 
\[
\iota_{p_c}  \circ \rho_\ast\colon K_\ast(RL^*_u(H_X\otimes B)\rtimes_rG) \to K_\ast(RL^*_u(\tilde H_X\otimes B)\rtimes_rG)
\]
is an isomorphism whenever $H_X$ is a universal $X$-$G$-module.
\end{corollary}

We now study the other composition.
\begin{lemma}\label{lem_second_composition} The composition 
\[
\rho_\ast \circ \iota_{p_c}  \colon K_\ast(C^*_{L, u}(H_X\otimes B)^G) \to K_\ast(C^*_{L, u}(\tilde H_X\otimes B)^G)
\]
coincides with the map induced by the strict cover $V_c\colon H_X\to  \tilde H_X$ of the identity map on $X$ defined as
\[
V_c\colon H_X \ni v\mapsto (g\mapsto g(c)v) \in \tilde H_X =H_X\otimes L^2(G).
\]

\end{lemma}
\begin{proof} We can directly see that the composition is induced by the c.c.p.\ map (which becomes a $\ast$-homomorphism after passing to the quotient),
\[
p_c \cdot p_c \colon C^*_{L, u}(H_X\otimes B)^G \ni  T  \mapsto p_c(T)_{g\in G}p_c \in C^*_{L, u}(\tilde H_X\otimes B)^G
\]
where the cut-off projection $p_c$ is represented by the right-regular representation of $C_0(X)\rtimes_rG$:
\[
C_0(X)\rtimes_rG \to \Linears(\tilde H_X)= \Linears(H_X\otimes L^2(G)), C_0(X)\ni \phi \mapsto (g(\phi))_{g\in G}, \,\, G\ni g\mapsto \rho_g.
\]
As before, more precisely, this c.c.p.\ map induces the map on K-theory as the composition of the $\ast$-homomorphism
\[
p_c \cdot p_c \colon C^*_{L, u}(H_X\otimes B)^G \to C^*_{L, Q}(\tilde H_X\otimes B)^G = C^*_{L, u}(\tilde H_X\otimes B)^G / C_0([1, \infty), C^*(\tilde H_X\otimes B)^G)
\]
and the inverse of the isomorphism (the quotient map)
\[
K_\ast(C^*_{L, u}(\tilde H_X\otimes B)^G) \to K_\ast(C^*_{L, Q}(\tilde H_X\otimes B)^G).
\]
We have $p_c=V_cV_c^\ast$ and as before, for $S\in \Linears(H_X\otimes B)$,
\[
V_c^\ast (S)_{g\in G}V_c = \int_{g\in G}g(c)S g(c) d\mu_g(g).
\]
To show that $p_c \cdot p_c$ and $\mathrm{Ad}_{V_c \ast}$ coincide on  $K_\ast(C^*_{L, u}(H_X\otimes B)^G)$, it is enough to show that for any $T \in C^*_{L, u}(H_X\otimes B)^G$, we have $\limt\lVert T_t - V^\ast_c(T_t)_{g\in G}V_c\rVert = 0$ but this is precisely the condition  (2) of Lemma \ref{lem_properly}, which is satisfied for any $T$ in $\cC_L(\pi)^G_{\proper}$ in particular for any  $T \in C^*_{L, u}(H_X\otimes B)^G= \cC_L(H_X\otimes B)^G_{\proper}$.
\end{proof}

The same proof shows the following:
\begin{lemma} The composition 
\[
\rho_\ast \circ \iota_{p_c}  \colon K_\ast(\cC_L(\pi)^G_{\proper}) \to K_\ast(\cC_L(\tilde\pi)^G_{\proper})
\]
coincides with the map induced by the strict cover $V_c\colon H_X\to  \tilde H_X$ of the identity map on $X$.
\end{lemma} 

Before jumping to the conclusion, we make an important remark here. The equivariant Roe algebra $C^*(H_X\otimes B)^G$, the localized equivariant Roe algebra $C^*_{L, u}(H_X\otimes B)^G$ as well as all the other algebras such as $\cC(\pi)^G$ and $\cC_L(\pi)^G$ are not well-behaved with respect to an arbitrary equivariant continuous cover $(V_t\colon H_X\to H_Y)_{t\in [1, \infty)}$ of a $G$-equivariant continuous map $f\colon X\to Y$ even if $f$ is the identity map on a $G$-compact space. This is because a continuous cover $V_t$ may not be properly supported in general and hence the conjugation by $V_t$ may not preserve locally compact operators. On the other hand, they are functorial with respect to a strict cover $V$, if exists, of the identity map.

More relevantly, if we focus on  $X$-$G$-modules of the form $\tilde H_X\otimes B$ for an $X$-$G$-module $H_X$, the equivariant Roe algebra $C^*(\tilde H_X\otimes B)^G$, the localized equivariant Roe algebra $C^*_{L, u}(\tilde H_X\otimes B)^G=\cC_L(\tilde H_X\otimes B)^G_{\proper}$ as well as $\cC_L(\tilde H_X\otimes B)^G$ are functorial with respect to a $G$-equivariant continuous cover of the identity map of the form $\tilde V=V\otimes 1\colon \tilde H_X \to \tilde H'_X$ for a $G$-equivariant continuous cover $V\colon H_X \to H_X'$. To see this, first, we have the following commutative diagram
\[
\xymatrix{
 C^*(\tilde H_X\otimes B)^G  \ar[r]^-{\mathrm{Ad}_{\tilde V}}   &  C^*(\tilde H'_X\otimes B)^G  \\
\Compacts(H_X\otimes B)\rtimes_rG  \ar[r]^-{\mathrm{Ad}_V}  \ar[u]^-{\cong}_{\rho}         &       \Compacts(H'_X\otimes B)\rtimes_rG  \ar[u]^-{\cong}_{\rho}      
}
\]
for any $G$-equivariant isometry $V\colon H_X \to H_X'$, which in particular says that $\mathrm{Ad}_{\tilde V}$ maps $C^*(\tilde H_X\otimes B)^G$ to $C^*(\tilde H'_X\otimes B)^G$. Now, if $(V_t\colon H_X \to H'_X)_{t\in [1, \infty)}$ is a $G$-equivariant continuous cover of the the identity map on $X$, we can see that $\mathrm{Ad}_{\tilde V_t}$ maps $\cC_L(\tilde H_X\otimes B)^G$ to $\cC_L(\tilde H'_X\otimes B)^G$ and maps $C_{L, u}^*(\tilde H_X\otimes B)^G=\cC_L(\tilde H_X\otimes B)^G_{\proper}$ to $C_{L, u}^*(\tilde H'_X\otimes B)^G=\cC_L(\tilde H'_X\otimes B)^G_{\proper}$. To see the latter, let $X_0$ be a compact subset of $X$ with $X=GX_0$. If $S\in \cC_L(\tilde H_X\otimes B)^G$ is properly supported, then $S'_t=\tilde V_tS_t\tilde V_t^\ast$ satisfies 
\[
\chi_{X_0} S'_t = \chi_{X_0} S'_t  \chi_{X_1}, \,\,\, S'_t \chi_{X_0} = \chi_{X_1} S'_t  \chi_{X_0}
\]
for some compact subset $X_1$ of $X$ for large enough $t$. This already implies that $S'$ satisfies the condition (3) of Lemma \ref{lem_properly}. Similarly, for any two covers $(V_{i, t}\colon H_X \to H'_X)_{t\in [1, \infty)}$ for $i=1, 2$, $\tilde V_{1, t} \tilde V_{2, t}^\ast$ multiplies $C_{L, u}^*(\tilde H'_X\otimes B)^G=\cC_L(\tilde H'_X\otimes B)^G_{\proper}$. Using this, we see that the induced map 
\[
\mathrm{Ad}_{\tilde V_t \ast}\colon  K_\ast(C_{L, u}^*(\tilde H_X\otimes B)^G) \to K_\ast(C_{L, u}^*(\tilde H'_X\otimes B)^G) 
\]
on $K$-theory is independent of the covers and it is an isomorphism whenever $H_X$ and $H_X'$ are universal $X$-$G$-modules.

We say that an $X$-$G$-module $H_X$ is of infinite multiplicity if $H_X\cong H^{0}_X\otimes l^2(\N)$ as an $X$-$G$-module for some $X$-$G$-module $H^0_X$. 

\begin{lemma} For any universal $X$-$G$-module $H_X$ of infinite multiplicity, the right-regular representation $\rho$ induces an isomorphism
\[
\rho_\ast \colon K_\ast(RL^*_u(H_X\otimes B)\rtimes_rG) \cong  K_\ast(C^*_{L, u}(\tilde H_X\otimes B)^G).
\]
\end{lemma}
\begin{proof} By Corollary \ref{cor_first_composition}, we already know that the composition 
\[
  \iota_{p_c}\circ \rho_\ast \colon K_\ast(RL^*_u(H_X\otimes B)\rtimes_rG)  \to K_\ast(RL^*_u(\tilde H_X\otimes B)\rtimes_rG) 
  \]
  is an isomorphism. It thus suffices to show that the composition
  \[
  \rho_\ast \circ   \iota_{p_c}\colon  K_\ast(C^*_{L, u}(\tilde H_X\otimes B)^G) \to   K_\ast(C^*_{L, u}(\tilde H^{(2)}_X\otimes B)^G)
  \]
  is injective (isomorphism) where $\tilde H^{(2)}_X= \tilde H'_X$ for $H'_X=\tilde H_X$. By Lemma \ref{lem_second_composition}, this composition is induced by the isometry $V_c\colon H_X' \to \tilde H_X'$ which is a strict cover of the identity map. On the other hand, $ H_X'\cong H^{0}_X\otimes l^2(\N)\otimes L^2(G)$ and $\tilde H_X' = H_X'\otimes L^2(G)$ are isomorphic as $X$-$G$-modules. From this, we see that $\rho_\ast \circ   \iota_{p_c}=\mathrm{Ad}_{V_c \ast}$ is an isomorphism.
\end{proof}

\begin{theorem}\label{thm_magic} For any universal $X$-$G$-module $H_X$, the right-regular representation $\rho$ and the natural inclusion induce isomorphisms
\[
\rho_\ast \colon K_\ast(RL^*_u(H_X\otimes B)\rtimes_rG) \cong  K_\ast(C^*_{L, u}(\tilde H_X\otimes B)^G) \cong K_\ast(\cC_L(\tilde \pi)^G_{\proper}).
\]
\end{theorem}
\begin{proof} The first isomorphism follows from the previous lemma using the functoriality of both $K_\ast(RL^*_u(H_X\otimes B)\rtimes_rG)$ and $K_\ast(C^*_{L, u}(\tilde H_X\otimes B)^G)$ with respect to a $G$-equivariant continuous cover $(V_t\colon H_X \to H'_X)_{t\in [1, \infty)}$ of the identity map on $X$, i.e. we can assume that the universal $X$-$G$-module $H_X$ is of infinite-multiplicity. We always have the second isomorphism.
\end{proof}



\section{$\rho_X$ is an isomorphism, part II}

In this section, we study a $G$-equivariant analogue of \cite{DWW18}.  We let $X$ be a $G$-compact proper $G$-space and $H_X$ be an $X$-$G$-module which will be assumed to be of infinite-multiplicity at some places. 
We recall that $\cC_L(H_X\otimes B)^G_\proper = C^*_{L, u}(H_X\otimes B)^G$ is the $C^*$-subalgebra of $C_{b, u}([1,\infty), C^*(H_X\otimes B)^G)$ consisting of functions $T$ such that $\lim_{t\to \infty}\lVert[\phi, T_t]\rVert=0$ and such that $T-\psi_c(T) \in C_0([1, \infty), C^*(H_X\otimes B)^G)$ for some (any) cut-off function $c$ on $X$, where 
\[
\psi_c(T) = \int_{g\in G}g(c)Tg(c)d\mu_G(g).
\]

We introduce the following $C^*$-algebras (c.f.\ \cite[Section 3]{DWW18}).
\begin{itemize}
\item $D^*(H_X\otimes B)^G$ is the $C^*$-subalgebra of $\Linears(H_X\otimes B)$ generated by $G$-equivariant, properly supported, pseudo-local operators. Here, an operator $T$ in $\Linears(H_X\otimes B)$ is pseudo-local if $[\phi, T] \in \Compacts(H_X\otimes B)$ for any $\phi$ in $C_0(X)$.
\item $\cD_L(H_X\otimes B)^G$ is the $C^*$-subalgebra of $C_{b, u}([1,\infty), D^*(H_X\otimes B)^G)$ consisting of functions $T$ such that $\lim_{t\to \infty}\lVert[\phi, T_t]\rVert=0$.
\item $\cD_L(H_X\otimes B)^G_{\proper}$ is the $C^*$-subalgebra of $\cD_L(H_X\otimes B)^G$ consisting of functions $T$ such that $T-\psi_c(T) \in C_0([1, \infty), D^*(H_X\otimes B)^G)$ for some (any) cut-off function $c$ on $X$. Similarly to Lemma \ref{lem_properly}, the second condition is equivalent to that there is a properly supported function $S$ in $\cD_L(H_X\otimes B)^G$ such that $T-S \in C_0([1, \infty), D^*(H_X\otimes B)^G)$.
\item $\cC_T(H_X \otimes B)^G_{\proper}$ is the $C^*$-subalgebra of $C_{b, u}([1, \infty), C^*(H_X\otimes B)^G)$ generated by functions $T$ which are properly supported.
\item $\cD_T(H_X\otimes B)^G_{\proper}$ is the $C^*$-subalgebra of $C_{b, u}([1, \infty), D^*(H_X\otimes B)^G)$ generated by  functions $T$ which are properly supported.
\end{itemize}

We have inclusions (each vertical map is an inclusion as an ideal)
\[
\xymatrix{
\cC_L(H_X\otimes B)^G_{\proper}  \ar[r]^-{} \ar[d]^-{} &\cC_T(H_X \otimes B)^G_{\proper} \ar[d]^-{}  \\
\cD_L(H_X\otimes B)^G_{\proper} \ar[r]^-{}  &  \cD_T(H_X \otimes B)^G_{\proper},
}
\]
as well as inclusions (horizontal maps are inclusions as constant functions)
\[
\xymatrix{
C^*(H_X\otimes B)^G \ar[r]^-{} \ar[d]^-{} &\cC_T(H_X \otimes B)^G_{\proper} \ar[d]^-{}  \\
D^*(H_X\otimes B)^G  \ar[r]^-{}  &  \cD_T(H_X \otimes B)^G_{\proper}.
}
\]

\begin{lemma} We have 
\[
\cD_L(H_X\otimes B)^G_{\proper} \cap \cC_T(H_X \otimes B)^G_{\proper}  = \cC_L(H_X\otimes B)^G_{\proper}.
\]
\end{lemma}
\begin{proof}  Let $T\in \cD_L(H_X\otimes B)^G_{\proper} \cap \cC_T(H_X \otimes B)^G_{\proper}$. Then, we have $T \in C_{b, u}([1, \infty), C^*(H_X\otimes B)^G)$ and $\limt\lVert[T_t, \phi]\rVert = 0$ and $\limt\lVert T_t-\psi_c(T_t)\rVert = 0$ for any cut-off function $c$ on $X$.  Thus, $T\in \cC_L(H_X\otimes B)^G_{\proper}$.  The converse is trivial.
\end{proof}

\begin{lemma}\label{lem_ex}(c.f.\ \cite[Lemma 4.1]{DWW18}) Let $D \subset   C_{b, u}([1,\infty), D^*(H_X \otimes B)^G)$ be a separable subset. Assume that the set $D$ is properly supported in a sense that for any compact subset $A$ of $X$, there is a compact subset $B$ so that $\chi_AT=\chi_AT\chi_B$ (and $T\chi_A=\chi_BT\chi_A$) for any $T\in D$. There is $x \in C_{b, u}([1,\infty), C^*(H_X \otimes B)^G)$ which is properly supported such that 
\begin{itemize}
\item $[x, \phi] \in C_0([1, \infty), \Compacts(H_X\otimes B))$ for any $\phi \in C_0(X)$,
\item $(1-x)[T, \phi] \in C_0([1, \infty), \Compacts(H_X\otimes B))$ for any $\phi\in C_0(X)$ and $T\in D$.
\end{itemize}
\end{lemma}
\begin{proof} Let $c$ be a cut-off function on $X$ and $X_0$ be the support of $c$ which is a compact subset of $X$ with $GX_0=X$.  Since $D$ is properly supported, there is $\chi\in C_c(X)$ such that $c=c\chi$ and $cT=cT\chi$ for any $T\in D$. Let $X_1$ be the support of $\chi$.  Let $K$ be a compact subset of $G$ so that $X_0\cap gX_1=\emptyset$ for $g\in G\backslash K$. Let $(\phi_n)_n$ be a countable subset of $C_0(X)$ such that $\supp(\phi_n)\subset X_0$ for all $n$ and such that $G(\phi_n)_n$ is total in $C_0(X)$. Let $(T^{(m)})_{m\geq1}$ be a countable dense subset of $D$.

Let $y_n$ be an approximate unit in $\Compacts(H_X\otimes B)$ quasi-central with respect to $C_0(X)$ so that we have $\limn\lVert[\phi, y_n]\rVert= 0$ for any $\phi \in C_0(X)$ and such that 
\[
 \lVert(1-y_n) c[T^{(m)}_t, g^{-1}(\phi_k)]\rVert < 1/n
\]
for any $g\in K$, $t\in [1, n+1]$, $1\leq m \leq n$, and $1\leq k \leq n$. Let 
\[
x_n=\int_{g\in G}g(c) g(y_n) g(c) d\mu_G(g) \in  C^*(H_X \otimes B)^G.
\]
Then, we have $(n\mapsto [\phi, x_n]) \in  C_0(\N, \Compacts(H_X\otimes B))$ for any $\phi$ in $C_0(X)$. Take any $\phi$ in $C_0(X)$ with support contained in $X_0$. For any $T\in D$, we have
\[
(1-x_n)[T_t, \phi] =\int_{g\in G}g(c) (1-g(y_n)) g(c)[T_t, \phi] g(\chi) d\mu_G(g)
\]
\[
= \int_{g\in K}g(c) (1-g(y_n)) g(c)[T_t, \phi] g(\chi) d\mu_G(g),
\]
and
\[
\lVert(1-g(y_n)) g(c)[T_t, \phi]\rVert = \lVert(1-y_n) c[T_t, g^{-1}(\phi)]\rVert.
\]
In particular, we have
\[
\lVert(1-g(y_n)) g(c)[T^{(m)}_t, \phi_k]\rVert \leq 1/n
\]
for any $g\in K$, $t\in [1, n+1]$, $1\leq m \leq n$ and $1\leq k \leq n$ and so
\[
\lVert(1-x_n)[T^{(m)}_t, \phi_k]\rVert \leq C/n
\]
for any $t\in [1, n+1]$, $1\leq m \leq n$ and $1\leq k \leq n$, where the constant $C>0$ only depends on fixed functions  $c$ and $\chi$. A desired function $x$ can be obtained by
\[
x_t= (n+1-t)x_n + (t-n)x_{n+1}
\]
for $t\in [n, n+1]$. 

\end{proof}

\begin{lemma}(c.f.\ \cite[Proposition 2.3]{QiaoRoe}, \cite[Proposition 4.2]{DWW18}) \label{lem_eta} The natural map (inclusion)
\[
\eta\colon \cD_L(H_X\otimes B)^G_{\proper} / \cC_L(H_X\otimes B)^G_{\proper}  \  \to  \cD_T(H_X \otimes B)^G_{\proper} / \cC_T(H_X \otimes B)^G_{\proper}
\]
is bijective (isomorphism).
\end{lemma}
\begin{proof} It is enough to show that for any $T \in   C_{b, u}([1,\infty), D^*(H_X \otimes B)^G)$ which is properly supported, there is $S \in \cD_L(H_X\otimes B)^G_{\proper} $ such that $T-S \in \cC_T(H_X \otimes B)^G_{\proper}$. 

We let $x$ as given by Lemma \ref{lem_ex} for $D=\{T\}$. Set $S=(1-x)T$. Then, we have 
\[
T -S = xT\in \cC_T(H_X \otimes B)^G_{\proper}
\]
since $x \in \cC_T(H_X \otimes B)^G_{\proper}$. We claim that $S \in  \cD_L(H_X\otimes B)^G_{\proper}$. Since $x$ and $T$ are properly supported, $S=(1-x)T$ is properly supported. Thus, it is enough to show that $\lim_{t\to \infty}\lVert[\phi, S_t]\rVert = 0$ but this follows from
\[
[\phi, S]=-[\phi, x]T + (1-x)[\phi, T]
\]
and from the property of $x$.
\end{proof}

\begin{lemma}(c.f.\ \cite[Proposition 3.5]{QiaoRoe}, \cite[Proposition 4.3]{DWW18}) \label{lem_boundary_isom} Assume that $H_X$ is of infinite multiplicity. We have $K_\ast(\cD_L(H_X\otimes B)^G_{\proper})=0$. Hence the boundary map
\[
\partial \colon K_\ast(\cD_L(H_X\otimes B)^G_{\proper}/ \cC_L(H_X\otimes B)^G_{\proper}) \to K_{\ast+1}(\cC_L(H_X\otimes B)^G_{\proper})
\]
is an isomorphism.
\end{lemma}
\begin{proof} We have $H_X\cong H^{(0)}(X)\otimes l^2(\N)$ and let $U_n$ be isometries on $l^2(\N)$ such that
\[
\sum_{n\geq0} U_nU_n^\ast = 1.
\]
The following map is well-defined on $\cD_L(H_X\otimes B)^G_{\proper}$,
\[
\alpha\colon T \mapsto \sum_{n\geq0} U_n T_{t+n} U_n^\ast.
\]
Indeed, it maps the ideal $C_0([1, \infty), D^*(H_X\otimes B)^G)$ to itself and maps properly supported functions in $\cD_L(H_X\otimes B)^G_{\proper}$ to themselves. Once we have $\alpha$ well-defined, it is routine to show that $K_\ast(\cD_L(H_X\otimes B)^G_{\proper})=0$.
\end{proof}

\begin{lemma}(c.f.\ \cite[Proposition 3.6]{QiaoRoe}, \cite[Proposition 4.3]{DWW18})\label{lem_constant_isom} Assume that $H_X$ is of infinite multiplicity. The evaluation map at $t=1$ induces an isomorphism
\[
\mathrm{ev}_{1 \ast}\colon K_\ast(\cD_T(H_X\otimes B)^G_{\proper}/  \cC_T(H_X \otimes B)^G_{\proper} ) \cong K_\ast(D^*(H_X\otimes B)^G/C^*(H_X\otimes B)^G ).
\]
\end{lemma}
\begin{proof} Let $\cD^0_T(H_X\otimes B)^G_{\proper}$ be the kernel of the evaluation map $\mathrm{ev}_{1}$ on $\cD_T(H_X\otimes B)^G_{\proper}$ and let $\cC^0_T(H_X\otimes B)^G_{\proper}=\cD^0_T(H_X\otimes B)^G_{\proper}\cap  \cC_T(H_X \otimes B)^G_{\proper}$. We have a split short exact sequence
\[
0 \to \cD^0_T(H_X\otimes B)^G_{\proper}/\cC^0_T(H_X\otimes B)^G_{\proper} 
\]
\[
\to  \cD_T(H_X\otimes B)^G_{\proper}/  \cC_T(H_X \otimes B)^G_{\proper}\to  D^*(H_X\otimes B)^G/C^*(H_X\otimes B)^G  \to 0.
\]
Thus, it suffices to show that $K_\ast( \cD^0_T(H_X\otimes B)^G_{\proper})=0$ and $K_\ast( \cC^0_T(H_X\otimes B)^G_{\proper})=0$.

Let $U_n$ be isometries on $H_X\otimes B$ as in the proof of Lemma \ref{lem_boundary_isom}. The following map is well-defined on $\cC^0_T(H_X\otimes B)^G_{\proper}$ and on $\cD^0_T(H_X\otimes B)^G_{\proper}$,
\[
\alpha\colon T \mapsto \sum_{n\geq0} U_n T_{t-n} U_n^\ast
\]
where functions $T$ on $[1,\infty)$ vanishing at $t=1$ are constantly extended to the left by zero. Indeed, it maps properly supported functions in $\cC^0_T(H_X\otimes B)^G_{\proper}$, resp. in $\cD^0_T(H_X\otimes B)^G_{\proper}$, vanishing at $t=1$ to themselves and it is not to hard to see that they form a dense subalgebra of $\cC^0_T(H_X\otimes B)^G_{\proper}$, resp. of $\cD^0_T(H_X\otimes B)^G_{\proper}$.

 Once we have $\alpha$ well-defined, it is routine to show that $K_\ast(\cC^0_T(H_X\otimes B)^G_{\proper})=0$ and $K_\ast(\cD^0_T(H_X\otimes B)^G_{\proper})=0$.
\end{proof}

Let $\pi\colon C_0(X) \to \Linears(H_X\otimes B)$ be the structure map for the $X$-$G$-module $H_X\otimes B$. We define:
\begin{itemize}
\item $C^*(\pi)^G$ is the $C^*$-subalgebra of $\Linears(H_X\otimes B)$ consisting of $G$-equivariant, locally compact operators.
\item $D^*(\pi)^G$ is the $C^*$-subalgebra of $\Linears(H_X\otimes B)$ consisting of $G$-equivariant, pseudo-local operators.
\end{itemize}

The following is a standard fact (c.f.\ \cite[Lemma 6.2]{HigsonRoe95},  \cite[Lemma 5.8]{Roe96}, \cite[Lemma 12.3.2]{HigsonRoe}).
\begin{lemma}\label{lem_eta0} The natural map 
\[
\eta_0\colon D^*(H_X\otimes B)^G/ C^*(H_X\otimes B)^G \to D^*(\pi)^G/C^*(\pi)^G
\]
is bijective (isomorphism).
\end{lemma}
\begin{proof} Surjectivity: Let $T\in D^*(\pi)^G$. For a cut-off function $c$ on $X$, consider
\[
T'= \int_{g\in G}g(c) T g(c) d\mu_G(g).
\]
Then, we have $T' \in D^*(H_X\otimes B)^G$ and $T-T' \in C^*(\pi)^G$. Injectivity: we need to show that $D^*(H_X\otimes B)^G\cap C^*(\pi)^G =  C^*(H_X\otimes B)^G$. Note that this is not a trivial consequence of the definitions. We first claim that $T - T'  \in C^*(H_X\otimes B)^G$ for any $T\in D^*(H_X\otimes B)^G$. In fact, if $T$ is properly supported, this is easy to see since then, $T -T'$ is $G$-equivariant, locally compact and properly supported. The general case follows from the continuity of the map $T\mapsto T-T'$. Let $T \in D^*(H_X\otimes B)^G\cap C^*(\pi)^G$. Then, $T\in C^*(H_X\otimes B)^G$ follows from $T'\in  C^*(H_X\otimes B)^G$ which is easy to see since $T'$ is $G$-equivariant, locally compact and properly supported.
\end{proof}

\begin{proposition}\label{prop_sequence_isom}   Assume that $H_X$ is of infinite multiplicity. We have the following sequence of isomorphisms.
\[
\xymatrix{
K_\ast(D^*(\pi)^G/C^*(\pi)^G )  \ar[r]_-{\cong}^-{\eta_0^{-1}}  & K_\ast(D^*(H_X\otimes B)^G/C^*(H_X\otimes B)^G )  \ar[d]^-{\cong}_-{\iota}   \\ 
&               K_\ast(\cD_T(H_X\otimes B)^G_{\proper}/  \cC_T(H_X \otimes B)^G_{\proper})     \ar[d]^-{\cong}_{\eta^{-1}}    \\
  K_{\ast+1}(\cC_L(H_X\otimes B)^G_{\proper})  & \ar[l]^-{\cong}_-{\partial} K_\ast(\cD_L(H_X\otimes B)^G_{\proper}/ \cC_L(H_X\otimes B)^G_{\proper}) 
}       
\]
where $\iota$ is the natural inclusion as constant functions.
\end{proposition}
\begin{proof} Combine Lemma \ref{lem_eta}, Lemma \ref{lem_boundary_isom}, Lemma \ref{lem_constant_isom} and Lemma \ref{lem_eta0}.
\end{proof}



\section{$\rho_X$ is an isomorphism, part III}

 Let us consider, in general, an $X$-$G$-Hilbert $B$-module $\E$, that is a (countably generated) $G$-Hilbert $B$-module equipped with a non-degenerate representation of the $G$-$C^*$-algebra $C_0(X)$. The support and the  propagation of operators between two $X$-$G$-Hilbert $B$-modules are defined in an obvious way. We let $\tilde \E$ be an $X$-$G$-Hilbert $B$-module $\E\otimes L^2(G)$ equipped with a unitary representation 
 \[
 G\ni g\mapsto u_g\otimes \lambda_g
 \]
 where $u_g$ is the unitary on $\E$ corresponding to $g\in G$ and $\lambda_g$ is the left-translation by $g$ and with the structure map
 \[
 C_0(X) \ni \phi \mapsto \phi\otimes 1.
 \]
For any cut-off function $c$ on $X$, we have a $G$-equivariant isometry $V_c\in \Linears(\E, \tilde \E)$ defined by sending $v \in \E$ to
\[
(V_c(v)\colon h \mapsto  h(c)v ) \in \tilde \E=L^2(G, \E).
\]
The isometry $V_c$ is a strict cover of the identity, i.e. it commutes with $C_0(X)$.
\begin{lemma}\label{lem_absorbing} Let $G$ be a locally compact group and $X$ be a $G$-compact, proper $G$-space. Let $H_X$ be an $X$-$G$-module which is ample as an $X$-module and let $\E$ be the $X$-$G$-Hilbert $B$-module $\tilde H_X\otimes B$. For any $X$-$G$-Hilbert $B$-module $\E_0$, there is a sequence $V_n$ of $G$-equivariant isometries in $\Linears(\E_0, \E)$ such that for any $\phi \in C_0(X)$, 
\begin{itemize}
\item $[V_n, \phi] \in \Compacts(\E_0, \E)$,
\item $\lim_{n\to \infty}\lVert[V_n, \phi]\rVert = 0$.
\end{itemize}
\end{lemma}
\begin{proof} Let $c \in C_c(X)$ be a cut-off function and $X_0$ be the support of $c$.

Let $W_n \in \Linears(\E_0, H_X\otimes B)$ be a sequence of isometries satisfying the following three conditions: for any $\phi$ in $C_0(X)$,
\begin{itemize}
\item $[W_n, \phi] \in \Compacts(\E_0, H_X\otimes B)$,
\item $\limn\lVert[W_n, \phi]\rVert = 0$,
\item $W_nc = \chi W_n c$ for some $\chi \in C_c(X)$. 
\end{itemize}
We explain how we can find such $W_n$. Fix a metric $d$ on $X^+$ and let $\delta>0$ be small enough so that the $\delta$-neighborhood of $X_0$ is relatively compact in $X$.  Let $(\phi_i)_{i\in S}$ be a finite partition of unity in $C(X^+)$ so that the support of each $\phi_i$ is contained in a ball of radius $\delta/2$. Let $X_i$ be the support of $\phi_i$ and let $\E_0^{(i)}=C(X^+)\phi_i\E_0$ which is naturally an $X_i$-module. Then, we have an isometry
\[
\Phi\colon \E_0 \to \bigoplus_{i\in S} \E_0^{(i)}, \,\,\, v\mapsto (\phi^{1/2}_iv).
\]
Let $H_{X_i}=C(X^+)\phi_iH_X$ which is naturally an ample $X_i$-module. We take a sequence $W_n^{(i)} \in \Linears(\E^{(i)}_0, H_{X_i}\otimes B)$ of isometries satisfying, 
\begin{itemize}
\item For each $n$, $(W_n^{(i)})_{i \in S}$ have mutually orthogonal ranges in $H_X\otimes B$,
\item $[W^{(i)}_n, \phi] \in \Compacts(\E_0, H_{X_i}\otimes B)$ for any $\phi$ in $C(X_i)$,
\item $\limn\lVert[W^{(i)}_n, \phi]\rVert = 0$  for any $\phi$ in $C(X_i)$.
\end{itemize}
Such isometries exist because $H_{X_i}$ is an ample $C(X_i)$-module (see \cite[Theorem 2.3]{DWW18}). Let $W_n$ be the composition of $\Phi$ and the (orthogonal) sum of $W_n^{(i)}$. We can see that $W_n$ are isometries satisfying the three conditions above.

We now define a sequence $V_n$ of $G$-equivariant isometries in $\Linears(\E_0, \E)$ as the composition 
\[
V_n=\tilde W_nV_c\colon \E_0 \to \tilde \E_0 \to \tilde H_X\otimes B
\]
where a $G$-equivariant isometry $\tilde W_n\colon  \tilde \E_0 \to \tilde H_X\otimes B$ is of the form
\[
(\tilde W_n\colon g \mapsto  g(W_n) ) \in C_{b, \mathrm{SOT}^*}(G, \Linears(\E_0, H_X\otimes B)).
\]
Here, $g(W_n)=u'_gW_nu^{-1}_{g}$ where $u'_g$ (resp. $u_g$) is the unitary on $H_X\otimes B$, (resp. on  $\E_0$) corresponding to $g$ in $G$. Note that $g(W_n)$ is adjointable although $u'_g$ and $u_g$ may not. We also note that $g\mapsto g(W_n)$ is not necessarily norm-continuous ($W_n$ may not be $G$-continuous). 

We show that $V_n$ satisfies the two desired conditions. Concretely, $V_n$ maps $v\in \E_0$ to
\[
(h\to h(\chi) h(W_n) h(c) v) \in L^2(G, H_X\otimes B).
\] 
Let $\phi\in C_c(X)$. Then, $[V_n, \phi]$ maps $v\in \E_0$ to
\[
(h\to h(\chi) [h(W_n), \phi] h(c) v) \in L^2(G, H_X\otimes B)
\]
and we note that $h(\chi) [h(W_n), \phi] h(c)=0$ for $h\in G\backslash K$ for some compact subset $K$ of $G$ which only depends on the support of $\phi$ (and $c$ and $\chi$) but not on $n$. Now, we have
\[
[V_n, \phi]^\ast [V_n, \phi] = \int_{h\in K} h(c)[h(W_n), \phi]^\ast h(\chi^2) [h(W_n), \phi] h(c) d\mu_G(h).
\]
We note $(h \mapsto [h(W_n), \phi]) \in C_b(G, \Compacts(\E_0, H_X\otimes B))$ since we have $(h\mapsto [W_n, h^{-1}(\phi)]) \in C_b(G, \Compacts(\E_0, H_X\otimes B))$. We also have $\limn\lVert[W_n, h^{-1}(\phi)]\rVert= 0$ uniformly in $h \in K$. From these, we see that 
\begin{itemize}
\item $[V_n, \phi]^\ast [V_n, \phi] \in \Compacts(\E_0)$,
\item $\limn\lVert[V_n, \phi]^\ast [V_n, \phi] \rVert = 0$.
\end{itemize}
We are done.
 \end{proof}

Let $H_X$ be an $X$-$G$-module which is ample as an $X$-module. Let
\[
\pi\colon C_0(X) \to \Linears(\tilde H_X\otimes B)
\]
be the structure map for the $X$-$G$-Hilbert $B$-module $\tilde H_X\otimes B$. Note that we may view $\pi$ as a non-degenerate representation
\[
\pi\colon C_0(X) \to M(\Compacts(\tilde H_X)\otimes B)
\]
and $B'=\Compacts(\tilde H_X)\otimes B$ is $G$-stable in a sense that $B'\otimes \Compacts(L^2(G)) \cong B'$ as $G$-$C^*$-algebras.
 
In general, for any separable $G$-$C^*$-algebra $A$ and for any separable $G$-stable $G$-$C^*$-algebra $B'$, let us say that a non-degenerate, $G$-equivariant representation
\[
\pi\colon A \to M(B')
\]
is non-degenerately $G$-equivariantly absorbing (c.f.\ \cite[Section 2]{Thomsen01}) if for any non-degenerate $G$-equivariant representation 
\[
\pi_0\colon A\to M(B'),
\]
there is a sequence $u_n$ of $G$-equivariant unitaries in $\Linears(B'\oplus B', B')$ such that for any $a\in A$,
\begin{itemize}
\item $u_n(\pi_0(a)\oplus \pi(a)) - \pi(a)u_n \in \Compacts(B'\oplus B', B')$,
\item $\limn\lVert u_n(\pi_0(a)\oplus \pi(a)) - \pi(a)u_n \rVert = 0$.
\end{itemize}

It is routine to deduce from Lemma \ref{lem_absorbing} that the structure map $\pi$ for $\tilde H_X\otimes B$ is non-degenerately $G$-equivariantly absorbing.
\begin{proposition}\label{prop_absorbing} For any $X$-$G$-module $H_X$ which is ample as an $X$-module, the structure map 
\[
\pi\colon C_0(X) \to M(B')
\]
for the $X$-$G$-module $\tilde H_X\otimes B$ where $B'=\Compacts(\tilde H_X)\otimes B$, is non-degenerately $G$-equivariantly absorbing.
 \end{proposition}
 \begin{proof} We use a standard trick (see the proof of \cite[Corollary 2, p341]{Arveson77}). Let $\pi_0\colon C_0(X) \to M(B')$ be a non-degenerate, $G$-equivariant  representation and let $\pi_0^\infty \colon C_0(X) \to M(B')$ be its amplification so that $\pi_0\oplus \pi_0^\infty$ is unitarily equivalent to $\pi_0^\infty$. 
 Lemma \ref{lem_absorbing} gives a sequence $v_n$ of $G$-equivariant isometries in $\Linears(B', B')$ such that for any $\phi \in C_0(X)$,
\begin{itemize}
\item $v_n\pi^\infty_0(\phi) - \pi(\phi)v_n \in \Compacts(B', B')$,
\item $\limn\lVert v_n\pi^\infty_0(\phi) - \pi(\phi)v_n \rVert = 0$.
\end{itemize}
Let $p_n=1-v_nv_n^\ast$ and $\sigma_n$ be a $G$-equivariant c.c.p.\ map from $C_0(X)$ to $\Linears(p_nB')$ defined by $\sigma_n(\phi)=p_n\pi(\phi)p_n$. We set $w_n$ to be a unitary in $\Linears(p_nB' \oplus B', B')$ given by the direct sum of  the identity map on $p_nB'$ and $v_n$. Then, $w_n$ is a $G$-equivariant unitary in $\Linears(p_nB' \oplus B', B')$ such that for any $\phi \in C_0(X)$,
\begin{itemize}
\item $w_n(\sigma_n(\phi)\oplus \pi^\infty_0(\phi)) - \pi(\phi)w_n \in \Compacts(p_nB' \oplus B', B')$,
\item $\limn\lVert w_n(\sigma_n(\phi)\oplus \pi^\infty_0(\phi)) - \pi(\phi)w_n \rVert = 0$.
\end{itemize}
Define unitaries $u_n$ in $\Linears(B' \oplus B', B')$ by the composition
\[
\xymatrix{
B' \oplus B'   \ar[r]^-{1_{B'}\oplus w_n^\ast} &  B' \oplus (p_nB' \oplus B') \ar[r]^-{\cong} & p_nB' \oplus (B' \oplus B') \ar[r]^-{1_{p_nB'} \oplus w} & p_nB' \oplus B'  \ar[r]^-{w_n} &   B'
}
\]
Where $w\in \Linears(B'\oplus B', B')$ is a unitary equivalence from $\pi_0\oplus \pi_0^\infty$ to $\pi_0^\infty$. Then, we have for any $\phi \in C_0(X)$,
\begin{itemize}
\item $u_n(\pi_0(\phi)\oplus \pi(\phi)) - \pi(\phi)u_n \in \Compacts(B'\oplus B', B')$,
\item $\limn\lVert u_n(\pi_0(\phi)\oplus \pi(\phi)) - \pi(\phi)u_n \rVert = 0$.
\end{itemize}
 \end{proof}
 
For any $G$-equivariant representation
\[
\pi\colon A \to M(B'),
\]
we define $\cD(\pi)^G$ to be the $C^*$-algebra of $G$-equivariant elements $x$ in $M(B')$ such that $[x, \pi(a)] \in B'$ for any $a\in A$ and $\cC(\pi)^G$ to be the $C^*$-algebra of $G$-equivariant elements $x$ in $M(B')$ such that $x\pi(a), \pi(a)x \in B'$ for any $a\in A$.

We have a group homomorphism
\[
\Theta\colon K_1(\cD(\pi)^G/\cC(\pi)^G) \to KK_0^G(A, B')
\]
which sends a unitary $u$ in $M_n(\cD(\pi)^G/\cC(\pi)^G)$ to the following even Kasparov triple for $KK_0^G(A, B')$,
\[
( B'^{\oplus n}\oplus B'^{\oplus n}, \pi^{\oplus n } \oplus \pi^{\oplus n }, \,    \begin{bmatrix} 0  & v\\  v^\ast  & 0 \end{bmatrix} )
\]
where $v$ is any lift of $u$ in $M_n(\cD(\pi)^G)$. Similarly, we have a group homomorphism
\[
\Theta \colon K_0(\cD(\pi)^G/\cC(\pi)^G) \to KK_1^G(A, B')
\]
which sends a projection $p$ in $M_n(\cD(\pi)^G/\cC(\pi)^G)$ to the following odd Kasparov triple for $KK_1^G(A, B')$,
\[
( B'^{\oplus n}, \pi^{\oplus n },    2P-1)
\]
where $P$ is any (self-adjoint) lift of $p$ in $M_n(\cD({\pi})^G)$.

We recall some terminologies from Section 3 \cite{Thomsen05}. An even Kasparov triple $\mathcal{E}=(E, \phi, F)$ for $KK_0^G(A, B')$ is called elementary if $E=B'\oplus B'$ and it it called essential if $\phi(A)E=E$. Similarly, an odd Kasparov triple $\mathcal{E}=(E, \phi, P)$ for $KK_1^G(A, B')$ is called elementary if $E=B'$ and it it called essential if $\phi(A)E=E$. 

\begin{proposition}(\cite[Theorem 3.9, Theorem 4.2]{Thomsen05}) Let $A$ be a separable proper $G$-$C^*$-algebra and $B'$ be a separable, $G$-stable $G$-$C^*$-algebra. 
\begin{enumerate}
\item[(1a)] Every element of $KK^G_0 (A, B')$ is represented by an even Kasparov triple which is both elementary and essential.
\item [(1b)] Two elementary and essential even Kasparov triples, $\cE_1$ and $\cE_2$ define the same element of $KK^G_0 (A, B')$ if and only if there are degenerate even Kasparov triples, $\cD_1$ and $\cD_2$ for $KK^G_0 (A, B')$ which are both elementary and essential, such that $\cE_1\oplus \cD_1$ is operator homotopic to $\cE_2\oplus \cD_2$.
\item [(2a)] Every element of $KK^G_1 (A, B')$ is represented by an odd Kasparov triple which is both elementary and essential.
\item [(2b)] Two elementary and essential odd Kasparov triples, $\cE_1$ and $\cE_2$ define the same element of $KK^G_1 (A, B')$ if and only if there are degenerate odd Kasparov triples, $\cD_1$ and $\cD_2$ for $KK^G_1 (A, B')$ which are both elementary and essential, such that $\cE_1\oplus \cD_1$ is operator homotopic to $\cE_2\oplus \cD_2$.
\end{enumerate}
\end{proposition}
\begin{proof} In \cite{Thomsen05}, the proof was given for $A$ which is $G$-stable but the only property of $A$ which is used for the proof is that for any $G$-equivariant representation $\phi\colon A \to \Linears(E)$ on a $G$-Hilbert $B'$-module $E$ with $\phi(A)E=E$, we have $E\oplus B'\cong B'$ (see \cite[Theorem 2.8, Corollary 2.9]{Thomsen05}). Any proper $G$-$C^*$-algebra has this property (see \cite[Proposition 8.6]{Meyer00}).
\end{proof}

\begin{theorem}(c.f.\  \cite[Theorem 3.2]{Thomsen01}) \label{thm_KKduality} Let $A$ be a separable proper $G$-$C^*$-algebra and $B'$ be a separable, $G$-stable $G$-$C^*$-algebra. Suppose a non-degenerate $G$-equivariant representation $\pi\colon A \to M(B')$ is non-degenerately $G$-equivariantly absorbing. Then, the group homomorphisms $\Theta$ induce isomorphisms 
\[
K_1(\cD(\pi)^G/\cC(\pi)^G)\cong KK_0^G(A, B') \,\,\, \text{and} \,\,\,  K_0(\cD(\pi)/\cC(\pi)^G)\cong KK_1^G(A, B').
\]
\end{theorem}

\begin{proof} The proof of \cite[Theorem 3.2]{Thomsen01} works with minor changes so we shall be brief. We give a proof for the isomorphism $K_1(\cD(\pi)^G/\cC(\pi)^G)\cong KK_0^G(A, B')$. The other one is parallel. Surjectivity: Take any even elementary, essential, Kasparov triple $\cE=(B'\oplus  B', \pi_0\oplus \pi_1,  \begin{bmatrix} 0  & x\\  x^\ast  & 0 \end{bmatrix} )$ for $KK^G_0 (A, B')$. By adding an essential, degenerate Kasparov triple 
\[
\cD= (B'^\infty \oplus  B'^\infty, (\pi^\infty_0\oplus\pi^\infty_1)\oplus  (\pi^\infty_0\oplus \pi^\infty_1),  \begin{bmatrix} 0  & 1\\  1  & 0 \end{bmatrix} )  \oplus  (B'\oplus  B', \pi\oplus \pi,  \begin{bmatrix} 0  & 1\\  1  & 0 \end{bmatrix} )
\]
to $\cE$, where $(\pi^\infty_0\oplus \pi^\infty_1)$ is the infinite direct sum of $(\pi_0 \oplus \pi_1)$, we see that the triple $\cE\oplus \cD$ is isomorphic to the essential triple $\cE_1$ of the form $(B'\oplus  B', \pi'\oplus \pi',  \begin{bmatrix} 0  & v\\  v^\ast  & 0 \end{bmatrix} )$ where $\pi'(a)-\pi(a) \in B'$ for any $a\in A$, i.e. $\cE_1$ is a compact perturbation of the triple $(B'\oplus  B', \pi\oplus \pi,  \begin{bmatrix} 0  & v\\  v^\ast  & 0 \end{bmatrix})$. Using a cut-off function (as usual), the latter triple is a compact perturbation of the triple of the same form where $v$ is $G$-equivariant. Such a triple is in the image of $\Theta$. 

Injectivity: We may assume $\pi$ is of infinite multiplicity. We take a unitary $u\in \cD(\pi)^G/\cC(\pi)^G$ and let $v \in \cD(\pi)^G$ be any lift of $u$ (it is enough to consider $1\times 1$-matrix since $\cD(\pi)^G/\cC(\pi)^G$ is properly infinite.). Suppose $\Theta(u)=0$ in $KK^G_0(A, B')$. Then, there are essential, degenerate triples $\cD_1, \cD_2$ such that  $(B'\oplus  B', \pi\oplus \pi,  \begin{bmatrix} 0  & v\\  v^\ast  & 0 \end{bmatrix} )\oplus \cD_1$ is operator homotopic to $(B'\oplus  B', \pi\oplus \pi,  \begin{bmatrix} 0  & 1\\  1  & 0 \end{bmatrix} )\oplus \cD_2$. By adding infinite copies of $\cD_1\oplus \cD_2$ to both, we may assume that $\cD_1=\cD_2$. We can further arrange that $\cD_1=\cD_2$ are of the form $(B'\oplus  B', \lambda\oplus \lambda,  \begin{bmatrix} 0  & 1\\  1  & 0 \end{bmatrix} )$ with $\lambda$ of infinite multiplicity. By adding $(B'\oplus  B', \pi\oplus \pi,  \begin{bmatrix} 0  & 1\\  1  & 0 \end{bmatrix} )$, we can assume that there is a $G$-equivariant unitary $w\in M(B')$ such that $w\lambda(a)w^\ast - \pi(a) \in B'$ for any $a\in A$. From here, using the same reasoning of the proof \cite[Theorem 3.2]{Thomsen01}, we get a $G$-equivariant norm-continuous path $G_t$ in the matrix algebra $M_n(M(B'))$ such that $G_0=1_n$ and $G_1= v\oplus 1_{n-1}$ and 
\[
(G_t^\ast G_t - 1_n)\pi(a), (G_tG_t^\ast - 1_n)\pi(a), [\pi(a), G_t] \in M_n(B')
\]
for all $t$ and for all $a \in A$. The $G$-equivariance can be arranged by perturbing the operator homotopy into $G$-equivariant one using a cut-off function $c$ for a proper $G$-space $X$ for which $A$ is a $G$-$C_0(X)$-algebra. More precisely, $G_t$ would be a homotopy between $G_0=1_n$ and $G_1= v'\oplus 1_{n-1}$ where $v'=\int_{g\in G}g(c)g(v)g(c)d\mu_G(g)=\int_{g\in G}g(c)vg(c)d\mu_G(g)$ but $v'$ is still a lift of $u$. This gives a path of unitaries in $M_n(\cD(\pi)^G/\cC(\pi)^G)$ connecting $u\oplus 1_{n-1}$ to $1_n$.
\end{proof}

\begin{corollary}\label{cor_canonical_isom} Let $G$ be a locally compact group, $X$ be a $G$-compact proper $G$-space and $B$ be a separable $G$-$C^*$-algebra. Let $H_X$ be an $X$-$G$-module which is ample as an $X$-module. Let $\pi\colon C_0(X) \to \Linears(\tilde H_X\otimes B)$ be the structure map for the $X$-$G$-Hilbert $B$-module $\tilde H_X\otimes B$. Then, there are canonical isomorphisms
\[
\Theta\colon K_1(D^*(\pi)^G/C^*(\pi)^G) \cong  KK_0^G(C_0(X), B),
\]
\[
\Theta\colon K_0(D^*(\pi)^G/C^*(\pi)^G) \cong  KK_1^G(C_0(X), B).
\]
\end{corollary}

\begin{corollary}\label{cor_sequence_isom}  Let $G$ be a locally compact group, $X$ be a $G$-compact proper $G$-space and $B$ be a separable $G$-$C^*$-algebra. Let $H_X$ be a universal $X$-$G$-module. Assume that $H_X$ is of infinite multiplicity. Let $\pi\colon C_0(X) \to \Linears(\tilde H_X\otimes B)$ be the structure map for the $X$-$G$-Hilbert $B$-module $\tilde H_X\otimes B$. We have the following sequence of isomorphisms.
\[
\xymatrix{
KK_\ast^G(C_0(X), B)   & \ar[l]^-{\cong}_-{\Theta}  K_{\ast+1}(D^*(\pi)^G/C^*(\pi)^G )     \ar[d]^-{\cong}_-{\eta_0^{-1}} \\
&   K_{\ast+1}(D^*(\tilde H_X\otimes B)^G/C^*(\tilde H_X\otimes B)^G )  \ar[d]^-{\cong}_-{\iota}   \\ 
&               K_{\ast+1}(\cD_T(\tilde H_X\otimes B)^G_{\proper}/  \cC_T(\tilde H_X \otimes B)^G_{\proper} )    \ar[d]^-{\cong}_{\eta^{-1}}    \\
& K_{\ast+1}(\cD_L(\tilde H_X\otimes B)^G_{\proper}/ \cC_L(\tilde H_X\otimes B)^G_{\proper})   \ar[d]^-{\cong}_-{\partial} \\
 K_\ast(RL^*_u(H_X\otimes B)\rtimes_rG) \ar[r]^-{\cong}_-{\rho_\ast} & K_{\ast}(\cC_L(\tilde H_X\otimes B)^G_{\proper})  
}       
\]
In particular, for any universal $X$-$G$-module $H_X$, there is a canonical isomorphism 
\[
 K_\ast(RL^*_u(H_X\otimes B)\rtimes_rG) \cong KK_\ast^G(C_0(X), B).  
\]
\end{corollary}
\begin{proof} Combine Theorem \ref{thm_magic}, Proposition \ref{prop_sequence_isom} and Corollary \ref{cor_canonical_isom}.
\end{proof}



\section{$\rho_X$ is an isomorphism, part IV}

Now we are ready to show that the group homomorphism 
\[
\rho_X\colon K_\ast(RL^*_u(H_X\otimes B)\rtimes_rG) \to E_\ast^G(C_0(X), B) \cong KK_\ast^G(C_0(X), B)
\]
is an isomorphism for any $G$-compact proper $G$-space $X$ where $H_X$ is a universal $X$-$G$-module. 

For any $X$-$G$-module $H_X$, we have a group homomorphism
\[
\iota_X\colon K_\ast(C_{L, u}^*(H_X\otimes B)^G ) \to E_\ast^G(C_0(X), B) \cong KK_\ast^G(C_0(X), B)
\]
defined in the same way as $\rho_X$ using a canonical asymptotic morphism
\[
\pi_X\otimes \iota \colon C_0(X) \otimes C_{L, u}^*(H_X\otimes B)^G \to \as(\Compacts(H_X\otimes B))
\]
such that the image of $\phi\otimes T$ is represented by
\[
\phi T \in C_b([1, \infty), \Compacts(H_X\otimes B)).
\]
It is clear that we have $\rho_X= \iota_X\circ  \rho_\ast$, that is $\rho_X$ factors through the map  
\[
\rho_\ast \colon K_\ast(RL^*_u(H_X\otimes B)\rtimes_rG) \to K_\ast(C_{L, u}^*(\tilde H_X\otimes B)^G)
\]
induced by the right-regular representation which is an isomorphism whenever $H_X$ is universal. 

\begin{lemma}(c.f.\ \cite[Theorem 5.1]{DWW18}) \label{lem_isom_commutes} Let $G$ be a locally compact group, $X$ be a $G$-compact proper $G$-space and $B$ be a separable $G$-$C^*$-algebra. Let $H_X$ be any $X$-$G$-module. Let $\pi\colon C_0(X) \to \Linears(H_X\otimes B)$ be the structure map for the $X$-$G$-Hilbert $B$-module $H_X\otimes B$. Then, the following diagram commutes.
\[
\xymatrix{
 K_{\ast+1}(D^*(H_X\otimes B)^G/C^*(H_X\otimes B)^G )  \ar[d]^-{\cong}_-{\eta_0}  \ar[r]^-{\iota}  & K_{\ast+1}(\cD_T(H_X\otimes B)^G_{\proper}/  \cC_T(H_X \otimes B)^G_{\proper})     \ar[d]^-{\cong}_{\eta^{-1}}      \\
K_{\ast+1}(D^*(\pi)^G/C^*(\pi)^G ) \ar[d]^-{}_-{\Theta}     &   K_{\ast+1}(\cD_L(H_X\otimes B)^G_{\proper}/ \cC_L(H_X\otimes B)^G_{\proper}) \ar[d]^-{}_-{\partial}   \\ 
KK_\ast^G(C_0(X), B) ) \ar[d]^-{\cong} &       K_{\ast}(\cC_L(H_X\otimes B)^G_{\proper})      \ar[d]^-{}_-{=}       \\ 
E_\ast^G(C_0(X), B) &  \ar[l]_-{\iota_X}  K_{\ast}(C^\ast_{L, u}(H_X\otimes B)^G_{\proper}).    
}       
\]
\end{lemma}
\begin{proof} The argument of the proof of \cite[Theorem 5.1]{DWW18} works with minor changes. We give a proof for the case $\ast=0$. The case of $\ast=1$ is simpler and this is the one of the two cases which is explained carefully in \cite{DWW18}. Take a unitary $\dot u$ in $D^*(H_X\otimes B)^G/C^*(H_X\otimes B)^G$ (for simplicity we are taking a unitary in the $1\times 1$-matrix algebra). We first compute the image of $\dot u$ in $E_0^G(C_0(X), B)$ by the clock-wise composition. The functional calculus for $\dot u$ gives a $\ast$-homomorphism
\[
\Sigma \ni f \mapsto f(\dot u) \in D^*(H_X\otimes B)^G/C^*(H_X\otimes B)^G
\]
where $\Sigma \cong C_0(0, 1)$ is identified as $C_0(S^1-\{1\})$. We let
\[
\Sigma\ni f\mapsto f(u) \in D^*(H_X\otimes B)^G
\]
be its (not necessarily linear) continuous, bounded lift to $D^*(H_X\otimes B)^G$ (which exists by Bartle--Graves Theorem \cite[Theorem 4]{BartleGraves}). This is an abuse of notations and $f(u)$ is not the functional calculus applied to some element $u$. We may compose this map $f\mapsto f(u)$ with a c.c.p.\ map
\[
x \to \int_{g\in G}g(c)xg(c)d\mu_G(g)
\]
on $D^*(H_X\otimes B)^G$ to assume that the image of $\Sigma$ is properly supported as a set. Note that such a composition is still a lift of $f \mapsto f(\dot u)$. Let $x \in  \cC_T(H_X \otimes B)^G_{\proper}$ be as given by Lemma \ref{lem_ex} for 
\[
D=\{ f(u) \in D^*(H_X\otimes B)^G \subset \cD_T(H_X\otimes B)^G_{\proper}   \mid f\in \Sigma \}.
\]
Then, 
\[
\Sigma\ni f\mapsto (1-x)f(u) \in \cD_L(H_X\otimes B)^G_{\proper}
\]
is a lift of the $\ast$-homomorphism from $\Sigma$ to $\cD_L(H_X\otimes B)^G_{\proper}/\cC_L(H_X\otimes B)^G_{\proper}$ associated to the unitary $\dot u$, or more precisely the unitary $\iota(\dot u) \in \cD_T(H_X\otimes B)^G_{\proper}/  \cC_T(H_X \otimes B)^G_{\proper} = \cD_L(H_X\otimes B)^G_{\proper}/ \cC_L(H_X\otimes B)^G_{\proper}$. The boundary map $\partial$ sends the element $[\dot u]$ in $K_1(\cD_L(H_X\otimes B)^G_{\proper}/\cC_L(H_X\otimes B)^G_{\proper})$ to the element in $K_0(\cC_L(H_X\otimes B)^G_{\proper})$ represented by an asymptotic morphism
\[
h\otimes f \mapsto h(v_t)(1-x)f(u)
\]
from $\Sigma^2$ to $\cC_L(H_X\otimes B)^G_{\proper}$. Here, we let $M$ be a separable $C^*$-subalgebra of $\cD_L(H_X\otimes B)^G_{\proper}$ generated by $(1-x)f(u)$ and $v_t$ is a continuous approximate unit for $M\cap \cC_L(H_X\otimes B)^G_{\proper}$, quasi-central with respect to $M$. The morphism $\iota_X$ sends this element in $K_0(\cC_L(H_X\otimes B)^G_{\proper})$ to the element in $E_0^G(C_0(X), B)$ represented by an asymptotic morphism 
\[
h\otimes f\otimes \phi \mapsto h(v_t(s(t)))(1-x(s(t)))f(u) \phi
\]
from  $\Sigma^2\otimes C_0(X)$ to $\Compacts(H_X\otimes B)$ where $t\mapsto s(t)$ is a continuous, increasing map on $[1 ,\infty)$ with $s(t)\to \infty$ sufficiently fast as $t\to \infty$.

On the other hand, the downward composition sends the element $[\dot u]$ in $K_1(D^*(H_X\otimes B)^G/C^*(H_X\otimes B)^G)$ to the element in $E_0^G(C_0(X), B)$ represented by an asymptotic morphism 
\[
h\otimes f\otimes \phi \mapsto h(u_t) f(u)\phi
\]
from $\Sigma^2\otimes C_0(X)$ to $\Compacts(H_X\otimes B)$ where $u_t$ is an asymptotically $G$-equivariant, continuous approximate (see \cite[Definition 6.2]{HigsonKasparov}) in $\Compacts(H_X\otimes B)$ which is quasi-central with respect to $C_0(X)$ and $ f(u)$ for $f\in \Sigma$. We can take $u_t$ so that we have
\[
\lVert(1-u_t)x(s(t))\phi \rVert \to 0 
\]
as $t\to \infty$ for any $\phi \in C_0(X)$. Then, the above asymptotic morphism is asymptotic to the one
\[
h\otimes f\otimes \phi \mapsto h(u_t)(1-x(s(t)))f(u)\phi.
\]
The two asymptotic morphisms 
\[
h\otimes f\otimes \phi \mapsto h(u_t)(1-x(s(t)))f(u)\phi, \,\,\, h\otimes f\otimes \phi \mapsto h(v_t(s(t)))(1-x(s(t)))f(u) \phi
\]
are homotopic through an asymptotic morphism
\[
h\otimes f\otimes \phi \mapsto h(w^{r}_t)(1-x(s(t)))f(u)\phi,
\]
from $\Sigma^2\otimes C_0(X)$ to $\Compacts(H_X\otimes B)\otimes C[0 ,1]$ where 
\[
w^{r}_t = (1-r)v_t(s(t)) + ru_t
\]
for $0\leq r \leq 1$. That this is well-defined can be seen from that 
\[
f \mapsto (1-x(s(t)))f(u)
\]
is asymptotically a $\ast$-homomorphism modulo the product of $\phi\in C_0(X)$ and $h(w^{r}_t)$ and from that 
\[
\lVert [h(w^{r}_t)  (1-x(s(t)))f(u), \phi]\rVert \to 0, \,\,\, \lVert[h(w^{r}_t),  (1-x(s(t)))f(u)]\rVert \to 0
\]
as $t\to \infty$. 
\end{proof}

\begin{theorem} \label{thm_main_isom} Let $G$ be a locally compact group. The natural transformation
\[
\rho_X\colon  \bD_\ast^{B, G}(X)  \to \varinjlim_{Y\subset X, \mathrm{Gcpt}}E_\ast^G(C_0(Y), B)  \cong  \varinjlim_{Y\subset X, \mathrm{Gcpt}}KK_\ast^G(C_0(Y), B) 
\]
is an isomorphism for any proper $G$-space $X$ and for any separable $G$-$C^*$-algebra $B$.
\end{theorem}
\begin{proof} The case when $X$ is $G$-compact follows from Proposition \ref{prop_ucsameB}, Corollary \ref{cor_sequence_isom} and Lemma \ref{lem_isom_commutes}. The general case follows since the transformation is natural with respect to $X$ and both sides are representable by $G$-compact subsets.
\end{proof}

\begin{theorem}  \label{thm_main_equivalent} Let $G$ be a locally compact group, $X$ be a proper $G$-space and $B$ be a separable $G$-$C^*$-algebra. The forget-control map 
\[
\cf\colon \bD_\ast^{B, G}(X) \to K_\ast(B\rtimes_rG) 
\]
is naturally equivalent to the Baum--Connes assembly map 
\[
\mu_{X}^{B, G}\colon \varinjlim_{Y\subset X, \mathrm{Gcpt}}KK_\ast^G(C_0(X), B) \to K_\ast(B\rtimes_rG). 
\]
That is, there is a natural isomorphism 
\[
\rho_X\colon  \bD_\ast^{B, G}(X) \to   \varinjlim_{Y\subset X, \mathrm{Gcpt}}KK_\ast^G(C_0(Y), B) 
\]
of the functors from $\mathcal{PR}^G$ to $\mathcal{GA}$ and the following diagram commutes
\begin{equation*}
\xymatrix{ \bD_\ast^{B, G}(X)   \ar[dr]^{\rho_X}_-{\cong}  \ar[rr]^{\cf} & &  K_\ast(B\rtimes_rG)  \\
   &  \varinjlim_{Y\subset X, \mathrm{Gcpt}}KK_\ast^G(C_0(Y), B). \ar[ur]^{\mu^{B, G}_X}   & 
   }
   \end{equation*}
\end{theorem}
\begin{proof} Combine Theorem \ref{thm_forget_factor} and Theorem \ref{thm_main_isom}.
\end{proof}

As before, let $RL^0_c(H_X\otimes B)$ be the kernel of the evaluation map $\mathrm{ev}_1$ on $RL_c^\ast(H_X\otimes B)$ (see the end of Section \ref{sec_forget}).
  
 \begin{corollary} \label{cor_N} Let $G$ be a locally compact group and $B$ be a separable $G$-$C^*$-algebra. The Baum--Connes assembly map $\mu^{B, G}_r$ is an isomorphism if and only if
 \[
 K_\ast(RL^0_c(H_X\otimes B)\rtimes_rG)=0
 \]
for a universal $X$-$G$-module $H_X$ for $X=\EG$.
 \end{corollary}



\section{$X$-$G$-localized element in $KK^G(\bC, \bC)$}

We recall some materials on Kasparov's $G$-equivariant $KK$-theory from \cite{Kasparov88} (see also \cite{Blackadar}). The graded (minimal) tensor product is denoted by $\hat\otimes$.

Let $G$ be a second countable, locally compact group. For any (not necessarily separable) graded $G$-$C^*$-algebras $A$ and $B$, Kasparov defines the abelian group $KK^G(A, B)=KK_0^G(A, B)$ (\cite[Definition 2.2]{Kasparov88} ) as the group of homotopy classes (up to isomorphisms) of even Kasparov triples $(E, \pi, T)$ where $E$ is a countably generated, graded $G$-Hilbert $B$-module $E$ which is equipped with a (graded, $G$-equivariant) $\ast$-homomorphism $\pi\colon A\to \Linears(E)$ and where $T$ is an odd, $G$-continuous operator in $\Linears(E)$ such that for any $a\in A$ and $g\in G$, (we write $a=\pi(a)$)
\[
a(1-T^2), \,\,[a, T], \,\,a(T-T^\ast),\,\, a(g(T)-T) \in \Compacts(\E).
\]
We often assume that $T$ is self-adjoint without loss of generality. The homotopy is defined by the even Kasparov triple for $KK^G(A, B[0 ,1])$ for $B[0, 1]=B\otimes C[0 ,1]$ and the group structure is given by the direct sum operation. For any $\ast$-homomorphisms $\phi\colon A_1 \to A_2$, $\psi\colon B_1 \to B_2$, we have canonically defined homomorphisms (\cite[Definition 2.5]{Kasparov88}
\[
\phi^\ast\colon KK^G(A_2, B) \to KK^G(A_1, B), \,\,\,  \psi_\ast\colon KK^G(A, B_1) \to  KK^G(A, B_2)
\]
and the group $KK^G(A, B)$ is homotopy invariant in both variables. If $D$ is $\sigma$-unital, we have a canonically defined homomorphism
\[
\sigma_D\colon KK^G(A, B) \to KK^G(A\hat\otimes D, B\hat\otimes D)
\]
which sends $(E, \pi, T)$ to $(E\hat\otimes D, \pi\hat\otimes \mathrm{id}_D, T\hat\otimes 1)$.  

For $A$, separable, Kasparov defines the well-defined, bilinear, product law (Kasparov product) (see \cite[Definition 2.10, Theorem 2.11]{Kasparov88})
\[
KK^G(A, B_1) \otimes KK^G(B_1, B_2) \to KK^G(A, B_2), \,\,\, (x_1, x_2) \mapsto x_1\otimes_{B_1}x_2 
\]
for any $B_1, B_2$. The Kasparov product satisfies several functorial properties, assuming separability or $\sigma$-unital for the relevant slots (see \cite[Theorem 2.14]{Kasparov88}). 

The descent homomorphism \cite[Section 3.11]{Kasparov88}
\[
j^G_r\colon KK^G(A, B) \to KK(A\rtimes_rG, B\rtimes_rG)
\]
is defined for any $A, B$ and it satisfies functorial properties, assuming separability or $\sigma$-unital for the relevant slots (see \cite[Theorem 3.11]{Kasparov88}).

The abelian group $KK^G_1(A, B)$ is defined to be 
\[
KK^G_1(A, B) = KK^G(A\hat\otimes \bC_1, B)= KK^G(A, B\hat\otimes \bC_1)
\]
where $\bC_1$ is the first Clifford algebra. We define
\[
K_\ast(A) = KK_\ast(\bC, A)
\]
for any graded $C^*$-algebra $A$ and when $A$ is ungraded, this group is canonically isomorphic to the topological K-theory of $A$.

In particular, a cycle for the commutative ring $R(G)=KK^G(\bC, \bC)$ is a pair $(H ,T)$ of an odd, self-adjoint, $G$-continuous operator $T$ on a (separable) graded $G$-Hilbert space $H$ satisfying for any $g\in G$,
\[
1-T^2, g(T)-T \in \Compacts(H).
\]
We call such a pair, a Kasparov cycle for $KK^G(\bC, \bC)$. A cycle $(\bC, 0)$ defines the unit of the ring $R(G)$, denoted by $1_G$,

For a proper $G$-space $X$, a graded $X$-$G$-module $H_X$ is a graded $G$-Hilbert space equipped with a non-degenerate (graded, $G$-equivariant) representation of $C_0(X)$. It is just a pair $H_X=H_X^{(0)}\oplus H_X^{(1)}$ of $X$-$G$-modules. For any graded $X$-$G$-module $H_X$, the representable localization algebra $RL^*_c(H_X)$ is naturally a graded $G$-$C^*$-algebra.

\begin{definition}\label{def_XGlocalized} An $X$-$G$-localized Kasparov cycle for $KK^G(\bC, \bC)$ is a pair $(H_X, T)$ of a graded $X$-$G$-module $H_X$ and an odd, self-adjoint, $G$-continuous element $T$ in the multiplier algebra $M(RL_c^*(H_X))$ satisfying for any $g\in G$,
\[
1-T^2, \,\,\, g(T)-T \in RL^*_c(H_X).
\]
\end{definition}
\begin{remark} Although $RL^*_c(H_X)$ is not $\sigma$-unital, one may think $(RL_c^*(H_X), T)$ as a cycle for $KK^G(\bC, RL^*_c(H_X))$ (see \cite[Section 3]{Skandalis85}). 
\end{remark}

The evaluation 
\[
\mathrm{ev}_{t=1}\colon RL^*_c(H_X) \to \Compacts(H_X)
\]
extends to
\[
\mathrm{ev}_{t=1}\colon M(RL^*_c(H_X)) \to \Linears(H_X).
\]
For any $T \in M(RL^*_c(H_X))$, we write $T_1 \in \Linears(H_X)$, its image by $\mathrm{ev}_{t=1}$.
 
\begin{lemma} Let $X$ be a proper $G$-space and $(H_X, T)$ be an $X$-$G$-localized Kasparov cycle for $KK^G(\bC, \bC)$. Then, a pair $(H_X, T_1)$ is a cycle for $KK^G(\bC, \bC)$. 
\end{lemma} 

\begin{definition}\label{def_XGlocalized_element} Let $X$ be a proper $G$-space. We say that an element $x$ in $KK^G(\bC, \bC)$ is $X$-$G$-localized if there is an $X$-$G$-localized Kasparov cycle $(H_X, T)$ for $KK^G(\bC, \bC)$ such that 
\[
[H_X, T_1] = x  \,\,\text{in}\,\,\, KK^G(\bC, \bC).
\]
\end{definition}

We say that $x\in KK^G(\bC, \bC)$ factors through a $G$-$C^*$-algebra $B$ if there is $y \in KK^G(\bC, B)$ and $z\in KK^G(B, \bC)$ such that $x=y\otimes_Bz$. By definition, $x\in KK^G(\bC, \bC)$ is the gamma element if it factors through a (separable) proper $G$-$C^*$-algebra $A$ and it satisfies $x=1_K$ in $R(K)$ for any compact subgroup $K$ of $G$.

\begin{proposition}\label{prop_XGlocalize}  Suppose that an element $x \in KK^G(\bC, \bC)$ factors through a separable $G$-$C_0(X)$-algebra $A$ for a proper $G$-space $X$. Then $x$ is $X$-$G$-localized.
\end{proposition}

The following is an immediate corollary.

\begin{theorem}\label{thm_XGgamma0} The gamma element $\gamma$ for $G$, if exists, is $X$-$G$-localized for $X=\EG$.
\end{theorem}

Before giving a proof of Proposition \ref{prop_XGlocalize}, we prove two lemmas.  

Let $A$ be a graded $G$-$C_0(X)$ algebra and $(H, \pi, F)$ be a Kasparov triple for $KK^G(\bC, A)$. If $\pi$ is non-degenerate, $\pi$ induces a natural non-degenerate representation of $C_0(X)$ on $H$. We view $H$ naturally as a (graded) $X$-$G$-module through this representation and set $H_X=H$. Any element $a$ in $M(A)$ commutes with $C_0(X)$ and hence $a$ (as a constant function) is naturally an element in $M(RL^*_c(H_X))\subset \Linears(H_X\otimes C_0[1, \infty))$. 

\begin{lemma}\label{lem_XG1} Let $X$ be a proper $G$-space and $A$ be a graded, separable $G$-$C_0(X)$-algebra. Let $(H, \pi, F_0)$ be a Kasparov triple for $KK^G(A, \bC)$ with $\pi$ non-degenerate. We view $H$ naturally as a graded $X$-$G$-module through this representation and set $H_X=H$. Then, there is an odd, $G$-continuous, self-adjoint element $F$ in  $M(RL^*_c(H_X))\subset \Linears(H_X\otimes C_0[1, \infty))$ such that
\begin{enumerate}[(I)]
\item $a(F - F_0) \in C_b([1, \infty), \Compacts(H_X))$ for any $a\in A$, where $F_0\in \Linears(H_X)$ is regarded as a constant function in $\Linears(H_X\otimes C_0[1, \infty))$,
\item  for any $a\in A$ and $g\in G$,
\[
a(1-F^2),  [a, F], a(g(F)-F) \in RL^*_c(H_X).
\]
\end{enumerate}
\end{lemma}
\begin{proof} We assume $F_0$ is self-adjoint. We assume $F_0$ is $G$-equivariant using the standard trick: Let $c\in C_b(X)$ be a cut-off function on $X$. Then, 
\[
F_0'=\int_{g\in G} g(c)g(F_0)g(c)  d\mu_G(g)
\]
is $G$-equivariant and satisfies
\[
a(F_0' -F_0) \in \Compacts(H_X)
\]
for any $a\in A$. Here, for an isometry $V_c\colon H_X\to H_X\otimes L^2(G)$ defined by
\[
V_c\colon v\mapsto (g\mapsto g(c)v ),
\]
we have $F_0'=V_c^\ast\tilde FV_c$ where $\tilde F$ on  $H_X\otimes L^2(G)$ is the diagonal operator $(g\mapsto g(F_0)) \in C_b(G, \Linears(H_X))$. That $a(F_0' -F_0) \in \Compacts(H_X)$ can be checked by
\[
a(F_0' -F_0) = \int_{g\in G}  a g(c)(g(F_0)-F_0)g(c) +  ag(c)[F_0, g(c)] d\mu_{G}(g)
\]
which is norm convergent in $\Compacts(H_X)$ for compactly supported $a$. We remark that a cut-off function $c$ may not be $G$-continuous in general unless $X$ is $G$-compact. Nonetheless, the map
\[
g \mapsto g(c)a, g\mapsto g(c)S
\]
are norm-continuous for any $a\in A$ and for any $S\in \Compacts(H_X)$. Analogous remark applies to some of the arguments in the subsequent proof.
 
Let $X^+$ be the one-point compactification of $X$ and fix any metric on $X^+$ which is compatible with its topology. Let $A$ be the support of a cut-off function $c$ on $X$. The closed set $A\subset X$ is not compact (unless $X$ is $G$-compact) but it satisfies that for any compact subset $X_0$ of $X$, $A\cap gX_0=\emptyset$ for $g$ outside a compact set of $G$. We have $GA=X$. 
 
 For $n\geq1$, let $\mathcal{U}_n$ be a finite open cover of $X^+$ such that each $U\in \mathcal{U}_n$ is contained in a ball of radius $1/n$. Let $(\phi^0_{k, n})_{k \in S_n}$ be a finite partition of unity in $C(X^+)=C_0(X) + \bC1_{X^+}$ which is sub-ordinate to the cover $\mathcal{U}_n$. We set $\phi_{k ,n}=c^2\phi^0_{k ,n} \in C_b(X)$.  Thus, we have 
\[
\sum_{k \in S_n} \phi_{k, n}=c^2
\]
and the support of $\phi_{k ,n}$ is contained in a ball of radius $1/n$ in $A$. We have
\[
\sum_{k\in S_n}\int_{g\in G} g(\phi_{k, n}) d\mu_G(g) = 1.
\]
Let
\[
F_n=\sum_{k\in S_n}\int_{g\in G}g(\phi_{k, n})^{1/2}  F_0 g(\phi_{k, n})^{1/2} d\mu_G(g).
\] 
We see that $F_n$ is an odd, $G$-equivariant, self-adjoint operator on $H_X$ such that 
\[
\lVert F_n\rVert \leq \lVert F_0\rVert.
\]
The last one can be seen from $F_n=V^\ast_n\tilde F_0V_n$ where 
\[
V_n\colon H_X\to \bigoplus_{k\in S_n}L^2(G)\otimes H_X
\]
is an isometry which sends $v\in H_X$ to
\[
V_n(v)(k, g)=g(\phi_{k, n})^{1/2} v  \,\,\, (k\in S_n, g\in G)
\]
and where $\tilde F_0$ on $\bigoplus_{k\in S_n}L^2(G)\otimes H_X$ is the diagonal operator $( (k, g)\mapsto F_0 ) \in C_b(S_n\times G, \Linears(H_X))$. 

We have
\[
a(F_n-F_0) =\sum_{k\in S_n}\int_{G}ag(\phi_{k, n})^{1/2}  [F_0, g(\phi_{k, n})^{1/2}] d\mu_G(g) \in  \Compacts(H_X)
\]
for any $a\in A$. The integral is norm convergent in $\Compacts(H_X)$ for $a\in A$ with compact support. For $t\in (n, n+1)$, we set 
\[
F_t= (n+1-t)F_n + (t-n)F_{n+1}.
\]
It is clear that $t\mapsto F_t\in \Linears(H_X)$ is (uniformly) norm-continuous, $G$-equivariant, $\lVert F_t\rVert\leq \lVert F_0\rVert$ and 
\[
a(F-F_0)\in C_b([1, \infty), \Compacts(H_X))
\]
for any $a\in A$.
We show that $F_t$ satisfies $\limt\lVert[F_t, \phi]\rVert = 0$ for any $\phi \in C_0(X)$. Let $\phi \in  C_c(X)$. We have
\[
[F_n, \phi] = \sum_{k\in S_n}\int_{G}g(\phi_{k, n})^{1/2}  [F_0, \phi] g(\phi_{k, n})^{1/2} d\mu_G(g).
\]
Recall that for any $n\geq1$ and $k\in S_n$, the support of $\phi_{k, n}$ is contained in a ball of radius $1/n$ in $A$. There is a compact subset $K$ of $G$ such that $g(\phi_{k, n})\phi=0$ for $g \notin K$, and so 
\[
g(\phi_{k, n})^{1/2}  [F_0, \phi] g(\phi_{k, n})^{1/2} =0
\]
for any $n\geq1$, $k\in S_n$ and for any $g\notin K$. Now, consider a family of functions $g^{-1}(\phi) \in C_0(X)$ for $g\in K$. This family is compact in $C_0(X)\subset C(X^+)$ so $g^{-1}(\phi)$ $(g\in K)$ are equi-uniformly-continuous on $X^+$. It follows that for any $\epsilon$, there is $N>0$ so that 
\[
|g^{-1}(\phi)(x) - g^{-1}(\phi)(x_0)| < \epsilon
\]
holds for any $n\geq N$, for any $g\in K$ and for any $x, x_0\in \supp(\phi_{k, n})$ for any $k\in S_n$. In other words, 
\[
|\phi(x) - \phi(x_0)| < \epsilon
\]
holds for any $n\geq N$, for any $x, x_0\in \supp(g(\phi_{k, n}))$ for any $g\in K$ and $k\in S_n$. Using this, we see that 
\[
\lVert[F_n, \phi]\rVert = \vert \vert \sum_{k\in S_n}\int_{K}g(\phi_{k, n})^{1/2}  [F_0, \phi-\phi(gx_{k, n})] g(\phi_{k, n})^{1/2} d\mu_G(g) \vert \vert < 2\epsilon\lVert F_0\rVert
\]
for $n\geq N$ where $x_{k, n}$ is any point in $\supp(\phi_{k, n})$. This shows $\limt\lVert[F_t, \phi]\rVert = 0$. 

It follows that $(t\mapsto F_t) \in \Linears(H_X\otimes [1, \infty))$ defines an odd, $G$-equivariant, self-adjoint element $F$ in $M(RL^*_c(H_X))$ such that 
\[
a(F - F_0) \in C_b([1, \infty), \Compacts(H_X))
\]
for any $a\in A$. It follow from this and from the properties of $F_0$, 
\[
a(1-F^2),  [a, F],  a(g(F)-F) \in C_b([1, \infty), \Compacts(H_X))
\]
for any $a\in A$ and $g\in G$. Moreover, if $x$ is any of these elements $a(1-F^2),  [a, F],  a(g(F)-F)$, we can see that for any $\phi \in C_0(X)$,
\[
\lVert[x, \phi]\rVert \in C_0([1, \infty), \Compacts(H_X))
\]
because $[F, \phi]\in C_0([1, \infty), \Linears(H_X))$. We can also see that for any $\epsilon>0$, there is $\phi \in C_0(X)$
\[
\lVert(1-\phi)x \rVert < \epsilon.
\]
These imply $x\in RL^*_c(H_X)$ as desired. For example, to see the last assertion holds for $x=[a, F]$, we can assume $a$ has compact support so that $a=\chi a$ for $\chi \in C_c(X)$. Then, we have
\[
(1-\phi)[a, F] = (1-\phi)aF - (1-\phi)F\chi a =  (1-\phi)aF - (1-\phi)[F, \chi] a -  (1-\phi)\chi Fa 
\]
which can be made arbitrarily small for some $\phi\in C_0(X)$ since $[F, \chi]a \in C_0([1, \infty), \Compacts(H_X))$. This ends our proof.
\end{proof}

\begin{lemma}\label{lem_XG2} Let $X$ be a proper $G$-space and $A$ be a separable, graded $G$-$C_0(X)$-algebra which is non-degenerately represented on a graded $G$-Hilbert space $H$. We view $H$ naturally as an $X$-$G$-module through this representation and set $H_X=H$. Let $D_1\subset M(A)$ and $D_2\subset M(RL^*_c(H_X))$ be separable $G$-$C^*$-subalgebras such that $[A, D_2]\subset RL^*_c(H_X)$. Let $D_3$ be the subalgebra of $D_2$ consisting of $d\in D_2$ such that $ad \in RL^*_c(H_X)$ for $a\in A$. 

There are even, $G$-continuous elements $M, N=1-M$ in $M(RL^*_c(H_X))$ such that
\begin{enumerate}[(I)]
\item $[M, D_1] \subset RL^*_c(H_X)$,
\item $[M, D_2] \subset RL^*_c(H_X)$,
\item $M (D_1\cap A) \subset RL^*_c(H_X)$,
\item $ND_3 \subset  RL^*_c(H_X)$,
\item $g(M)-M \in RL^*_c(H_X)$.
\end{enumerate}
\end{lemma}
\begin{proof} We construct $M, N=1-M \in M(RL^*_c(H_X))$ following the standard argument (see \cite[Theorem 12.4.2]{Blackadar}, \cite{Higson87}, \cite[Section 1.4]{Kasparov88},\cite[Section 3.8]{HigsonRoe}). Let $U_n$ be an increasing sequence of relatively compact, open subsets of $G$ such that $\cup U_n=G$ and let $K_n$ be the closure of $U_n$. Let $Y$ be a compact total subset of $C_0(X)$. For $i=1, 2, 3$, let $Y_i$ be a compact total subset of $D_i$ such that $Y_1\cap A$ is total in $D_1\cap A$.  Let $Z$ be a compact total subset of $C_0([1, \infty), \Compacts(H_X))$.

First, let $(a_n)_{n\geq 0}$ ($a_0=0$) be an even, quasi-central approximate unit in $A\subset M(A)$ so that $d_n=(a_{n+1}-a_n)^{1/2}$ satisfies for $n\geq1$,
\begin{enumerate}[({1}.1)]
\item $\lVert[d_n, x]\rVert<2^{-n}$ for $x\in Y_1$,
\item  $\lVert d_nx\rVert<2^{-n}$ for $x\in Y_1\cap A$,
\item $\lVert g(d_n)-d_n\rVert<2^{-n}$ for $g\in K_n \subset G$,
\item $\lVert d_nx\rVert < 2^{-n}$ for $x\in Z$.
\end{enumerate}

Secondly, let $J$ be a separable $G$-$C^*$-subalgebra of $RL^*_c(H_X)$ which contains the ideal $RL^*_0(H_X)$, $[A, D_2]$ and $AD_3$. We arrange $J$ so that elements in $D_1$, $D_2$ and $C_0(X)$ multiply (and hence, derive) $J$. This is possible because $D_1$, $D_2$ and $C_0(X)$ multiply $RL^*_c(H_X)$.

Thirdly, let $u_n$ be an even, quasi-central approximate unit in $J$ so that $u_n$ satisfies
\begin{enumerate}[({2}.1)]
\item $\lVert[u_n, x]\rVert<2^{-n}$ for $x \in Y\cup Y_1 \cup Y_2$,
\item $\lVert(1-u_n)x\rVert<2^{-n}$ for $x\in d_nY_3 \cup [d_n, Y_2]\cap [d_n, Y_2]^\ast$,
\item $\lVert g(d_n)-d_n\rVert<2^{-n}$ for $g\in K_n \subset G$.
\end{enumerate}

Now, we consider
\[
M=\sum_{n\geq0} d_nu_nd_n, \,\,\, N=1-M=\sum_{n\geq0}d_n(1-u_n)d_n.
\]
By (1.4), these infinite sums converge in the strict topology in $M(\Compacts(H_X)\otimes C_0[1, \infty))$ to define an element in $\Linears(H_X\otimes C_0[1, \infty))$ (see \cite[Proposition 12.1.2]{Blackadar}) but they may not converge in the strict topology in $M(RL^*_c(H_X))$. We claim that $M \in M(RL^*_c(H_X))$ and so is $N=1-M$. This is because for any $\phi$ in $C_0(X)$, 
\[
[M, \phi] = \sum_{n\geq0} d_n[u_n, \phi]d_n \in C_0([1, \infty), \Compacts(H_X))
\]
which absolutely converges in $C_0([1, \infty), \Compacts(H_X))$ for $\phi \in Y$ by (2.1). Note that $[u_n, \phi] \in C_0([1, \infty), \Compacts(H_X))$ since $u_n\in J \subset RL^*_c(H_X)$. 

We check $M, N$ satisfy the conditions (I) - (V).
\begin{enumerate}[(I)]
\item $[M, D_1] \subset RL^*_c(H_X)$  follows from (1.1) and (2.1).
\item $[M, D_2]=[N, D_2] \subset RL^*_c(H_X)$ follows from (2.1) and (2.2).
\item $M (D_1\cap A) \subset RL^*_c(H_X)$ follows from (1.2).
\item $ND_3 \subset  RL^*_c(H_X)$ follows from (2.2).
\item $g(M)-M \in RL^*_c(H_X)$  follows from (1.3) and (2.3).
\end{enumerate}
That $M$ is $G$-continuous follows from (1.3) and (2.3).
\end{proof}

\begin{proof}[Proof of Proposition \ref{prop_XGlocalize}]
Let $y \in KK^G(\bC, A)$ and $z\in KK^G(A, \bC)$ so that $x=y\otimes_A z$. By stabilizing $A$ by $\Compacts(H_0)$ for a graded Hilbert space $H_0$ if necessary, we can assume that $y$ is represented by an odd, $G$-continuous, self-adjoint element $b \in M(A)$ such that $1-b^2 \in A$ and $g(b)-b\in A$ for $g\in G$. We can also assume that $z$ is represented by a cycle $(H, \pi, F_0)$ with $\pi$ non-degenerate. As in Lemma \ref{lem_XG1}, we set $H_X=H$ and let $F\in M(RL^*_c(H_X))$ be as given by the lemma. Now, let $M, N=1-M$ in $M(RL^*_c(H_X))$ as given by Lemma \ref{lem_XG2} for unital separable $G$-$C^*$-algebras $D_1 \subset M(A)$ containing $b$ and $D_2 \subset M(RL^*_c(H_X))$ containing $F$, $[F, b]$. Note we have
\[
1-b^2, g(b)-b \in D_1\cap A, \,\,\, 1-F^2, g(F)-F, [F, b] \in D_3
\]
for $g\in G$. Let 
\[
T= M^{1/4}bM^{1/4}  + N^{1/4}FN^{1/4} \in M(RL^*_c(H_X)).
\]
It is routine to check that this odd, $G$-continuous, self-adjoint element $T$ satisfies $1-T^2 \in RL^*_c(H_X)$ and $g(T)-T \in RL^*_c(H_X)$ for any $g$ in $G$. Thus, $(H_X, T)$ is an $X$-$G$-localized Kasparov cycle for $KK^G(\bC, \bC)$. By the construction, and since $a(F_1-F_0) \in \Compacts(H_X)$ for any $a \in A$, the pair $(H_X, T_1)$ for $KK^G(\bC, \bC)$ represents the Kasparov product of $y$ and $z$, that is $[H_X, T_1]=x$. 

\end{proof}


\section{Controlled algebraic aspect of the gamma element method}

We recall for any $x\in R(G)=KK^G(\bC, \bC)$ and for any separable (graded) $G$-$C^*$-algebra $B$, we have the following canonical ring homomorphism
\[
\xymatrix{
KK^G(\bC, \bC) \ar[r]^-{\sigma_B} & KK^G(B, B)  \ar[r]^-{j^G_r} & KK(B\rtimes_rG, B\rtimes_rG)   \ar[r] & \mathrm{End}(K_\ast(B\rtimes_rG))
}
\]
where the last map is defined by the Kasparov product. Let us write the image of $x$ by 
\[
x^{B\rtimes_rG}_\ast \in \mathrm{End}(K_\ast(B\rtimes_rG)).
\]
In this section, we prove the following:

\begin{theorem}\label{thm_XGfactor} Let $X$ be a proper $G$-space. Suppose that an element $x\in R(G)$ is $X$-$G$-localized, that is $x=[H_X, T_1]$ for an $X$-$G$-localized Kasparov cycle $(H_X, T)$ for $KK^G(\bC, \bC)$. Then, there is a natural group homomorphism
\[
\nu^{B, T}\colon K_\ast(B\rtimes_rG) \to \bD^{B, G}_\ast(X)
\]
for any separable $G$-$C^*$-algebra $B$ such that
\[
x^{B\rtimes_rG}_\ast =  \cf \circ \nu^{B, T} \colon K_\ast(B\rtimes_rG)  \to \bD^{B, G}_\ast(X)  \to K_\ast(B\rtimes_rG).
\]
In particular, $x^{B\rtimes_rG}_\ast$ factors through the Baum--Connes assembly map $\mu_r^{B ,G}$. Here, $\nu^{B, T}$ is natural with respect to a $G$-equivariant $\ast$-homomorphism $B_1\to B_2$ in a sense that the following diagram commutes 
\[
\xymatrix{
K_\ast(B_1\rtimes_rG) \ar[r]^-{\nu^{B_1, T}} \ar[d]^-{\pi\rtimes_r1_\ast}   & \bD^{B_1, G}_\ast(X) \ar[d]^-{\pi_\ast} \\
K_\ast(B_2\rtimes_rG) \ar[r]^-{\nu^{B_2, T}}  & \bD^{B_2, G}_\ast(X). 
}
\]
\end{theorem}

\begin{theorem}\label{thm_XGgamma} Let $X$ be a proper $G$-space. If there is an $X$-$G$-localized element $x\in R(G)$ such that $x=1_K$ in $R(K)$ for any compact subgroup $K$, the Baum--Connes assembly map $\mu_r^{B, G}$ is split-injective for any $B$ and in this case the image of $\mu_r^{B, G}$ coincides with the image of  $x^{B\rtimes_rG}_\ast$. In particular, if $x=1_G$, BCC holds for $G$.
\end{theorem}

The proof of these theorems will be quite simple and formal. We start with the construction of $\nu^{B, T}$. We should think that the map $\nu^{B, T}$ is obtained by sending an $X$-$G$-localized cycle $(H_X, T)$ through 
\[
KK^G(\bC, RL^*_c(H_X)) \to  KK^G(B, RL^*_c(H_X\otimes B)) \to KK(B\rtimes_rG, RL^*_c(H_X\otimes B)\rtimes_rG) 
\]
\[
\to  \mathrm{End}(K_\ast(B\rtimes_rG), K_\ast(RL^*_c(H_X\otimes B)\rtimes_rG))
\]
where the first map is $\sigma_B$ followed by the inclusion $RL^*_c(H_X)\otimes B \to RL^*_c(H_X\otimes B)$, the second map is Kasparov's descent $j^G_r$ and the last map is given by the Kasparov product. Although we can take it as the definition of $\nu^{B, T}$, we reduce this to the separable $G$-$C^*$-algebras setting.

Recall, for any $C^*$-subalgebra $A$ of $RL^*_c(H_X)$ containing the ideal $RL^*_0(H_X)$, the multiplier algebra $M(A)$ is naturally a subalgebra of $M(RL^*_0(H_X))$. Thus, it makes sense to ask if an operator $T$ in $M(RL^*_c(H_X)) \subset M(RL^*_0(H_X))$ belongs to $M(A)$.

\begin{lemma} Let $(H_X, T)$ be an $X$-$G$-localized Kasparov cycle for $KK^G(\bC, \bC)$. Then, there is a separable, graded $G$-$C^*$-subalgebra $A_X$ of $RL^*_c(H_X)$ containing the ideal $RL^*_0(H_X)$ such that 
\[
T \in M(A_X) \subset M(RL^*_0(H_X))
\]
and for any $g\in G$,
\[
1-T^2, g(T) - T \in A_X
\]
Note, $T$ is automatically a $G$-continuous element in $M(A_X)$.
\end{lemma}
\begin{proof} Just take any separable, graded $G$-$C^*$-subalgebra $A_X$ of $RL^*_c(H_X)$ containing the ideal $RL^*_0(H_X)$, such that it contains $1-T^2$ and $T-g(T)$ for all $g\in G$ and such that $T$ multiplies $A_X$.
\end{proof}

Now, for any $X$-$G$-localized cycle $(H_X, T)$ for $KK^G(\bC, \bC)$, let $A_X$ be a separable, graded $G$-$C^*$-subalgebra $A_X$ of $RL^*_c(H_X)$ as in the previous lemma. Then, the pair $(A_X, T)$ is a cycle for $KK^G(\bC, A_X)$ and hence by applying the usual homomorphism
\[
\xymatrix{
KK^G(\bC, A_X) \ar[r]^-{j^G_r\circ \sigma_B}  &  \mathrm{End}(K_\ast(B\rtimes_rG),  K_\ast((A_X\otimes B)\rtimes_rG)),
}
\]
we obtain a group homomorphism
\[
 (j^G_r\circ\sigma_B)[A_X, T]_\ast \colon K_\ast(B\rtimes_rG) \to K_\ast((A_X\otimes B)\rtimes_rG)
\]
which is natural with respect to a $G$-equivariant $\ast$-homomorphism $B_1\to B_2$. The desired group homomorphism
\[
\nu^{B, T} \colon K_\ast(B\rtimes_rG) \to K_\ast(RL^*_c(H_X\otimes B)\rtimes_rG)
\]
is obtained as the composition of $ (j^G_r\circ\sigma_B)[A_X, T]_\ast $ with the $\ast$-homomorphism
\[
(A_X\otimes B)\rtimes_rG \to (RL^*_c(H_X)\otimes B)\rtimes_rG  \to RL^*_c(H_X\otimes B)\rtimes_rG.
\]
It is clear, from the naturality of Kasparov product and the descent map, that $\nu^{B, T}$ is natural with respect to a $G$-equivariant $\ast$-homomorphism $B_1\to B_2$.

Note that for a graded $X$-$G$-module $H_X$, the K-theory group $K_\ast(RL^*_c(H_X\otimes B)\rtimes_rG)$ of the graded $C^*$-algebra $RL^*_c(H_X\otimes B)\rtimes_rG$ is functorial with respect to a graded, $G$-equivariant continuous cover $(V_t\colon H_X\to H_Y)_{t\in [1,\infty)}$ for a $G$-equivariant map $f\colon X\to Y$. Such a cover exists for any $H_X$ whenever $H_Y$ is universal, that is if both the even space $H_Y^{(0)}$ and the odd space $H_Y^{(1)}$ is universal. Moreover, if $H^0_X$ is a (ungraded) $X$-$G$-module and $H_X$ is a graded $X$-$G$-module given by $H_X^{(0)}= H_X^{(1)} =H^0_X$, the inclusion 
\[
RL^*_c(H^{0}_X\otimes B) \to RL^*_c(H_X\otimes B) 
\]
induces an isomorphism
\[
K_\ast( RL^*_c(H^{0}_X\otimes B)\rtimes_rG ) \cong K_\ast( RL^*_c(H_X\otimes B)\rtimes_rG ).
\]
Therefore, we may redefine the functor $\bD^{B, G}_\ast(X)$ by using a graded universal $X$-$G$-module $H_X$ instead, and this functor is naturally equivalent to the original one. Moreover, in this case, for any graded $X$-$G$-module $H_X$, we have a natural group homomorphism
\[
K_\ast(RL^*_c(H_X\otimes B)\rtimes_rG) \to \bD^{B, G}_\ast(X).
\]
The forget control map 
\[
\cf\colon \bD^{B, G}_\ast(X) \to K_\ast(B\rtimes_rG)
\]
is defined as before. With these in mind, the natural group homomorphism $\nu^{B, T}\colon  K_\ast(B\rtimes_rG) \to K_\ast(RL^*_c(H_X\otimes B)\rtimes_rG)
$ extends to  
\[
\nu^{B, T}\colon K_\ast(B\rtimes_rG) \to \bD^{B, G}_\ast(X).
 \]

\begin{lemma}\label{lem_XGfactor} Let $(H_X, T)$ be an $X$-$G$-localized cycle for $KK^G(\bC, \bC)$ and $x=[H_X, T_1]$ in $R(G)$. Then, we have
\[
\cf \circ \nu^{B, T} =  x^{B\rtimes_rG}_\ast
\]
on $K_\ast(B\rtimes_rG)$.
\end{lemma}
\begin{proof} The evaluation $\mathrm{ev}_1\colon RL^*_c(H_X)\to \Compacts(H_X)$ restricts to $A_X$ and it extends to the map
\[
\mathrm{ev}_1\colon M(A_X) \to \Linears(H_X)
\]
which sends $T$ to $T_1$. The evaluation map $\mathrm{ev}_1\colon A_X\to \Compacts(H_X)$ defines an element $[H_X, \mathrm{ev}_1, 0] \in KK^G(A_X, \bC)$ whose Kasparov product with $[A_X, T]\in KK^G(\bC, A_X)$ is just $x=[H_X, T_1]$ in $KK^G(\bC, \bC)$.
 It is easy to see that the composition $\cf \circ \nu^{B, T}$ coincides with the composition of 
 \[
 (j^G_r\circ\sigma_B)[A_X, T]_\ast\colon K_\ast(B\rtimes_rG) \to K_\ast((A_X\otimes B)\rtimes_rG)
 \]
 and 
 \[
(j^G_r\circ\sigma_B)[H_X, \mathrm{ev}_1,  0]_\ast \colon K_\ast((A_X\otimes B)\rtimes_rG) \to  K_\ast(B\rtimes_rG) 
 \]
defined by the element $[H_X, \mathrm{ev}_1,  0] \in KK^G(A_X, \bC)$ through
\[
\xymatrix{
KK^G(A_X, \bC) \ar[r]^-{j^G_r\circ \sigma_B}  &  \mathrm{End}(K_\ast(A_X\otimes B\rtimes_rG),  K_\ast(B\rtimes_rG)  ).
}
\]
From this, we can see that $\cf \circ \nu^{B, T}= x^{B\rtimes_rG}_\ast$.
\end{proof}

 \begin{proof}[Proof of Theorem ~\ref{thm_XGfactor}] The first part of the theorem follows from Lemma \ref{lem_XGfactor}. The second part follows from the first part and from Theorem \ref{thm_forget_factor}. Note that we do not need the isomorphism Theorem \ref{thm_main_isom}.
 \end{proof}

  \begin{proof}[Proof of Theorem ~\ref{thm_XGgamma}] 
 Thanks to Theorem \ref{thm_XGfactor}, there are group homomorphisms
  \[
  \nu^{B, T}\colon K_\ast(B\rtimes_rG) \to \varinjlim_{Y\subset \EG, \mathrm{Gcpt}}K_\ast^G(C_0(Y), B)
  \]
which are natural with respect to $B$ and we have 
\[
\mu_r^{B, G}\circ   \nu^{B, T} = x^{B\rtimes_rG}_\ast \colon K_\ast(B\rtimes_rG) \to K_\ast(B\rtimes_rG).
\]
Under the assumption $x=1_K$ in $R(K)$ for any compact subgroup $K$ of $G$, we can show that the other composition $\nu^{B, T} \circ \mu_r^{B, G}$ is the identity for all $B$ just as in the proof of \cite[Proposition 5.3]{Nishikawa19} using the naturality of   $\mu_r^{B, G}$ and of $\nu^{B, T}$ with respect to a $G$-equivariant $\ast$-homomorphism $\pi \colon B_1\to B_2$. It follows that $\mu_r^{B, G}$ is split-injective for all $B$ and the image of $\mu_r^{B, G}$ coincides with the idempotent $\mu_r^{B, G}\circ   \nu^{B, G} = x^{B\rtimes_rG}_\ast $ on $K_\ast(B\rtimes_rG)$.
   \end{proof}

\noindent \textbf{Description of the splitting of the forget-control map} Let $(H_X, T)$ be an $X$-$G$-localized Kasparov cycle for $KK^G(\bC, \bC)$. We give an explicit description of the map 
\[
\nu^{B, T}\colon K_0(B\rtimes_rG) \to \bD^{B, G}_0(X)
\]
which is a splitting of the forget-control map $\cf$ if $[H_X, T_1]=1_K$ in $R(K)$ for any compact subgroup $K$ of $G$.

Let $\hill_0=l^2(\N)^{(0)}\oplus l^2(\N)^{(1)}$ be the standard graded Hilbert space. For any (not necessarily separable) graded $C^*$-algebra $A$, the K-theory group $K_0(A)=KK(\bC, A)$ can be identified (see \cite[Section 3]{Skandalis85}) as the set of homotopy equivalence classes of odd, self-adjoint operators $F \in M(\Compacts(\hill_0)\hat\otimes A)\cong \Linears(\hill_0\hat\otimes A)$ such that $1-F^2 \in \Compacts(\hill_0)\hat\otimes A \cong \Compacts(\hill_0\hat\otimes A)$.

We describe the image $\nu^{B, T}([F]) \in  \bD^{B, G}_0(X)$ of $[F] \in K_0(B\rtimes_rG)$ for $F\in M(\Compacts(\hill_0)\hat\otimes B\rtimes_rG)$.

We have canonical representations 
\[
M(RL^*_c(H_X)) \to M(\Compacts(\hill_0)\hat\otimes (RL^*_c(H_X)\otimes B)\rtimes_rG ) 
\]
\[
M(\Compacts(\hill_0)\hat\otimes B\rtimes_rG) \to  M(\Compacts(\hill_0)\hat\otimes (RL^*_c(H_X)\otimes B)\rtimes_rG )
\]
We still denote by $F$, resp. by $T$, the image of $F \in M(\Compacts(\hill_0)\hat\otimes B\rtimes_rG)$, resp. $T \in M(RL^*_c(H_X))$ in $M(\Compacts(\hill_0)\hat\otimes (RL^*_c(H_X)\otimes B)\rtimes_rG )$ by these representations.

Let
\[
T \sharp F = F + (1-F^2)^{1/4} T (1-F^2)^{1/4}  \in M(\Compacts(\hill_0)\hat\otimes (RL^*_c(H_X)\otimes B)\rtimes_rG ).
\] 
It satisfies 
\[
1- (T \sharp F )^2 \in \Compacts(\hill_0)\hat\otimes (RL^*_c(H_X)\otimes B)\rtimes_rG.
\]
Hence, the odd, self-adjoint element $T \sharp F$ defines the element $[T \sharp F]$ in
\[
K_0((RL^*_c(H_X)\otimes B)\rtimes_rG) \to K_0(RL^*_c(H_X\otimes B)\rtimes_rG) \to \bD_0^{B, G}(X). 
\]
We have
\[
\nu^{B, T}[F]  = [T \sharp F] \in \bD^{B, G}_0(X).
\]
In this sense, $\nu^{B, T}$ is just given by the Kasparov product by $T$. Of course, the forget control map $\cf\colon  \bD^{B, G}_0(X) \to K_0(B\rtimes_rG)$ sends this element $[T \sharp F]$ to 
\[
[T_1 \sharp F] \in K_0(B\rtimes_rG)
\]
represented by an odd, self-adjoint operator 
\[
T_1 \sharp F = F + (1-F^2)^{1/4} T_1 (1-F^2)^{1/4}
\]
in $M(\Compacts(\hill_0)\hat\otimes (\Compacts(H_X)\otimes B)\rtimes_rG ) \cong \Linears(\hill_0\hat\otimes H_X \otimes B\rtimes_rG)$, which represents the Kasparov product $[F]\otimes_{B\rtimes_rG}j^G_r(\sigma_B [H_X, T_1] )$ in $KK(\bC, B\rtimes_rG)$.


\section{Examples}

In this section, we give some examples of $X$-$G$-localized Kasparov cycles for $KK^G(\bC, \bC)$.

Recall from \cite{BaajJulg} that an unbounded cycle for $KK^G(\bC, \bC)$ is a pair $(H, D)$ of a separable, graded $G$-Hilbert space $H$ and an odd, self-adjoint unbounded operator $D$ on $H$ such that 
\begin{itemize}
\item $(D\pm i)^{-1}\in \Compacts(H)$ ($D$ has compact resolvent) and 
\item the $G$-action on $H$ preserves the domain of $D$, and $g(D)-D$ extends to a strongly continuous, locally bounded, $\Linears(H)$-valued function of $g \in G$.
\end{itemize}
For any unbounded cycle $(H, D)$, the pair $(H, T)$ with $T=D(1+D^2)^{-1}$ is a (bounded) Kasparov cycle for $KK^G(\bC, \bC)$ \cite{BaajJulg}.
 
Now suppose we have a family $\{D_t\}_{t\geq1}$ of odd, self-adjoint unbounded operators on a graded $G$-Hilbert space $H$ which defines an odd, regular, self-adjoint unbounded operator $D$ (\cite[Chapter 10]{Lance}) on a Hilbert $C_0[1, \infty)$-module $H\otimes C_0[1, \infty)$. Then, the bounded transform $T=D(1+D^2)^{-1}$ of $D$ is an odd, self-adjoint operator in $\Linears(H\otimes C_0[1, \infty))$.

\begin{proposition}\label{prop_unboundedXG} Let $H_X$ be a graded $X$-$G$-module and $D$ be an odd, regular self-adjoint unbounded operator on a Hilbert $C_0[1, \infty)$-module $H_X\otimes C_0[1, \infty)$ such that 
\begin{enumerate}
\item there is a dense subalgebra $B$ of $C_0(X)$ which preserves the domain of $D$ such that for any $\phi \in B$, $[D, \phi]$ extends to a bounded operator $S \in \Linears(H\otimes C_0[1, \infty))$ with $\lVert S_t\rVert\to 0$ as $t \to \infty$,
\item the $G$-action on $H$ preserves the domain of $D$, and $g(D)-D$ extends to a norm-continuous $\Linears(H\otimes C_0[1, \infty))$-valued function of $g \in G$, and
\item $(D\pm i)^{-1}\in RL^*_c(H_X)$.
\end{enumerate}
Then, the pair $(H_X, T)$ with $T=D(1+D^2)^{-1}$ is an $X$-$G$-localized Kasparov cycle for $KK^G(\bC, \bC)$. If $D$ satisfies the conditions (1), (2), then the condition (3) is equivalent to the following
\begin{enumerate}
\item [(3.1)] $(D\pm i) \in C_b([1,\infty), \Compacts(H_X))$ and
\item  [(3.2)] for any $\epsilon>0$, there is $\phi \in C_0(X)$ such that $\lVert(1-\phi)(D\pm i)\rVert<\epsilon$.
\end{enumerate}
\end{proposition}
\begin{proof} The condition (3) implies $1-T^2 = (1+D^2)^{-1} \in RL^*_c(H_X)$. We have the following Baaj--Julg formula \cite{BaajJulg}
\[
T = \frac{\pi}{2} \int_{0}^\infty D(1+D^2+\lambda^2)^{-1}d\lambda  = \frac{\pi}{4}  \int_{0}^\infty (D+\sqrt{1+\lambda^2} i)^{-1} + (D-\sqrt{1+\lambda^2} i)^{-1} d\lambda.
\]
Here, the integral converges in the strong topology on $\Linears(H\otimes C_0[1, \infty))$. For $\phi \in B$, we have
\begin{equation}\label{eq_Tphi}
[T, \phi] = \frac{\pi}{4}  \int_{0}^\infty [(D+\sqrt{1+\lambda^2} i)^{-1}, \phi] + [(D-\sqrt{1+\lambda^2} i)^{-1}, \phi] d\lambda
\end{equation}
and
\[
[(D\pm \sqrt{1+ \lambda^2} i)^{-1}, \phi]   = (D\pm \sqrt{1+ \lambda^2} i)^{-1}[\phi, D](D\pm\sqrt{1+ \lambda^2} i)^{-1}.
\]
From the condition (1), we see $[T_t, \phi] \to 0$ as $t\to \infty$. For $g\in G$, we have
\begin{equation}\label{eq_gT1}
g(T) - T = \frac{\pi}{4}  \int_{0}^\infty A_g^{+} +A_g^{-} d\lambda
\end{equation}
where 
\begin{equation}\label{eq_gT2}
A_g^{\pm} = (g(D)\pm \sqrt{1+\lambda^2} i)^{-1} (D-g(D)) (D\pm \sqrt{1+\lambda^2} i)^{-1}.
\end{equation}
Using the conditions (2), (3) and $\lim_{t\to \infty}\lVert[T_t, \phi]\rVert= 0$ for $\phi \in C_0(X)$, we see that $T$ is a $G$-continuous element in $M(RL^*_c(H_X))$ and that
\[
g(T) - T \in C_b([1 ,\infty), \Compacts(H_X)).
\] 
To see $g(T) - T \in RL^*_c(H_X)$, since we have $\lim_{t\to \infty}[T_t, \phi] = 0$, it is enough to show that for any $\epsilon>0$, there is $\phi \in C_0(X)$ such that $\lVert(g(T)-T)(1-\phi)\rVert < \epsilon$. This follows from \eqref{eq_gT1}, \eqref{eq_gT2}, using $ (D\pm \sqrt{1+\lambda^2} i)^{-1} \in RL^*_c(H_X)$ for any $\lambda\geq0$ (which follows from the condition (1)). We may also use $g(D)-D \in M(RL^*_c(H_X))$ (which follows from the conditions (1), (3)) to see $g(T) - T \in RL^*_c(H_X)$. We explain that for $D$ satisfying (1) and (2), the condition (3) is equivalent to (3.1) and (3.2).  That (3) implies (3.1) and (3.2) is obvious. On the other hand, the conditions (1) and (3.1) imply $[(D\pm i)^{-1}, \phi]\in C_0([1,\infty), \Compacts(H_X))$ and the condition (2) implies that $(D\pm i)^{-1}$ is $G$-continuous. Together with (3.2), these imply (3).

\end{proof}

\begin{example}(c.f.\ \cite[Section 5]{Kasparov88} \cite[Section 3]{Lafforgue2002} \cite[Example 9.6]{Valette02}) Let $X$ be a complete, simply connected Riemannian manifold of non-positive sectional curvature, with curvature bounded below, on which a locally compact group $G$ acts properly and isometrically. Let $x_0 \in X$, $d_X$ be the Riemannian distance function on $X$ and $\rho\in C^\infty(X)$ be defined by 
\[
\rho(x)=\sqrt{d_X(x_0, x)^2+ 1}.
\]
Let $d\rho= \xi$ be the exterior derivative of $\rho$. The norm $\lVert\xi\rVert_x$ of the co-vector $\xi$ at $x$ satisfies $\lVert\xi\rVert_x=\frac{d_X(x_0, x)}{\sqrt{d_X(x_0, x)^2+ 1}}\to 1$ as $x$ goes to infinity.
Let $H_X=\Omega^*_{L^2}(X)$ be the $L^2$-space of the de-Rham complex on $X$ and for $t>0$, let 
\[
D_t = t^{-1}(d+d^\ast) + \mathrm{ext}(\rho d\rho) +  \mathrm{int}(\rho d\rho)
\]
where $d$ is the de-Rham differential operator and $\mathrm{ext}(v)$ and resp. $\mathrm{int}(v)$ are the exterior, resp. interior, multiplication by the co-vector $v$. The odd, unbounded operators $D_t$ defined on the space $\E_c^\infty$ of compactly supported differential forms are essentially self-adjoint and have compact resolvent. We have
\[
D_t^2= t^{-2}\Delta + \rho^2\lVert\xi\rVert^2 + t^{-1}A
\]
where $\Delta=dd^\ast + d^\ast d$ and $A= [(d+d^\ast), \mathrm{ext}(\rho d\rho) +  \mathrm{int}(\rho d\rho)]$ is an order-zero, smooth bundle-endomorphism. There are constants $C_0, C_1>0$ so that we have for $x\in X$,
\[
 \lVert A_x\rVert  \leq C_0 + C_1\rho(x).
\]
Using this, we see that there is $N>0$ so that we have for $t\geq1$, 
\begin{equation}\label{eq_N}
D_t^2 + N^2 \geq  t^{-2}\Delta  +  \rho^2/2.
\end{equation}
The family $\{D_t\}_{t\geq1}$ defines an odd, essentially self-adjoint, regular operator $D$ on $H_X\otimes C_0[1, \infty)$ with domain $C_c([1, \infty), \E_c^\infty)$. We have
\[
(D\pm i)^{-1} \in C_b([1, \infty), \Compacts(H_X)).
\]
Let $T=D(1+D^2)^{-1}$ in $\Linears(H_X\otimes C_0[1, \infty))$. 

\begin{proposition} The pair $(H_X, T)$ is an $X$-$G$-localized Kasparov cycle for $KK^G(\bC, \bC)$.
\end{proposition}
\begin{proof} We check that $D$ satisfies the conditions (1) -(3) of Proposition \ref{prop_unboundedXG}. (1): Let $B=C_c^\infty(X) \subset C_0(X)$. We have for any $\phi \in B$,
\[
[D_t, \phi] = t^{-1} ( \mathrm{ext}(d\phi) -  \mathrm{int}(d\phi))
\] 
from which the condition (1) holds for $D$. (2): we have
\[
g(D_t) - D_t =  \mathrm{ext}(\rho' d\rho' - \rho d\rho) +  \mathrm{int}(\rho' d\rho' - \rho d\rho)
\]
where $\rho'(x) =\sqrt{d(g(x_0), x)^2+ 1}$. We have $\lVert d\rho'-d\rho\rVert_x \leq 6d(g(x_0) ,x_0)(\rho(x)+\rho'(x))^{-1}$  \cite[Lemma 5.3]{Kasparov88}. Using this, we see that the condition (2) holds for $D$. (3): Since $(D\pm i)^{-1} \in C_b([1, \infty), \Compacts(H_X))$, it is enough to show that for any $\epsilon>0$, there is $\phi \in C_0(X)$, such that $\lVert(1-\phi)(D_t\pm i)^{-1}\rVert < \epsilon$ for all $t\geq1$. For this, we may instead prove the same claim not for $(D_t\pm i)^{-1}$ but for  $(D_t\pm Ni)^{-1}$ for the constant $N>0$ as in \eqref{eq_N}. Take any $v=(D_t \pm Ni)v_0$ for $v_0 \in \E_c^\infty$. Then, we have
\[
(1-\phi) (D_t\pm i)^{-1}v = (1-\phi)v_0
\]
and 
\[
\lVert v\rVert^2 = \lVert(D_t \pm Ni)v_0\rVert^2 = \s{(D_t^2 + N^2)v_0, v_0} \geq   \frac{1}{2} \s{\rho^2v_0, v_0}.
\]
Using these, it is not hard to see that for any $\epsilon>0$, there is $\phi\in C_0(X)$ with $\lVert(1-\phi)(D_t\pm Ni)^{-1}\rVert < \epsilon$ for all $t\geq1$. Hence, the condition (3) holds for $D$. 
\end{proof}

One can show that $[H_X, T_1]=1_K$ in $R(K)$ for any compact subgroup $K$ of $G$. By Theorem \ref{thm_XGgamma}, we see that the Baum--Connes assembly map $\mu^{B, G}_r$ is split-injective for any $G$ which acts properly and isometrically on a complete, simply connected Riemannian manifold of non-positive sectional curvature, with curvature bounded below. Of course, $[H_X, T_1]$ in $KK^G(\bC, \bC)$ is nothing but the gamma element for $G$ constructed by Kasparov \cite{Kasparov88}.
\end{example}

\begin{example} (c.f.\ \cite{JulgValette84}) Let $X$ be a locally finite tree on which a locally compact group $G$ acts properly (and continuously) by automorphisms. Let 
\[
H^{\mathrm{JV}}_X= l^2(X^0) \oplus l^2(X^1)
\]
be the direct sum of $l^2$-spaces of the set $X^0$ of vertices of $X$ and of the set $X^1$ of the edges. More precisely, we define $l^2(X^1)$ as the quotient of $l^2$-space of oriented edges by the closed subspaces spanned by $[x, y] + [y, x]$ for oriented edges $[x, y]$ of $X$. The Hilbert space $H^{\mathrm{JV}}_X$ is naturally a graded $X$-$G$-module. Fix $x_0\in X$ and $T_{\mathrm{JV}}$ be the associated Julg--Valette operator on $H^{\mathrm{JV}}_X$. The operator $T_{\mathrm{JV}}$ can be defined as follows. For any $x\in X^0\backslash \{x_0\}$, let $e_x=[x', x]$ be the unique (oriented) edge which is closest to $x_0$ among the all edges containing $x$. Then, the Hilbert space $H^{\mathrm{JV}}_X$ has the following direct sum decomposition
\[
H^{\mathrm{JV}}_X = \bC\delta_{x_0} \oplus  \bigoplus_{x\in X^0 \backslash \{x_0\}} \bC\delta_{x} \oplus \bC\delta_{e_x}.
\]
Under this decomposition of $H^{\mathrm{JV}}_X$, the Julg--Valette operator $T_{\mathrm{JV}}$ is expressed as the block-diagonal operator 
\[
 0 + \sum_{x\in X^0\backslash \{x_0\}} \begin{bmatrix} 0 & 1 \\ 1 & 0 \end{bmatrix}.
\]
The Julg--Valette operator $T_{\mathrm{JV}}$ is an, odd, self-adjoint, bounded, $G$-continuous operator on $H^{\mathrm{JV}}_X$ such that for any $g\in G$,
\[
1-T_{\mathrm{JV}}^2, \,\,\, g(T_{\mathrm{JV}})- T_{\mathrm{JV}} \in \Compacts(H^{\mathrm{JV}}_X) 
\]
(in fact they are of finite rank). Thus, $(H^{\mathrm{JV}}_X, T_{\mathrm{JV}})$ is a (bounded) Kasparov cycle for $KK^G(\bC, \bC)$ and we have $[H^{\mathrm{JV}}_X, T_{\mathrm{JV}}]=1_G$ in $KK^G(\bC, \bC)$ \cite[Proposition 1.6]{JulgValette84} (for this, the $G$-action does not have to be proper).

Now, we recall some idea from \cite{BGHN}. Instead of the $X$-$G$-module $H^{\mathrm{JV}}_X$ which is rather ``discrete'', we consider the Hilbert space $H^{\mathrm{dR}}_X$ of the de-Rham complex on $X$ which is more ``continuous''. The Hilbert space $H^{\mathrm{dR}}_X$ is defined as 
\[
H^{\mathrm{dR}}_X = \Omega^\ast_{L^2}(X) = \bigoplus_{i=0, 1} \bigoplus_{\sigma \in X^i}  \Omega^\ast_{L^2}(\sigma)
\]
where for a vertex $\sigma\in X^0$, $\Omega^\ast_{L^2}(\sigma)$ is just the one-dimensional space $\bC\delta_{\sigma}$ and for a (non-oriented) edge $\sigma\in X^1$, $\Omega^\ast_{L^2}(\sigma)$ is the space of $L^2$-sections of the de-Rham complex on $\sigma$ equipped with the standard metric so that $\sigma \cong [0 ,1]$ isometrically. The space $H^{\mathrm{dR}}_X$ is naturally a graded $X$-$G$-module and we have an embedding
\[
H^{\mathrm{JV}}_X  \to H^{\mathrm{dR}}_X 
\]
of $X$-$G$-modules which sends $\delta_x$ for $x\in X^0$ to $\delta_x \in \Omega^0_{L^2}(x)$ and $\delta_e$ for an oriented edge $e=[x, y]$ to the constant 1-form in $\Omega^1_{L^2}(e)$ which corresponds to $ds$ under the isomorphism 
\[
\Omega^1_{L^2}(e) \cong \Omega^1_{L^2}[0 ,1]  =  L^2[0, 1]ds
\]
under which the vertex $x$ is identified as $0$ and $y$ is identified as $1$ in $[0, 1]$. Fix $x_0\in X$, we have the following direct sum decomposition
\begin{equation}\label{eq_decomp_dR}
H^{\mathrm{dR}}_X  = \bC\delta_{x_0} \oplus  \bigoplus_{x\in X^0\backslash \{x_0\}} \bC\delta_{x} \oplus \Omega^\ast_{L^2}(e_x)
\end{equation}
where $e_x=[x', x]$ as before. Furthermore, we have an isomorphism
\begin{equation}\label{eq_identify}
\Omega^\ast_{L^2}(e_x) \cong \Omega^\ast_{L^2}[0 ,1] = L^2[0, 1] \oplus L^2[0, 1]ds
\end{equation}
under which $x'$ is identified as $0$ and $x$ is identified as $1$ in $[0, 1]$. For any $t>0$, let 
\[
d_t = t^{-1}d + \mathrm{ext}(ds)\colon \Omega^0_{0}[0 ,1] \to  \Omega^1[0 ,1]
\]
be the Witten-type perturbation $t^{-1}(e^{-ts}de^{ts})$ of the de-Rham differential operator $d$ on $\Omega^\ast_{L^2}[0 ,1]$ and consider it as an unbounded operator on $\Omega^\ast_{L^2}(e_x)$. Here,  $\Omega^\ast[0 ,1]$ is the space of (smooth) forms and 
\[
\Omega_0^0[0 ,1] = \{ s \in C^\infty[0 ,1]\mid s(0)=s(1)=0 \}.
\]
Let
\[
D_t = d_t + d_t^\ast \colon \Omega^\ast_{L^2}[0 ,1] \to \Omega^\ast_{L^2}[0 ,1]
\]
with domain $\Omega_0^0[0 ,1] \oplus \Omega^1[0 ,1]$. Then, $D_t$ is an odd, essentially self-adjoint, unbounded operator on $\Omega^\ast_{L^2}[0 ,1]$ with compact resolvent. For each integer $k>0$, the span of the forms
\[
\sin(\pi k s), \,\,\, d_t\sin(\pi k s) = t^{-1}\pi k \cos(\pi k s)ds + \sin(\pi k s)ds
\]
is invariant under $D_t$ and we have
\[
D_t^2 = t^{-2}\pi^2k^2 +  1
\]
on the span. These spans are mutually orthogonal in $\Omega^\ast_{L^2}[0 ,1]$ and their mutual orthogonal complement $\Omega^\ast_{L^2}[0 ,1]$ is spanned by
\[
\tilde e_t = e^{ts}ds \in \Omega^1[0 ,1]
\]
which lies in the kernel of $D_t$. Now, let
\[
S_t = \chi(D_t)
\]
where $\chi\in C_b(\R)$ is any odd, continuous, increasing function on $\R$ such that $\chi(x)=1$ for $x\geq1$. Then, we have
\[
1-S_t^2 = P_t
\]
where $P_t$ is the rank one projection onto the span of $e_t$ in $\Omega^\ast_{L^2}[0 ,1]$. Under the identification \eqref{eq_identify}, $S_t$ defines an, odd, self-adjoint operator $S_{t, x}$ on each $\Omega^\ast_{L^2}(e_x)$ for $x\in X^0 \backslash \{x_0\}$. Now, let 
\[
e_t = \frac{\tilde e_t}{\lVert\tilde e_t \rVert } = \sqrt{\frac{2t}{e^{2t}-1}} e^{ts}ds =  \sqrt{\frac{2t}{1-e^{-2t}}} e^{t(s-1)}ds \in\Omega^1[0 ,1]
\]
and $e_{t, x}$ be the corresponding unit vector in $\Omega^\ast_{L^2}(e_x)$ under the identification \eqref{eq_identify}. Then, $1-S^2_{t, x}$ on $\Omega^\ast_{L^2}(e_x)$ is the rank one projection onto the span of $e_{t, x}$. We set 
\[
R_{t, x} = \theta_{e_{t, x}, \delta_x} +  \theta_{\delta_x, e_{t, x}}
\]
be the partial isometry on $\bC\delta_{x} \oplus \Omega^\ast_{L^2}(e_x)$ which sends $e_{t, x}$ to $\delta_x$ and $\delta_x$ to $e_{t, x}$ and is zero on their complement. We define the block-diagonal operator 
\[
T_t = 0 + \sum_{x\in X^0\backslash \{x_0\}} S_{t, x} + R_{t, x} 
\]
on $H^{\mathrm{dR}}_X$ with respect to the decomposition \eqref{eq_decomp_dR}. Then, the family $\{T_t\}_{t\geq1}$ defines an odd, self-adjoint operator $T$ in $\Linears(H^{\mathrm{dR}}_X\otimes C_0[1 ,\infty))$.

\begin{proposition} The pair $(H^{\mathrm{dR}}_X, T)$ is an $X$-$G$-localized Kasparov cycle for $KK^G(\bC, \bC)$.
\end{proposition}
\begin{proof}
 It is clear that $1-T_t^2$ is the rank-one projection onto the span of $\delta_{x_0}$. Moreover, as in the case of the Julg--Valette operator, for any $g\in G$, the difference $g(T_t)-T_t$ is contributed from only finitely many blocks corresponding to the vertices $x$ on the geodesics from $x_0$ to $g(x_0)$. After noticing this, it is not hard to see that $T$ is $G$-continuous and that for any $g\in G$,
 \[
 g(T) - T \in C_b([1, \infty), \Compacts(H^{\mathrm{dR}}_X))
 \]
and $g(T)-T$ has uniform compact support with respect to $X$. Finally, we have $\lim_{t\to \infty}\lVert[T_t, \phi]\rVert=0$ for any $\phi$ in $C_0(X)$. To see this, we take $\phi$ to be a compactly supported function on $X$ which is smooth on each edge. Then, we have
\[
[D_t, \phi] = t^{-1}(\mathrm{ext}(d\phi) - \mathrm{int}(d\phi) )
\]
on each block $\Omega^\ast_{L^2}(e_x)\cong \Omega^\ast_{L^2}[0 ,1]$ and this implies $\lim_{t\to \infty}\lVert[S_{t, x}, \phi]\rVert=0$ for each $x$. To see $\lim_{t\to \infty}\lVert[R_{t, x}, \phi]\rVert=0$, we compute
\[
 [\theta_{e_{t, x}, \delta_x} , \phi] = (\phi - \phi(x))\theta_{e_{t, x}, \delta_x}.
\]
We have $\lVert(\phi - \phi(x))e_{t, x}\rVert \to 0$ as $t\to \infty$, uniformly in $x\in X\backslash \{x_0\}$. This follows from the following elementary estimate:
\[
\sqrt{\frac{2t}{1-e^{-2t}}} \lVert(1-s) e^{t(s-1)}\rVert_{L^2[0, 1]} = \sqrt{\frac{2t}{1-e^{-2t}}}\lVert se^{-ts}\rVert_{L^2[0 ,1]} 
\]
\[
=\frac{1}{\sqrt{t}} \sqrt{\frac{2}{1-e^{-2t}}}\lVert tse^{-ts}\rVert_{L^2[0 ,1]} \leq \frac{1}{\sqrt{t}} \sqrt{\frac{2}{1-e^{-2t}}} \sup \{ se^{-s} \mid s\geq 0\}  
\]
\[ 
\leq  \frac{1}{\sqrt{t}} \sqrt{\frac{2}{1-e^{-2t}}}  \to 0 
\]
as $t\to \infty$. We now see that $\lim_{t\to \infty}\lVert[T_t, \phi]\rVert=0$ holds for any $\phi \in C_0(X)$. Hence, $T\in M(RL^*_c(H^{\mathrm{dR}}_X))$ is an odd, self-adjoint, $G$-continuous operator and we have for any $g\in G$, 
\[
1-T^2, \,\,\ g(T)-T \in RL^*_c(H^{\mathrm{dR}}_X).
\] 
\end{proof}
One can show (see \cite[Section 5]{BGHN}) that $[H^{\mathrm{dR}}_X, T_1]=[H^{\mathrm{JV}}_X, T_{\mathrm{JV}}]=1_G$ in $KK^G(\bC, \bC)$. In fact, $(H^{\mathrm{dR}}_X, T_t)$ for $t\in [0, 1]$ is a homotopy from $[H^{\mathrm{dR}}_X, T_1]$ to $[H^{\mathrm{JV}}_X, T_{\mathrm{JV}}]$.  By Theorem \ref{thm_XGgamma}, we see that the Baum--Connes assembly map $\mu^{B, G}_r$ is an isomorphism  for any $G$ which acts properly (and continuously) on a locally finite tree by automorphisms. See \cite{KasparovSkandalis91} and \cite{Tu1999} for the proof by the gamma element method.
\end{example}


\section{BCC for groups acting on finite-dimensional CAT(0)-cubical spaces}

In this final section, we extend the recently obtained new proof of the Baum--Connes conjecture with coefficients for groups $G$ which act properly (and continuously) and co-compactly on a finite-dimensional CAT(0)-cubical space with bounded geometry \cite{BGHN} to the not-necessarily co-compact setting. The groundwork for this is already done in \cite{BGHN} (almost all results are proven in the general, not-necessarily co-compact setting). 

We first recall some materials and established results from \cite{BGHN} (see also \cite{BGH}). For the rest of this section, $X$ will be a bounded geometry CAT(0)-cubical space, as in \cite[Section 2.2]{NibloReeves98}. In particular, $X$ is obtained by identifying cubes, each of which is isometric to the standard Euclidean cube $[0, 1]^q$ of the appropriate dimension, by isometries of their faces. We shall refer to $0$-cubes as vertices, $1$-cubes as edges, etc. The bounded geometry condition asserts that there is a uniform bound on the number of edges that contain a given vertex, and this implies that $X$ is finite dimensional. Further, we shall assume all group actions on $X$ are by automorphism of its structure as a CAT(0)-cubical space. In particular, a group acting on $X$ permutes the vertices, the edges, and the $q$-cubes for each fixed $q$.

The midplanes of a $q$-cube $C$ in $X$ correspond to the intersections of the standard cube with the coordinate hyperplanes $x_j=1/2$ ($j=1, \cdots, q$). We generate an equivalence relation on the collection of all midplanes of all cubes in $X$ by declaring as equivalent any two midplanes whose intersection is itself a midplane. The union of all midplanes in a given class is called a hyperplane, and each hyperplane $H$ satisfies the following important properties: (i) $H$ is connected, and (ii) the complement $X \backslash H$ has precisely two path components. We shall say that a hyperplane $H$ separates two subsets of $X$ if these subsets lie in distinct components of $X \backslash H$. A group acting on $X$ necessarily permutes the hyperplanes. A hyperplane $H$ is adjacent to a cube $C$ if it is disjoint from $C$, but intersects a cube that includes $C$ as a codimension-one face. 

On the cube $[0, 1]^q$, the smooth $p$-forms can be written as $\alpha=\sum_{I}f_Idx_I$ where $I=\{i_1 < \cdots < i_p \}$ is a multi-index, and $f_I$ is a smooth function on the cube $[0, 1]^q$. For each $q$-cube $C$, the space $\Omega^{p}(C)$ of $p$-forms carries a natural inner product and we let $\Omega_{L^2}^{p}(C)$ be the Hilbert space completion.

We fix a base vertex ($0$-cube) $P$ of $X$. Let $C$ be a cube in $X$ and $H$ be a hyperplane that intersects $C$. A coordinate function $x_H\colon D\to \R$ is defined as the affine function on $C$ which takes $1$ on the codimension-one face of $C$, which is disjoint from $H$ and is separated by $H$ from the base vertex, while it takes $0$ on the opposite codimension-one face disjoint from $H$ but not separated by $H$ from the base vertex. We define $\alpha_H=dx_H$. This is a constant-coefficient one-form on $C$.

A weight function $w$ for $X$ is a positive function on the set of hyperplanes. A weight function $w$ is called proper if the set $\{H \mid w(H)\leq M \}$ is finite for any $M>0$. If a locally compact group $G$ acts on $X$ by automorphisms, it is $G$-adapted if $\sup_{H}|w(H) - w(gH)| < \infty$ for any $g\in G$ and for some open subgroup $G$, $w(gH)=w(H)$ for all $H$ and for all $g$ in the subgroup. We take $w(H)=\mathrm{distace}(P, H)$ as our preferred weight function which is both proper and $G$-adapted but is is helpful for us not to fix a weight function in what follows. 

For each cube $C$, and for each $p\geq0$, we denote by $\Omega_0^p(C)\subset \Omega^p(C)$ the space of those smooth $p$-forms on $C$ that pull back to zero on each open face of $C$. We consider the de-Rham differential $d$ as an unbounded operator from $\Omega_{L^2}^p(C)$ to $\Omega_{L^2}^{p+1}(C)$ with domain $\Omega_{0}^p(C)$. We consider its formal adjoint $d^\diamond$ as an unbounded operator with domain $\Omega^{p+1}(C)$. We form the direct sum
\[
\Omega_{0}^{\ast}(C) = \bigoplus_p \Omega_{0}^{p}(C), \,\,\, \Omega_{L^2}^{\ast}(C) = \bigoplus_p \Omega_{L^2}^{p}(C),
\]
and then form the unbounded operator
\[
D = d+ d^{\diamond} \colon   \Omega_{L^2}^{\ast}(C) \to  \Omega_{L^2}^{\ast}(C).
\]
The operator $D$ on $\Omega_{L^2}^{\ast}(C)$ is essentially self-adjoint on the domain $\Omega_{0}^{\ast}(C)$, and the self-adjoint closure of $D$ has compact resolvent. The kernel of the closure is precisely the one-dimensional space of top-degree constant coefficient forms (see \cite[Lemma 4.2]{BGHN}).

Let $w$ be a weight function for $X$. For each cube $C$ in $X$, an affine function $w_C\colon C\to \R$ is defined by the formula
\[
w_C= \sum_{H\in \mathrm{Mid}(C)} w(H)x_H
\]
where $\mathrm{Mid}(C)$ is the set of hyperplanes that contain a midplane of $C$. We form the Witten-type perturbation
\begin{equation}\label{eq_dw}
d_w= e^{-w_c} d e^{w_c}\colon   \Omega_{0}^{\ast}(C) \to  \Omega^{\ast}(C).
\end{equation}
We have
\[
d_w\colon \beta \mapsto d\beta +  \sum_{H\in \mathrm{Mid}(C)} w(H)x_H\wedge \beta,
\]
so that $d_w$ is a bounded perturbation of $d$. Let $d_w^\diamond$ be its formal adjoint considered as an unbounded operator with domain $\Omega^{\ast}(C)$, then the unbounded operator
\[
D_w = d_w+ d_w^{\diamond} \colon   \Omega_{L^2}^{\ast}(C) \to  \Omega_{L^2}^{\ast}(C)
\]
with domain $\Omega_{0}^{\ast}(C)$ is essentially self-adjoint and has compact resolvent. 

 Let $C$ be a cube in $X$ and $H$ be a hyperplane that intersects $C$. We let
 \[
 y_H= x_H-1 \colon C\to \R.
 \]
Note that $dy_H=\alpha_H=dx_H$. We define a linear operator (see \cite[Definition 4.4]{BGHN})
\[
e_w\colon \Omega^p(X) \to  \Omega^{p+1}(X)
\]
by the formula
\[
e_w\colon \beta \mapsto \sum_{H\in \mathrm{SAH}(C)} w(H)e^{w(H)y_H}\alpha_H \wedge \beta
\]
for $\beta \in \Omega^\ast(C)$ where:
\begin{itemize}
\item $\mathrm{SAH}(C)$ is the set of all hyperplanes that are adjacent to $C$ and separate $C$ from the base vertex $P$.
\item  For a hyperplane $H\in \mathrm{SAH}(C)$, if $D_H$ is the cube that is intersected by $H$ and that includes $C$ as a codimension-one face, then we view $\alpha_H\wedge \beta$ as a differential form on $D_H$ by pulling back $\beta$ along the orthogonal projection from $D_H$ to $C$.
\end{itemize}

We denote by $D^{\mathrm{dR}}_w$ (see \cite[Definition 4.5]{BGHN}, it was denoted by $D_{\mathrm{dR}}$) the following symmetric operator on $\Omega^\ast_{L^2}(X)$ with domain $\Omega^\ast_{0}(X)$
\[
D^{\mathrm{dR}}_w = d_w + e_w + d_w^\diamond + e_w^\diamond.
\]
By $d_w$, we mean here the direct sum of all the Witten-type operators \eqref{eq_dw} over all cubes of $X$ and the same for the formal adjoint $d_w^\diamond$. The operator $D^{\mathrm{dR}}_w$ is essentially self-adjoint and it has compact resolvent \cite[Theorem 4.6]{BGHN} if the weight function $w$ is proper.

For any vertex $Q$ of $X$, a cube $C$ leads to $Q$ to the base vertex $P$ if (i) the cube $C$ includes the vertex $Q$ and (ii) all the hyperplanes that intersect $C$ also separate $Q$ from $P$. That is $\mathrm{Mid}(C) \subset \mathrm{SAH}(Q)$. We denote by $\Omega^\ast_{L^2}(X)_Q \subset \Omega^\ast_{L^2}(X)$ the direct sum of the spaces $\Omega^\ast_{L^2}(C)$ over those cubes $C$ that lead $Q$ to $P$. We define $\Omega^\ast_{0}(X)_Q$ and  $\Omega^\ast(X)_Q$ similarly (see \cite[Definition 4.7]{BGHN}).

We have the following Hilbert space direct sum decomposition:
\begin{equation}\label{eq_dR_decomposition}
  \Omega^\ast_{L^2}(X) = \bigoplus_Q \Omega^\ast_{L^2}(X)_Q.
\end{equation}
The operator $D^{\mathrm{dR}}_w$ is block-diagonal with respect to this decomposition \cite[Lemma 4.9]{BGHN}. We have the following useful estimate.

\begin{lemma}(See the proof of \cite[Theorem 4.6]{BGHN})\label{lem_estimate} For any $\beta \in \Omega^\ast_{0}(X)_Q$,
\[
\lVert D^{\mathrm{dR}}_w\beta \rVert^2 \geq \sum_{H \in \mathrm{SAH}(Q)} \min\{\frac12 w(H)(1-e^{-2w(H)}), \pi^2+ w(H)^2  \} \lVert\beta\rVert^2 
\]
\[
= \sum_{H \in \mathrm{SAH}(Q)} \frac12 w(H)(1-e^{-2w(H)}) \lVert\beta\rVert^2.
\]
\end{lemma}

\begin{lemma}(See \cite[Lemma 4.11]{BGHN}) If the weight function $w$ is $G$-adapted, then for any $g$ in $G$, the difference $g(D^{\mathrm{dR}}_w) - D^{\mathrm{dR}}_w$ is a bounded operator. Moreover $\lVert g(D^{\mathrm{dR}}_w) - D^{\mathrm{dR}}_w\rVert$ is bounded by a constant times
\[
\sup \{ gw(H) - w(H)  \mid   \text{ $H$ is a hyperplane of $X$} \} + \max \{  w(H) \mid \text{$H$ separates $P$ and $gP$}  \}
\]
The constant depends only on the dimension of $X$.
\end{lemma}

If $G$ acts on $X$ by automorphisms, for any proper, $G$-adapted weight function $w$, the pair $(\Omega^\ast_{L^2}(X), D^{\mathrm{dR}}_w)$ is an unbounded Kasparov cycle for $KK^G(\bC, \bC)$ and we denote by $[D^{\mathrm{dR}}_w]$ its associated class in $KK^G(\bC, \bC)$ (see \cite[Theorem 4.12, Definition 4.13]{BGHN}).

\begin{theorem}(See \cite[Theorem 5.1]{BGHN} and \cite[Theorem 9.14]{BGH}) \label{thm_dR1G} If $G$ acts on $X$ by automorphisms, the element $[D^{\mathrm{dR}}_w]\in KK^G(\bC, \bC)$ satisfies  
\[
[D^{\mathrm{dR}}_w] = 1_G \,\,\, \text{in $KK^G(\bC, \bC)$}.
\]
\end{theorem}  

Now, suppose the $G$-action on $X$ is proper.  Let $H_X=\Omega^\ast_{L^2}(X)$ which is naturally a graded $X$-$G$-module. We shall construct an $X$-$G$-localized Kasparov cycle $(H_X, T)$ for $KK^G(\bC, \bC)$ such that $[H_X, T_1]=[D^{\mathrm{dR}}_w] =1_G$ in $KK^G(\bC, \bC)$.

In what follows, for simplicity, we assume $w\geq1/2$. For example, we can use $w(H)=\mathrm{distace}(P, H)$.

For $t\geq1$, we define an odd, symmetric unbounded operator 
\begin{equation}\label{eq_formulaDt}
\cD_t=  t^{-1}(d_{tw} + d_{tw}^\diamond) +  (\sqrt{t})^{-1}(e_{tw} + e_{tw}^\diamond) \colon \Omega^\ast_{L^2}(X) \to \Omega^\ast_{L^2}(X)
\end{equation}
with domain $\Omega^\ast_{0}(X)$. Note that $\cD_1=D^{\mathrm{dR}}_{w}$ but $\cD_t$ is slightly different from $t^{-1}D^{\mathrm{dR}}_{tw}$. The operators $\cD_t$ are block-diagonal with respect to the decomposition \eqref{eq_dR_decomposition} and it is essentially self-adjoint with compact resolvent on each block $\Omega^\ast_{L^2}(X)_Q$. From these and from the formula \eqref{eq_formulaDt} of $\cD_t$, we see that the family $\{\cD_t\}_{t\geq1}$ defines an odd, regular, symmetric unbounded operator $\cD$ on $H_X\otimes C_0[1, \infty)$ with domain $C_c([1,\infty), \Omega^\ast_{0}(X))$ and that its closure is self-adjoint. 

We have the following estimate.

\begin{lemma}\label{lem_testimate} For $t\geq1$ and for any $\beta \in \Omega^\ast_{0}(X)_Q$,
\[
\lVert \cD_t\beta \rVert^2 \geq \sum_{H \in \mathrm{SAH}(Q)} \min\{\frac12 w(H)(1-e^{-2tw(H)}), t^{-2}\pi^2+ w(H)^2  \} \lVert\beta\rVert^2.
\]
\end{lemma}
\begin{proof} This can be proven as in the proof of \cite[Theorem 4.6]{BGHN}.
\end{proof}

\begin{lemma}\label{lem_Xlem} If the weight function $w$ is proper, The odd, regular, self-adjoint unbounded operator $\cD$ on $H_X\otimes C_0[1, \infty)$ satisfies 
\[
(\cD \pm i)^{-1} \in C_b([1, \infty), \Compacts(H_X)).
\]
Moreover, for any $\epsilon>0$, there is $\phi\in C_0(X)$, such that $\lVert(1-\phi)(\cD \pm i)^{-1} \rVert< \epsilon$.
\end{lemma}
\begin{proof} The unbounded operator $\cD$ is block diagonal with respect to \eqref{eq_dR_decomposition} and it is not hard to see that its restriction $\cD_Q$ on each block $\Omega^\ast_{L^2}(X)_Q$ satisfies
\[
(\cD_Q \pm i)^{-1} \in C_b([1, \infty), \Compacts(\Omega^\ast_{L^2}(X)_Q)).
\]
Lemma \ref{lem_testimate} implies the sum
\[
(\cD \pm i)^{-1}  = \sum_{Q} (\cD_Q \pm i)^{-1}
\]
is absolutely norm-convergent in $C_b([1, \infty), \Compacts(H_X))$. The second claim also follows from Lemma \ref{lem_testimate}.
\end{proof}

Note that the $G$-action preserves the domain $C_c([1,\infty), \Omega^\ast_{0}(X))$ of $\cD$.

\begin{lemma}\label{lem_Glem} If the weight function $w$ is $G$-adapted, then for any $g$ in $G$, the difference $g(\cD) - \cD$ is a bounded operator. Moreover $\lVert g(\cD) - \cD\rVert$ is bounded by a constant times
\[
\sup \{ gw(H) - w(H)  \mid   \text{ $H$ is a hyperplane of $X$} \} + \max \{  w(H) \mid \text{$H$ separates $P$ and $gP$}  \}
\]
The constant depends only on the dimension of $X$.
\end{lemma}
\begin{proof} This can be proven in the same way as the proof of \cite[Lemma 4.11]{BGHN} with bit more refined estimates. We write $\cD_t=\cD_{t, w, P}$ to make it clear its dependence on the weight function $w$ and on the base vertex $P$ so that $g(\cD_t)=\cD_{t, gw, gP}$. We bound the differences
\[
\cD_{t, gw, gP} - \cD_{t, w, gP}\,\,\,  \text{and} \,\,\, \cD_{t, w, gP} - \cD_{t, w, P}
\]
separately, uniformly in $t\geq1$ to obtain the estimate in the statement. We have
\[
\cD_{t, gw, gP} - \cD_{t, w, gP}=  t^{-1}(d_{tgw, gP} - d_{tw, gP}) +  t^{-1}(d_{tgw, gP}^\diamond - d_{tw, gP}^\diamond) 
\]
\[
+  (\sqrt{t})^{-1}(e_{tgw, gP} -  e_{tw, gP}) + (\sqrt{t})^{-1}(e_{tgw, gP}^\diamond - e_{tw, gP}^\diamond).
\]
Since $w$ is $G$-adapted, there is $M\geq0$ such that $|gw(H)-w(H)| \leq M$ for all $H$. Using the formula
\[
t^{-1}d_{tw, gP}\colon \beta \mapsto t^{-1}d\beta +  \sum_{H\in \mathrm{Mid}(C)} w(H)x_{H, gP}\wedge \beta,
\]
we see that
\[
\lVert t^{-1}(d_{tgw, gP} - d_{tw, gP})\rVert  \leq \dim(X)\cdot M
\]
for all $t\geq1$ and thus $\lVert t^{-1}(d_{tgw, gP}^\diamond - d_{tw, gP}^\diamond)\rVert  \leq \dim(X)\cdot M$ for all $t\geq1$. 

For $w\geq1/2$ and for $M\geq m \geq0$, we have
\begin{equation}\label{eq_elementary}
\sqrt{t}\lVert we^{-twx}- (w+m)e^{-t(w+m)x} \rVert_{L^2[0 ,1]} \leq 2M
\end{equation}
for all $t\geq1$ since for instance we can write 
\[
|we^{-twx}- (w+m)e^{-t(w+m)x}| \leq  w(tmx)e^{-twx}\frac{(1-e^{-tmx})}{tmx}  + me^{-t(w+m)x}
\]
\[
\leq m(twx)e^{-twx} + me^{-t(w+m)x}
\]
and since (we use $w\geq1/2$)
\[
\lVert twxe^{-twx}\rVert^2_{L^2[0,1]} \leq  (tw)^{-1} \int_0^\infty y^2e^{-2y}dy \leq \frac{1}{4wt} \leq \frac{1}{2t},
\]
and
\[
\lVert e^{-t(w+m)x}\rVert^2_{L^2[0, 1]} \leq \frac{1}{t(w+m)}\int_0^\infty e^{-2y} dy \leq \frac{1}{2tw} \leq \frac{1}{t}
\]
so that $\sqrt{t}\lVert we^{-twx}- (w+m)e^{-t(w+m)x} \rVert_{L^2[0 ,1]}$ is bounded by
\[
\sqrt{t}m (\lVert twxe^{-twx}\rVert_{L^2[0,1]} + \lVert e^{-t(w+m)x}\rVert_{L^2[0, 1]} ) \leq \sqrt{t}m (\frac{1}{\sqrt{2t}} + \frac{1}{\sqrt t} ) \leq 2M.
\]
Using \eqref{eq_elementary} and the formula 
\[
(\sqrt{t})^{-1}e_{tw, gP}\colon \beta \mapsto \sum_{H\in \mathrm{SAH}(C)} \sqrt{t}w(H)e^{tw(H)y_{H, gP}}\alpha_{H, gP} \wedge \beta
\]
we obtain 
\[
\lVert(\sqrt{t})^{-1}(e_{tgw, gP} -  e_{tw, gP})\rVert \leq 2\dim(X) \cdot M
\]
for all $t\geq1$ and thus $\lVert(\sqrt{t})^{-1}(e_{tgw, gP}^\diamond -  e_{tw, gP}^\diamond)\rVert \leq 2\dim(X) \cdot M$ for all $t\geq1$. Putting everything together, we obtain
\[
\lVert\cD_{t, gw, gP} - \cD_{t, w, gP}\rVert \leq 6\dim(X)\cdot M.
\]
Next, we consider the difference $\cD_{t, w, gP} - \cD_{t, w, P}$. As in the proof \cite[Lemma 4.11]{BGHN}, 
\[
(t^{-1}d_{tw, gP} -  t^{-1}d_{tw, P})\beta = \sum_{H \in \mathrm{Mid}} w(H)(\alpha_{H, gP} - \alpha_{H, P}) \wedge \beta
\]
where the difference $\alpha_{H, gP} - \alpha_{H, P}$ is zero unless the hyperplane $H$ separates $P$ and $gP$ in which case the difference is $2\alpha_{H, gP}$. Thus, we have
\[
\lVert t^{-1}d_{tw, gP} -  t^{-1}d_{tw, P}\rVert \leq 2\dim(X) \cdot \max\{w(H) \mid \text{$H$ separates $P$ and $gP$}\}.
\]
The same is true for $\lVert t^{-1}d_{tw, gP}^\diamond -  t^{-1}d_{tw, P}^\diamond\rVert$. Lastly, we have
\[
(\sqrt{t})^{-1}(e_{tw, gP} -  e_{tw, P})\beta = \sum_{H} \sqrt{t}w(H)(e^{tw(H)y_{H, gP}}\alpha_{H, gP} - e^{tw(H)y_{H, P}} \alpha_{H, P}) \wedge \beta
\]
where again the summands are zero except when $H$ separates $P$ and $gP$ and we can bound the sum by 
\[
2\dim (X) \cdot \max\{ \sqrt{t}w(H)\lVert e^{-tw(H)x}\rVert_{L^2[0 ,1]}  \mid  \text{$H$ separates $P$ and $gP$} \}
\]
\[
\leq 2\dim (X) \cdot  \max\{ w(H) \mid  \text{$H$ separates $P$ and $gP$} \}.
\]
The same is true for $\lVert(\sqrt{t})^{-1}(e_{tw, gP}^\diamond -  e_{t, w, P}^\diamond)\rVert$. Putting everything together, we obtain
\[
\lVert\cD_{t, w, gP} - \cD_{t, w, P}\rVert \leq 8\dim(X)\cdot \max\{ w(H) \mid   \text{$H$ separates $P$ and $gP$}   \}
\]
so we are done.
\end{proof}

\begin{theorem}\label{thm_cubeXG} Let $w$ be a weight function for $X$ which is proper and $G$-adapted. Let $H_X=\Omega^\ast_{L^2}(X)$ and $T=\cD(1+\cD^2)^{-1}$. Then, the pair $(H_X, T)$ is an $X$-$G$-localized Kasparov cycle for $KK^G(\bC, \bC)$ with $[H_X, T_1]=[D^{\mathrm{dR}}_w] =1_G$ in $KK^G(\bC, \bC)$.
\end{theorem}
\begin{proof} We check that $\cD$ satisfies the conditions (1) -(3) of Proposition \ref{prop_unboundedXG}. (1): Let $B\subset C_0(X)$ be the subalgebra of compactly supported functions that are smooth on each cube. For any $\phi \in B$, we compute
\begin{align*}
[\cD_t, \phi] & = [t^{-1}(d_{tw} + d_{tw}^\diamond) +  (\sqrt{t})^{-1}(e_{tw} + e_{tw}^\diamond), \phi] \\
 &= t^{-1}[d_{tw} + d_{tw}^\diamond, \phi] + (\sqrt{t})^{-1}[e_{tw} + e_{tw}^\diamond, \phi] \\
& = t^{-1}c(\phi) +  (\sqrt{t})^{-1}[e_{tw} + e_{tw}^\diamond, \phi] 
\end{align*}
where $c(\phi)$ denotes the Clifford multiplication by the gradient of $\phi$ in each cube which is a bounded operator. Let $C>0$ so that 
\[
|\phi(x)-\phi(y)|\leq C\cdot \mathrm{distance}(x, y)\,\,\,  \text{for any $x, y \in X$}.
\]
We have (see the proof of \cite[Theorem 5.2]{BGHN}) for all $\beta \in \Omega^\ast(C)$
\begin{align*}
\lVert[e_{tw}, \phi]\beta \rVert^2 & \leq \sum_{H \in \mathrm{SAH}(C)}C^2 \lVert xtw(H)e^{-tw(H)x}\rVert^2_{L^2[0, 1]}\cdot \lVert\beta\rVert^2 \\
 & \leq  \dim(X) \cdot C^2 \cdot \max\{ (xe^{-x})^2 \mid x\geq0 \} \cdot \lVert\beta\rVert^2 \\
 & \leq  \dim(X) \cdot C^2 \cdot \lVert\beta\rVert^2
\end{align*}
for all $t\geq1$. From these, we see that $[\cD, \phi]$ extends to a bounded operator $S \in \Linears(H_X\otimes C_0[1, \infty))$ with $\lVert S_t\rVert\to 0$ as $t\to \infty$.  (2): This follows from Lemma \ref{lem_Glem}. (3): This follows from Lemma \ref{lem_Xlem}. Thus, by Proposition \ref{prop_unboundedXG}, $(H_X, T)$ is an $X$-$G$-localized Kasparov cycle for$KK^G(\bC, \bC)$. That $[H_X, T_1]=[D^{\mathrm{dR}}_w]$ follows from the definition of $\cD$. We have $[D^{\mathrm{dR}}_w]=1_G$ by Theorem \ref{thm_dR1G}.
\end{proof}

\begin{theorem}\label{thm_cube} Let $G$ be a second countable, locally compact group $G$ which acts properly and continuously on a finite-dimensional CAT(0)-cubical space with bounded geometry by automorphisms. Then, the Baum--Connes assembly map $\mu^{B, G}_r$ is an isomorphism for any separable $G$-$C^*$-algebra $B$, i.e. BCC holds for $G$.
\end{theorem}
\begin{proof} Combine Theorem \ref{thm_XGgamma} and Theorem \ref{thm_cubeXG}.
\end{proof}

\begin{remark} Any group $G$ which acts properly on a CAT(0)-cubical space has the Haagerup Approximation property \cite[Section 1.2.7]{CCJJV}, so $G$ is a-T-menable. Thus, BCC for these groups are already known by the Higson--Kasparov Theorem \cite{HK97}, \cite{HigsonKasparov}.
\end{remark}

\bibliography{Refs}
\bibliographystyle{plain}

\end{document}